\documentclass{article}
\usepackage[height=9in, width=6.5in]{geometry}
\usepackage{graphicx}
\usepackage{tikz}
\usepackage{amsfonts}
\usepackage{amsthm}
\usepackage{amsmath}
\usepackage{xfrac}
\usepackage{subcaption}
\usepackage{hyperref}

\DeclareMathAlphabet{\mathpzc}{OT1}{pzc}{m}{it}

\tikzset{
  frac arrow/.style={postaction={decorate,decoration={
        markings,
        mark=at position #1 with {\arrow[]{Stealth[round]}}
      }}},
}

\usetikzlibrary{arrows.meta, decorations.markings, decorations.pathmorphing}
\newtheorem{theorem}{Theorem}[section]
\newtheorem{corollary}{Corollary}[theorem]
\newtheorem{lemma}[theorem]{Lemma}
\newtheorem{definition}{Definition}[section]
\newtheorem{conjecture}{Conjecture}[section]

\numberwithin{equation}{section}

\newcommand{\N}{\mathbb{N}}
\newcommand{\Z}{\mathbb{Z}}
\newcommand{\R}{\mathbb{R}}
\newcommand{\C}{\mathbb{C}}

\newcommand{\dd}{\mathrm{d}}
\DeclareMathOperator{\sgn}{sgn}
\DeclareMathOperator{\csch}{csch}
\DeclareMathOperator{\sech}{sech}
\DeclareMathOperator*{\res}{res}
\newcommand{\bigO}[1]{O \left( #1 \right)}
\newcommand{\lilO}[1]{o \left( #1 \right)}
\newcommand{\norm}[1]{\left| \left| #1 \right| \right|}
\newcommand{\re}[1]{\mathrm{Re} \left[ #1 \right]}
\newcommand{\im}[1]{\mathrm{Im} \left[ #1 \right]}
\newcommand{\wlim}[1]{\mathrm{w-}\lim_{#1}}
\newcommand{\id}[1]{I_{#1}}
\newcommand{\minor}[2]{\mathrm{Minor}_{#1}\left( #2 \right)}
\newcommand{\epsN}{\epsilon_{\scriptscriptstyle N}}
\newcommand{\ii}{\mathrm{i}}
\newcommand{\ee}{\mathrm{e}}
\newcommand{\quadratrix}[1]{\mathcal{Q}_{#1}}
\newcommand{\quadplus}[1]{\mathcal{Q}_{#1 \, +}}
\newcommand{\quadminus}[1]{\mathcal{Q}_{#1 \, -}}
\newcommand{\strip}[1]{\mathcal{S}_{#1}}
\newcommand{\maxrm}{\mathrm{max}}
\newcommand{\minrm}{\mathrm{min}}

\newcommand{\Brm}{\mathrm{B}}
\newcommand{\Ai}{\mathrm{Ai}}
\newcommand{\wyl}{\mathrm{Wyl}}
\newcommand{\dbar}{\overline{\partial}}
\newcommand{\sse}{\mathrm{SSE}}

\newcommand{\Quad}[1][1]{\hspace*{#1em}\ignorespaces}

\definecolor{navy}{RGB}{51, 101, 138}
\definecolor{maize}{RGB}{246, 174, 45}
\definecolor{seafoam}{RGB}{26, 147, 111}
\definecolor{fire}{RGB}{236, 78, 32}

\title{The Small-Dispersion Limit of the Intermediate Long Wave Equation via Semiclassical Soliton Ensembles}
\author{Matthew Dominique Mitchell\thanks{current affiliation: School of Data, Mathematics and Statistical Sciences, University of Central Florida, Orlando, Florida 32816; email: matt.mit@ucf.edu}}
\date{\textit{Department of Physics, University of Michigan, Ann Arbor, Michigan 48109}}

\begin{document}

\maketitle

\begin{abstract}
    We study the small-dispersion limit of the intermediate long wave (ILW) equation, specifically on a class of well-behaved initial conditions $u_0$ where the number of solitons in the solution increases without bound.
    First, we conduct a formal WKB-style analysis on the ILW direct scattering problem, generating approximate eigenvalues and norming constants.
    We then use this to define a modified set of scattering data and rigorously analyze the associated inverse scattering problem.
    The main results include demonstrating $L^2$-convergence of the solution at $t = 0$ to the original initial condition $u_0$ and for $0 < t < t_\mathrm{c}$ to the associated solution of invicid Burgers' equation, where $t_\mathrm{c}$ is the time of gradient catastrophe.
\end{abstract}



\vspace*{0.1cm}

\noindent\textbf{Acknowledgments:} 
The author would like to thank Peter Miller for his unwavering support, wealth of knowledge, and endless discussions over the years; Peter Perry for sharing his work; and Kurt Schmidt for his patient listening and helpful comments.

\section{Introduction}
\label{sec:Intro}

The intermediate long wave (ILW) equation is a nonlinear PDE serving as a model for the long wavelength, weakly nonlinear perturbations of a thin pycnocline boundary layer between two fluids, one of vanishing depth and the other of fixed depth characterized by the parameter $0 <\delta$, \cite{Joseph_1977, KubotaKoDobbs_1978, ChoiCamassa_1999, Saut_2019}
\begin{equation}
    0 = u_t + 2 u u_x + \epsilon \mathpzc{T}_{\delta \epsilon}[u_{xx}] \, .
    \label{eq:ILW}
\end{equation}
Here, $u:\R_x \times \R_t \to \R$ represents the displacement of the pycnocline from equilibrium, $x$ the spatial coordinate and $t$ time.
The specifics of the fluid system are shown in Figure \ref{fig:ILWSetUp}.
$\mathpzc{T}_{\delta \epsilon}$ in \eqref{eq:ILW} is the singular convolution operator
\begin{equation}
    \mathpzc{T}_{\delta \epsilon}[f](x) := \frac{1}{2 \delta \epsilon} \, \mathrm{P.V.} \int_\R \left[ \coth \left( \frac{\pi}{2 \delta \epsilon}(y - x) \right) - \sgn \left( y - x \right) \right] f(y) \ \dd y \, ,
    \label{eq:ILWTOp}
\end{equation}
where the principal value in \eqref{eq:ILWTOp} is with regard to the simple pole present in the kernel at $y = x$.
Since $\mathpzc{T}_{\delta \epsilon}$ is a convolution operator, for appropriate classes of functions we can also define it by its action in the Fourier domain
\begin{equation}
    \widehat{\mathpzc{T}_{\delta \epsilon}[f]}(k) = \ii \, \tau(\delta \epsilon k) \hat{f}(k) 
    \label{eq:FourierTransKernel}
\end{equation}
where the $\tau: \R \to \R$ is the Fourier multiplier
\begin{equation}
    \tau(k) := \left\{ \begin{array}{cl}
        \displaystyle \coth \left( k \right) - \frac{1}{k} & k \neq 0 \\
        0 & k = 0
    \end{array} \right. \, ,
    \label{eq:TConvolKernFT}
\end{equation}
i.e. the removable singularity in $\tau$ is plugged.

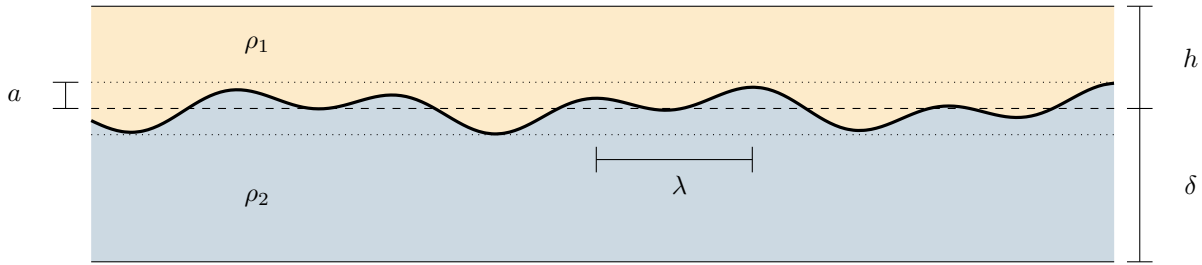
\begin{figure}
    \centering
    \begin{tikzpicture}[scale=1.7]
        \fill[navy!25!white] plot[smooth, samples=100, domain=-4:4] ({\x}, {1.2+0.1*sin(2*deg(\x))+0.1*sin(4.5*deg(\x+0.5))}) -- (4,0) -- (-4, 0) -- cycle;
        \fill[maize!25!white] plot[smooth, samples=100, domain=-4:4] ({\x}, {1.2+0.1*sin(2*deg(\x))+0.1*sin(4.5*deg(\x+0.5))}) -- (4,2) -- (-4, 2) -- cycle;
        
        \draw (-4, 2) -- (4, 2);
        \draw (-4, 0) -- (4, 0);
        \draw[dashed] (-4, 1.2) -- (4, 1.2);
        \draw[very thick, domain=-4:4, smooth, samples=100, variable=\x] plot ({\x}, {1.2+0.1*sin(2*deg(\x))+0.1*sin(4.5*deg(\x+0.5))});
        \draw (4.2, 0) -- (4.2, 2);
        \draw (4.1, 0) -- (4.3, 0);
        \draw (4.1, 1.2) -- (4.3, 1.2);
        \draw (4.1, 2) -- (4.3, 2);
        \draw (-4.2, 1.405) -- (-4.2, 1.2);
        \draw (-4.1, 1.2) -- (-4.3, 1.2);
        \draw (-4.1, 1.405) -- (-4.3, 1.405);
        \draw[dotted] (-4, 0.995) -- (4, 0.995);
        \draw[dotted] (-4, 1.405) -- (4, 1.405);
        \draw (-0.05, 0.8) -- (1.17, 0.8);
        \draw (-0.05, 0.7) -- (-0.05, 0.9);
        \draw (1.17, 0.7) -- (1.17, 0.9);
        
        \node at (-2.7, 0.5) {$\rho_2$};
        \node at (-2.7, 1.7) {$\rho_1$};
        \node at (4.6, 0.6) {$\delta$};
        \node at (4.6, 1.6) {$h$};
        \node at (-4.6, 1.3) {$a$};
        \node at (0.6, 0.6) {$\lambda$};
    \end{tikzpicture}
    \caption{Schematic for the two dimensional incompressible and irrotational fluid system from which the ILW equation is derived.
    The equilibrium position (dashed black) of the thin pycnocline (solid black) between the two fluid densities (yellow and blue) are depicted above.
    It is required that the density of the top fluid $\rho_1$ be less than that of the lower fluid $\rho_2$ so the fluids are stably separated by gravity.
    The system is confined between two rigid and flat surfaces, a ceiling and floor.
    Labeled in the schematic, $h$ and $\delta$ are the equilibrium depths of the fluids, $a$ is the maximum amplitude of the disturbance of the pycnocline and $\lambda$ is a characteristic wavelength of the disturbance.
    The nonlinearity governing the dynamics of the system is characterized by the parameter $\alpha = a / h$ as well as that for the linear dispersion by $\epsilon = h / \lambda$.
    The ILW equation obtains in the asymptotic limit where $h \to 0^+$, $\delta$ and $\epsilon / \alpha$ fixed with $\epsilon$ and $\alpha$ individually bounded.
    Here, the top layer is depicted with vanishing depth $h$ while the bottom layer has fixed depth $\delta$.
    However, the model is independent of which fluid layer has the fixed depth $\delta$ and which has the vanishing depth $h$ \cite{KubotaKoDobbs_1978}.}
    \label{fig:ILWSetUp}
\end{figure}

Additionally, in \eqref{eq:ILW}, we have introduced a second auxiliary parameter $0 < \epsilon$ via a scaling symmetry of the equation.
Taking $\epsilon \to 0^+$ is known as the small-dispersion limit, so called because the coefficient of the linear dispersive term in \eqref{eq:ILW} vanishes.
Setting $\epsilon = 0$, we see the ``zero dispersion'' ILW equation is invicid Burgers' equation \cite{Burgers_1948}
\begin{equation} \label{eq:IB}
    0 = u^\Brm_t + 2 u^\Brm u^\Brm_x \, , 
\end{equation}
whose Cauchy problem for $\mathcal{C}^1(\R)$ initial condition $u_0$ is solvable by the method of characteristics:
\begin{equation}\label{eq:IBMethOfCharSolution}
    u^\Brm(x, t) = u_0(y) \quad \text{ where } \quad x = y + 2u_0(y) t \, .
\end{equation}
It is well-known that for initial conditions with negative slope, the solution to \eqref{eq:IB} via \eqref{eq:IBMethOfCharSolution} exists globally up to a critical time
\begin{equation}
    t_\mathrm{c} := \frac{1}{\displaystyle \max_{x \in \R} - 2 u_0'(x)} \, , \label{eq:CatTime}
\end{equation}
at which a gradient catastrophe occurs and after which the method of characteristics \eqref{eq:IBMethOfCharSolution} yields a solution which is multi-valued in some regions of spacetime.

Fixing initial condition $u_0$ for \eqref{eq:ILW} and letting $\epsilon > 0$ be small but non-zero, one may expect that the solution $u(x, t) \approx u^\Brm(x, t)$ for $0 \leq t < t_\mathrm{c}$.
Indeed, simulations of the ILW equation (see Figure \ref{fig:SmallDispDSW}) indicate exactly as much.
However, near the point of gradient catastrophe, the large derivative of $u \approx u^\Brm$ in \eqref{eq:ILW} results in the linear dispersive term becoming comparable to that of the nonlinear term.
The linear term then regulates the gradient catastrophe and the result is the generation of a dispersive shock wave (DSW) (see Figures \ref{fig:SmallDispEps05T06}, and \ref{fig:SmallDispEps05T15}), an expanding spatial region of fast oscillations.
Rigorously studying the small-dispersion limit means showing, as $\epsilon \to 0^+$, convergence of the solution for small time to that of the zero dispersion equation and describing the asymptotics of the DSW generated after the critical time.
For generic nonlinear dispersive equations, Whitham's modulation theory \cite{Whitham_1965} can be used to generate formal results.
Integrable PDE make for a great case study of DSW for which rigorous descriptions can be recovered.
Previous examples in this area include work on the Korteweg-de Vries (KdV) \cite{LaxLevermoreI_1983, LaxLevermoreII_1983, DeiftVenakidesZhou_1998, ClaeysGrava_2010_SolEdge, ClaeysGrava_2010_DispEdge}, Nonlinear Schr\"odinger \cite{KamvissisMcLaughlinMiller_2003_semiclass, JenkinsMcLaughlin_2014, BuckinghamJenkinsMiller_2025} and Benjamin-Ono \cite{XuMiller_2010, MillerWetzel_2016_SmallDispersion, BlackstoneGassotGerardMiller_2024_SmallDisp, BlackstoneMillerMitchell_2024} equations.
Like these equations, the ILW equation is integrable by inverse scattering transform (IST) \cite{KodamaAblowitzSatsuma_1982}.

\begin{figure}
    \centering
    \begin{subfigure}{0.49\textwidth}
        \centering
        \includegraphics[width=\textwidth]{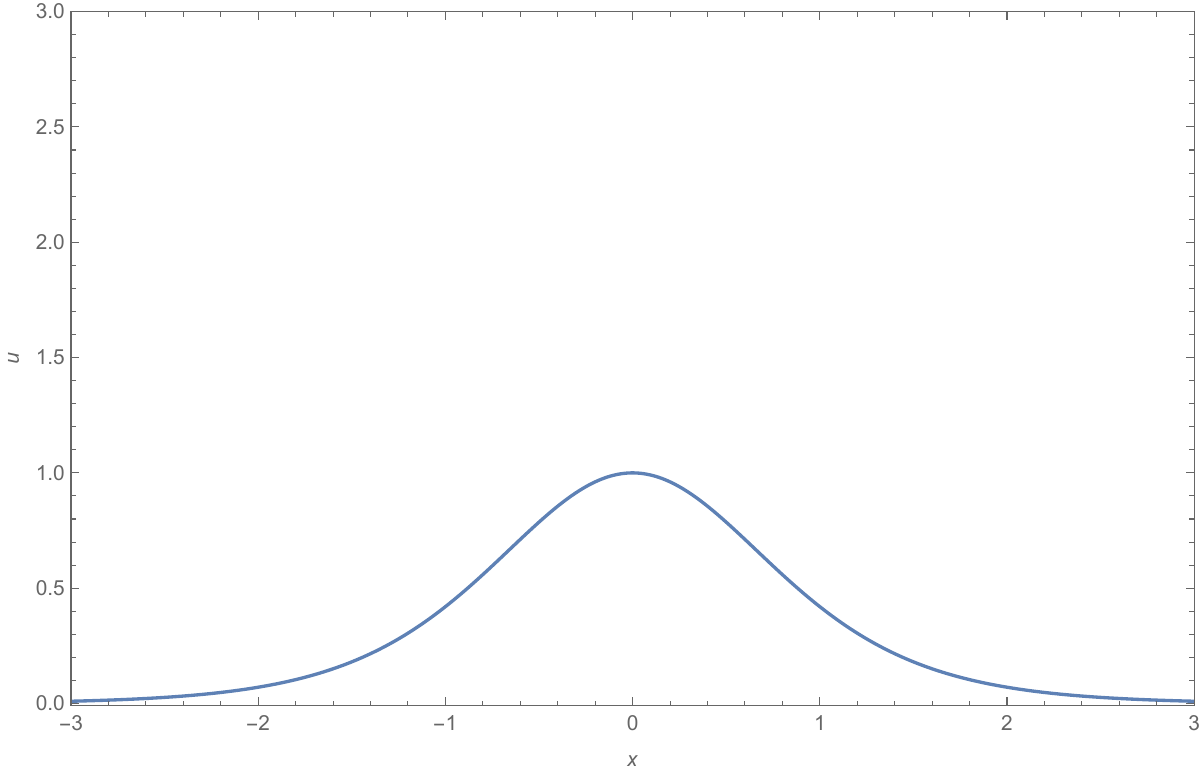}
        \subcaption{$\begin{array}{l} \epsilon = 0.05 \\ t = 0.000 \end{array}$}
        \label{fig:SmallDispEps05T00}
    \end{subfigure}
    \hfill
    \begin{subfigure}{0.49\textwidth}
        \centering
        \includegraphics[width=\textwidth]{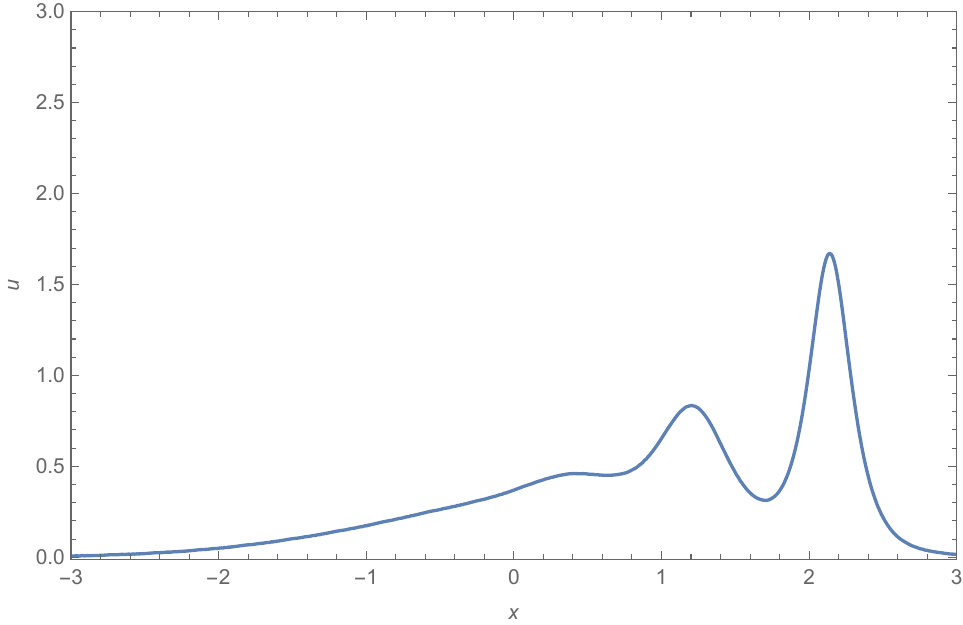}
        \subcaption{$\begin{array}{l} \epsilon = 0.20 \\ t = 1.500 \end{array}$}
        \label{fig:SmallDispEps20T15}
    \end{subfigure}
    
    \vspace*{0.1cm}
    
    \begin{subfigure}{0.49\textwidth}
        \centering
        \includegraphics[width=\textwidth]{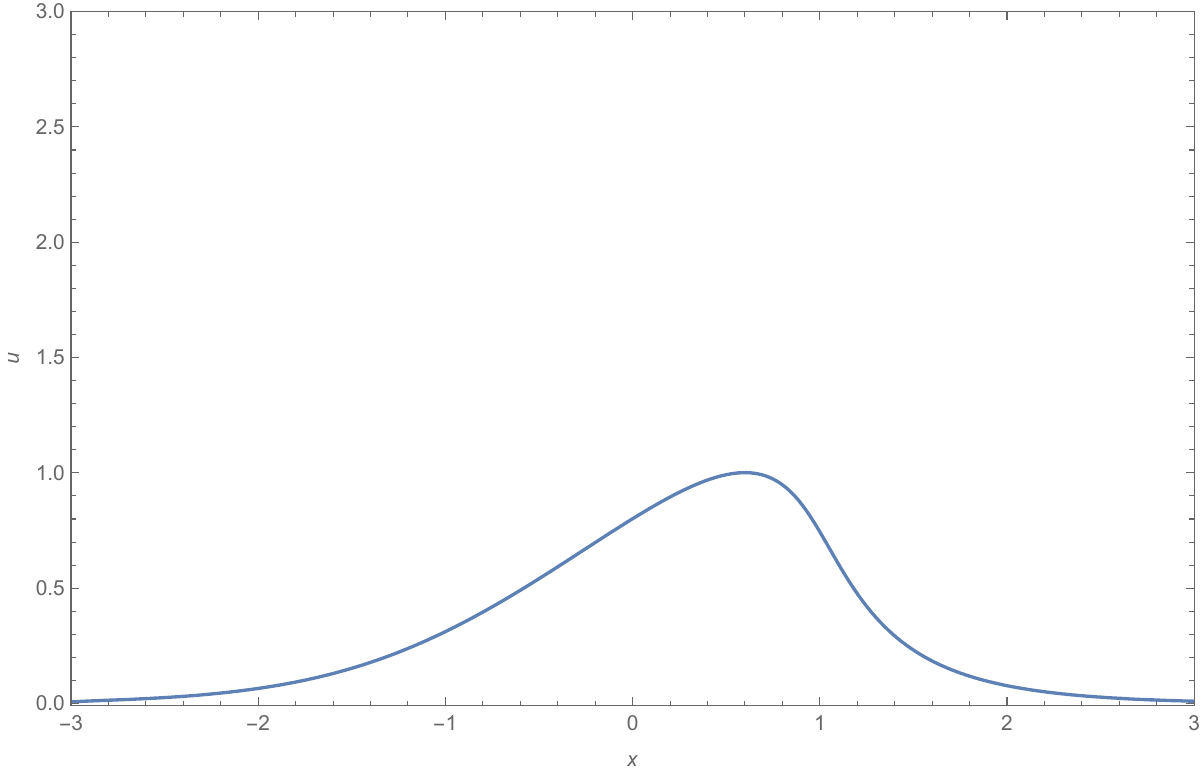}
        \subcaption{$\begin{array}{l} \epsilon = 0.05 \\ t = 0.300 \end{array}$}
        \label{fig:SmallDispEps05T03}
    \end{subfigure}
    \hfill
    \begin{subfigure}{0.49\textwidth}
        \centering
        \includegraphics[width=\textwidth]{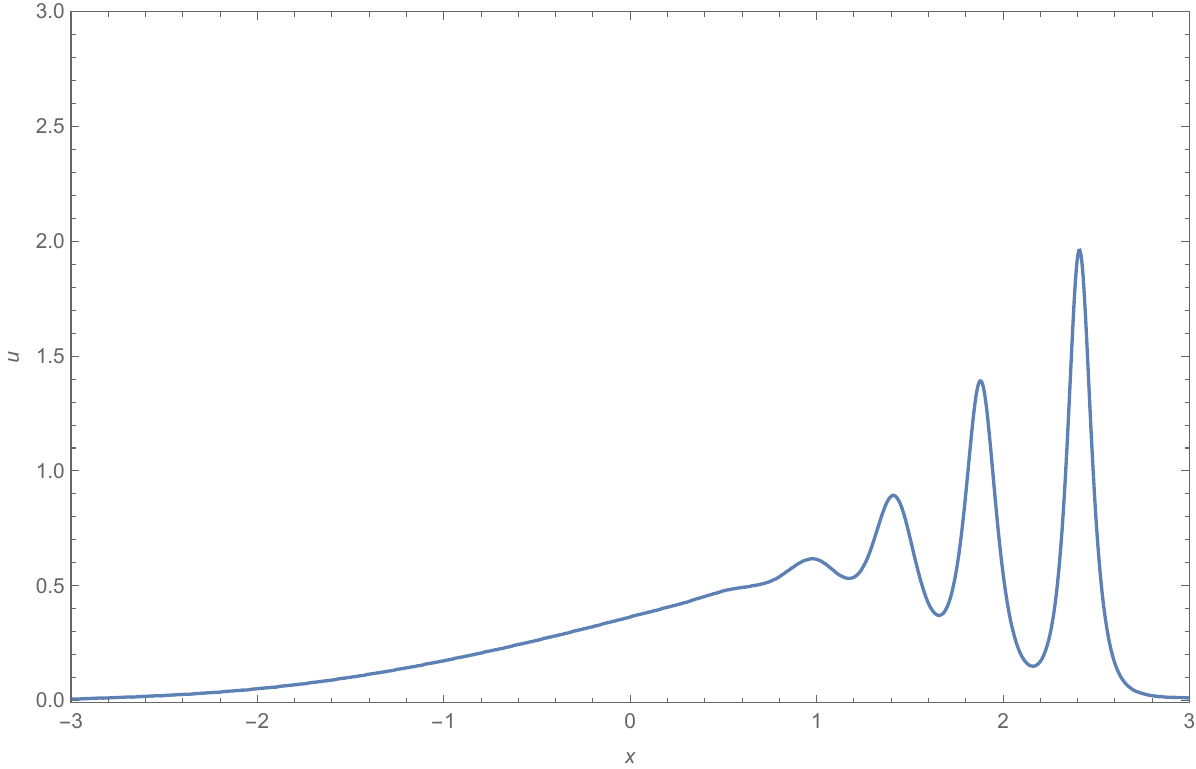}
        \subcaption{$\begin{array}{l} \epsilon = 0.10 \\ t = 1.500 \end{array}$}
        \label{fig:SmallDispEps10T15}
    \end{subfigure}
    
    \vspace*{0.1cm}
    
    \begin{subfigure}{0.49\textwidth}
        \centering
        \includegraphics[width=\textwidth]{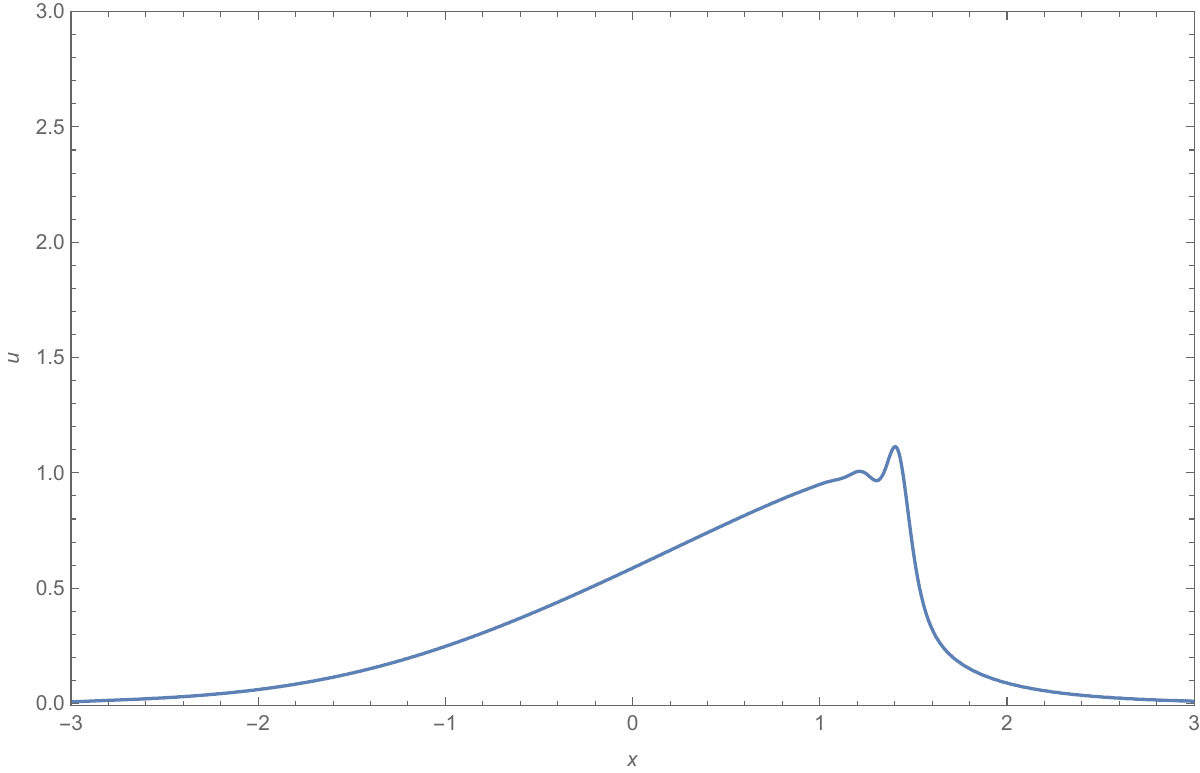}
        \subcaption{$\begin{array}{l} \epsilon = 0.05 \\ t = 0.650 \end{array}$}
        \label{fig:SmallDispEps05T06}
    \end{subfigure}    
    \hfill
    \begin{subfigure}{0.49\textwidth}
        \centering
        \includegraphics[width=\textwidth]{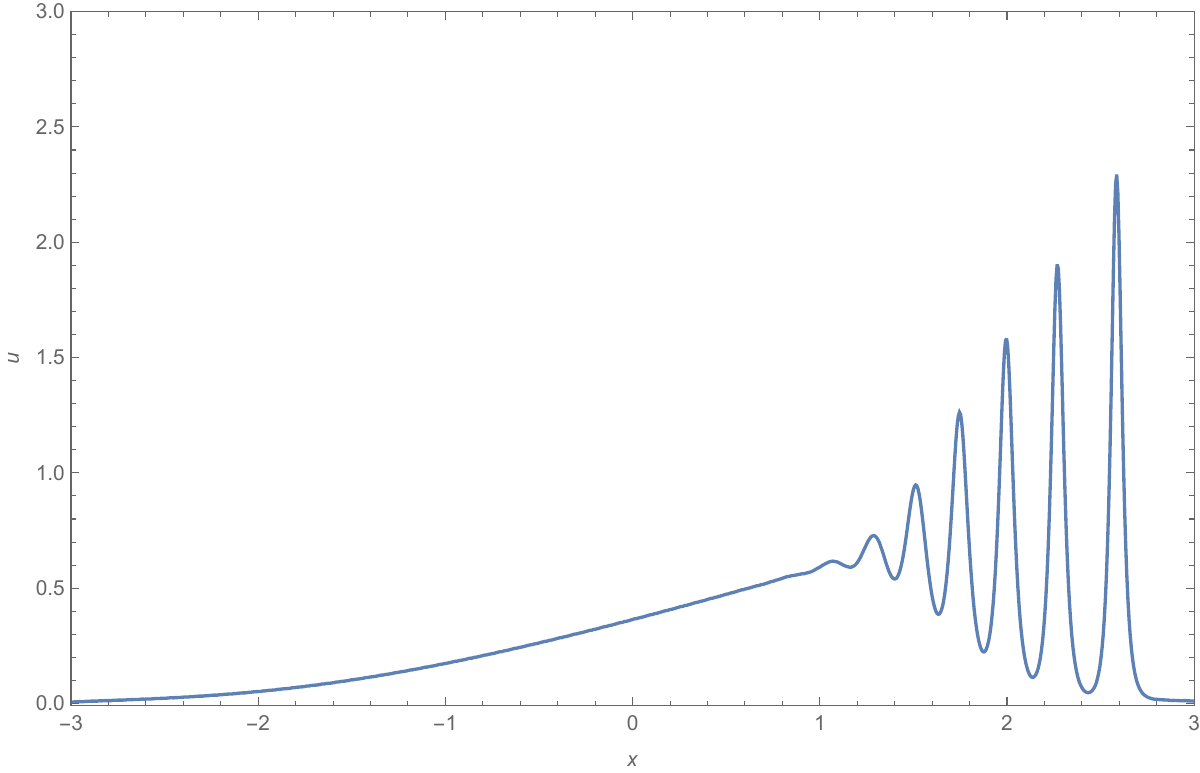}
        \subcaption{$\begin{array}{l} \epsilon = 0.05 \\ t = 1.500 \end{array}$}
        \label{fig:SmallDispEps05T15}
    \end{subfigure}
    \caption{Split-step simulation of the ILW equation in the small-dispersion limit with initial condition $u_0(x) = \sech^2(x)$ on a symmetric periodic domain of length $L=12$, $N_x=2000$ sample points in $x$ and time step $\Delta t = 0.0001$.
    The first column is at a fixed $\epsilon = 0.05$ and for different times between $t=0.00$ and the gradient catastrophe time $t_\mathrm{c} = 3\sqrt{3}/8 \approx 0.650$.
    The second column shows the effect of decreasing $\epsilon$ at fixed time $t = 1.500$.}
    \label{fig:SmallDispDSW}
\end{figure}

\subsection{Formal Direct Scattering for the ILW Equation}
\label{seq:ILWDirect}

Y. Kodama, M. J. Ablowitz and J. Satsuma laid out the formal inverse scattering transform of the ILW equation in 1982 \cite{KodamaAblowitzSatsuma_1982}.
The scattering equation with potential $u$ is
\begin{equation}
    -\ii \epsilon \psi_x^+(x) + (u(x) - \lambda(k) ) \psi^+(x) = \mu(k) \psi^-(x) \, ,
    \label{eq:ILWScatteringEq1}
\end{equation}
where the wavefunction $\psi: \strip{\delta \epsilon} \to \C$ is analytic on the interior of the closed complex strip denoted by
\begin{equation}
    \strip{\delta \epsilon} := \{ \ z \in \C \ | \ -\delta \epsilon \leq \im{z} \leq \delta \epsilon \ \}
\end{equation}
and the notation $\psi^{\pm}$ in \eqref{eq:ILWScatteringEq1} represents the boundary values of the function at the bottom and top edges of the strip:
\begin{equation}
    \psi^\pm(x) := \lim_{y \to \mp \delta \epsilon^\pm} \psi(x + i y) \, .
\end{equation}
Additionally, the spectral parameter is $k \in \C$ which shows up in the equation through the functions $\lambda$ and $\mu$:
\begin{align}
    \lambda(k) := - k \coth(2 \delta k) \quad \text{and} \quad \mu(k) := k \csch(2 \delta k) \, .
\end{align}

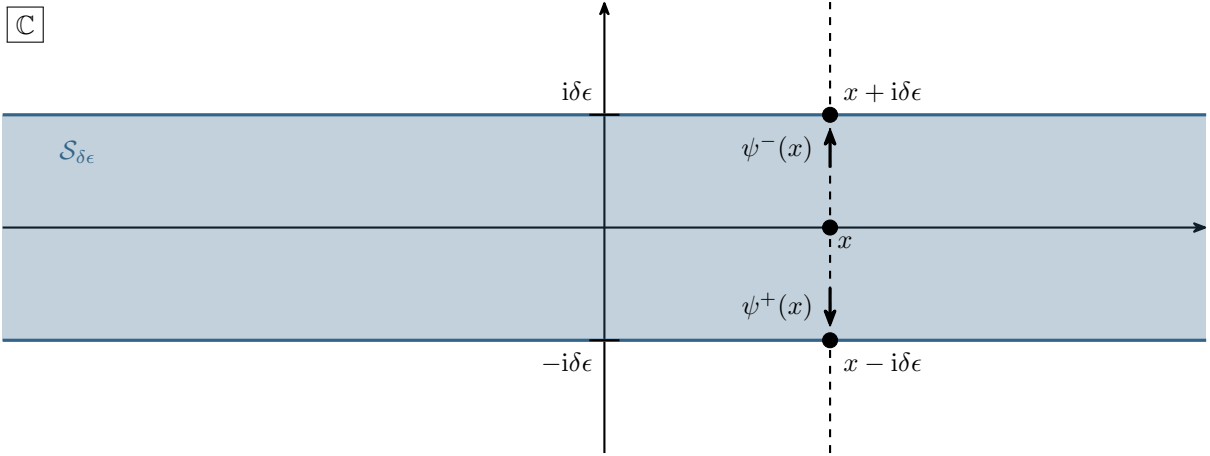
\begin{figure}
    \centering
    \begin{tikzpicture}
        \draw[-{Stealth[round]}, thick] (-8, 0) -- (8, 0);
        \draw[-{Stealth[round]}, thick] (0, -3) -- (0, 3);
        \node at (-7.7, 2.7) {$\boxed{\C}$};

        \fill[navy, opacity=0.3] (-8, 1.5) -- (8, 1.5) -- (8, -1.5) -- (-8, -1.5) -- cycle;
        \draw[very thick, navy] (-8, 1.5) -- (8, 1.5);
        \draw[very thick, navy] (-8, -1.5) -- (8, -1.5);
        \node[navy] at (-7, 1) {$\strip{\delta \epsilon}$};

        \draw[thick] (-0.2, 1.5) -- (0.2, 1.5);
        \node at (-0.5, 1.8) {$\phantom{-} \ii \delta \epsilon$};
        \draw[thick] (-0.2, -1.5) -- (0.2, -1.5);
        \node at (-0.5, -1.8) {$-\ii \delta \epsilon$};
        
        \draw[thick, dashed] (3, -3) -- (3, 3);
        \node at (3, 0) [circle, fill=black, draw=black, inner sep=2]{};
        \node at (3.2, -0.2) {$x$};
        \node at (3, -1.5) [circle, fill=black, draw=black, inner sep=2]{};
        \node at (2.3, -1.05) {$\psi^+(x)$};
        \draw[-{Stealth[round]}, very thick] (3, -0.8) -- (3, -1.3);
        \node at (3.7, -1.8) {$x - \ii \delta \epsilon$};
        \node at (3, 1.5) [circle, fill=black, draw=black, inner sep=2]{};
        \draw[-{Stealth[round]}, very thick] (3, 0.8) -- (3, 1.3);
        \node at (2.3, 1.05) {$\psi^-(x)$};
        \node at (3.7, 1.8) {$x + \ii \delta \epsilon$};
    \end{tikzpicture}
    \caption{The strip domain of analyticity of the wavefunction $\psi$. Also shown are the definitions of the continuous boundary values $\psi^\pm(x)$ via limits from inside the strip.}
    \label{fig:Strip}
\end{figure}

The spectrum for potential $u$ which decays to zero sufficiently fast as $x \to \pm \infty$ is reported to be continuous for $k \in \R$ and discrete for $k = \ii \kappa$ with $0 < \kappa < \pi / 2 \delta$.
\cite{KodamaAblowitzSatsuma_1982} gives a minimal justification that this indeed constitutes the spectrum, relying on the the condition of a particular Neumann series converging, which they show to be true when several norms of the potential and $\delta$ are small enough.
To the knowledge of this paper's author, the only other rigorous analysis of the ILW scattering equation is done in the dissertation of J. Klipfel \cite{Klipfel_2020} which sharpens the analysis in \cite{KodamaAblowitzSatsuma_1982} and shows that the direct scattering map is well-defined for a space of small-norm potentials. 

Note that the ILW scattering equation \eqref{eq:ILWScatteringEq1}, due to the analytic continuation between the boundary values $\psi^+$ and $\psi^-$, might be viewed formally as an ODE of infinite order.
To build some intuition for the solutions of \eqref{eq:ILWScatteringEq1}, we consider the spectrum of the simplest possible potential: $u(x) = 0$ identically.
This can be solved exactly by an exponential ansatz: $\psi(z) = \ee^{\ii \sigma z / \epsilon}$ where
\begin{equation}
    - \sigma - \lambda(k) = \mu(k) \ee^{- 2 \delta \sigma} \, .
    \label{eq:ILWScattAsympWaveFuncSol}
\end{equation}
There are countably infinitely many solutions for $\sigma$ to \eqref{eq:ILWScattAsympWaveFuncSol}, but only two of them are identified in \cite{KodamaAblowitzSatsuma_1982} as important by analogue with the Schr\"odinger operator: $\sigma = \pm k$.
The rest appear to be considered extraneous by the authors.
For $k > 0$, these two exponential solutions are bounded and oscillatory, like the continuous spectrum of the Schr\"odinger operator, while for $k = \ii \kappa$ and $\kappa > 0$, these solutions are real and exponential.
Note that $\lambda(k)$ and $\mu(k)$ blow up for $k = \ii n \pi / 2 \delta$ with $n \in \Z$, thus \cite{KodamaAblowitzSatsuma_1982} only includes the imaginary portion around zero up to the first blow up, $\kappa \in [0, \pi/2\delta)$, in the discrete spectrum.
Using these exponential solutions as the asymptotics for potentials which decay as $x \to \pm\infty$, the authors set up a quantum scattering problem analogous to that of the Schr\"odinger equation to define a reflection coefficient on the continuous spectrum $r: \R \to \C$, eigenvalues $\kappa_n \in [0, \pi/2\delta)$ where normalizable wavefunctions $\psi_n: \strip{\delta \epsilon} \to \C$ exist, that is
\begin{equation}
    \psi_n(z) = \left( 1 + \lilO{1} \right) \left\{\begin{array}{ll}
        \phantom{b} \, \ee^{-\kappa_n z} & \text{ as } \re{z} \to +\infty \\
        b \, \ee^{\kappa_n z} & \text{ as } \re{z} \to -\infty
    \end{array} \right.
\end{equation}
for some $b \in \C$, and their associated norming constants
\begin{equation}
    c_n := \norm{\psi_n}_{L^2(\R)}^{-2} \, .
    \label{eq:AKSNormingConstDef}
\end{equation}
These three constituent types of quantities (the eigenvalues, norming constants and reflection coefficient) formally make up the ILW scattering data.

A reformulation of the ILW scattering equation that also appears in \cite{KodamaAblowitzSatsuma_1982}, and which will be more useful in the present work, requires the substitution 
\begin{equation}
    \psi(z) = \ee^{\ii k z/\epsilon} \phi(z)
    \label{eq:KappaToZetaWaveFuncTransform}
\end{equation}
which, plugging into \eqref{eq:ILWScatteringEq1} yields, after some rearranging
\begin{equation}
    -\ii \epsilon \phi_x^+(x) - u(x) \phi^+(x) = \zeta \left( \phi^+(x) - \phi^-(x) \right)
    \label{eq:ILWScatteringEq2}
\end{equation}
where $\zeta$ is given by 
\begin{equation}
    \zeta = k - \lambda(k) = \mu(k) \ee^{2 \delta k} = \frac{k}{1 - \ee^{-4 \delta k}} \, .
\end{equation}
It turns out $\zeta$ is our preferred spectral parameter over $k$. If we use the definitions above of the discrete and continuous spectrums for $k$, we can see that, in the $\zeta$ complex plane, the continuous spectrum lies on the positive, real axis while the discrete spectrum lies on the central connected component of a $\delta$-scaled quadratrix of Hippias \cite{quadratrix} with parameterization in terms of the imaginary coordinate
\begin{equation}
    \zeta(\kappa) := \kappa \cot(2 \delta \kappa) + \ii \kappa = \kappa \csc(2 \delta \kappa) \ee^{\ii 2 \delta \kappa} \, .
    \label{eq:QuadZetaDef}
\end{equation}
The quadratrix of Hippias is the curve in the complex plane where the argument is equal to the imaginary part.
It consists of an infinite vertical stack of continuous components, one for each branch of the argument function.
The central connected component, henceforth just the quadratrix, is the one that corresponds to the principal argument.
The quadratrix crosses the real axis at the value $1$ and is depicted in Figure \ref{fig:quadratrix}.
To be accurate, the true quadratrix is parameterized by $2 \delta \zeta(\kappa)$ for $-\pi/2 \delta < \kappa < \pi/2 \delta$, but since a value of $\delta$ is always implied in an ILW context, we will use the term quadratrix to refer to both the true quadratrix and the $1/2\delta$-scaled version,
\begin{equation}
    \quadratrix{\delta} := \Big\{ \ \zeta(\kappa) \ \Big| \ -\frac{\pi}{2 \delta} < \kappa < \frac{\pi}{2 \delta} \ \Big\} \, ,
\end{equation}
on which the $\zeta$-eigenvalues of the discrete spectrum lie.
We will also need to distinguish the portion with positive imaginary part $\quadplus{\delta}$ from the portion with negative imaginary part $\quadminus{\delta}$ so that $\quadratrix{\delta} = \quadminus{\delta} \cup \{ 1/2\delta \} \cup \quadplus{\delta}$.
Lastly, we also define an interval notation for arcs along the quadratrix, i.e. for $\eta, \xi \in \quadratrix{\delta}$ such that $\im{\eta} < \im{\xi}$, 
\begin{equation}
    ( \eta, \xi )_{\mathcal{Q}_\delta} = \left\{ \ \zeta(\kappa) \ | \ \im{\eta} < \kappa < \im{\xi} \ \right\} \, ,
    \label{eq:QuadratrixIntervalNotationDef}
\end{equation}
and with the usual square bracket interchanges made to denote inclusion of the endpoints in the arc.

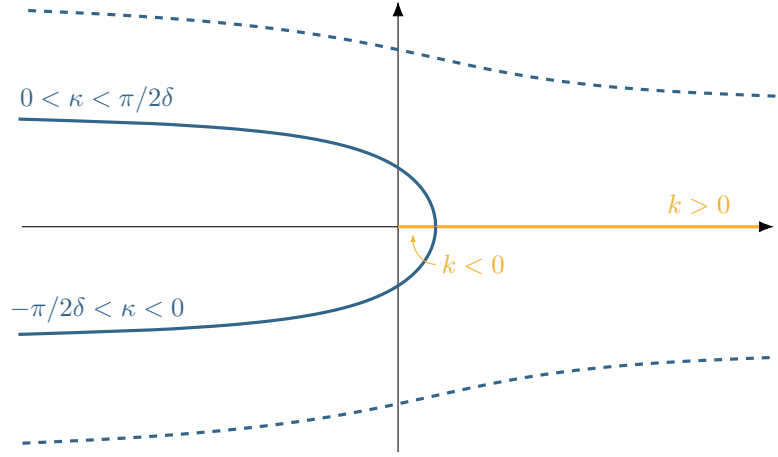
\begin{figure}
    \centering
    \begin{tikzpicture}
        \draw[-{Latex[length=2mm]}] (0, -3) -- (0, 3);
        \draw[-{Latex[length=2mm]}] (-5, 0) -- (5, 0);
        \draw[maize, very thick] (0, 0) -- (4.8, 0);
        \draw[scale=0.5, domain=3.475:5.76, smooth, variable=\y, navy, very thick, dashed]  plot ({\y*cos(\y r)/sin(\y r)}, {\y});
        \draw[scale=0.5, domain=0.01:2.865, smooth, variable=\y, navy, very thick]  plot ({\y*cos(\y r)/sin(\y r)}, {\y});
        \draw[scale=0.5, domain=-2.865:-0.01, smooth, variable=\y, navy, very thick]  plot ({\y*cos(\y r)/sin(\y r)}, {\y});
        \draw[scale=0.5, domain=-5.76:-3.475, smooth, variable=\y, navy, very thick, dashed]  plot ({\y*cos(\y r)/sin(\y r)}, {\y});
        \node[navy] at (-4, 1.7) {$0 < \kappa < \pi/2\delta$};
        \node[navy] at (-4, -1.1) {$-\pi/2\delta < \kappa < 0$};
        \node[maize] at (4, 0.3) {$k > 0$};
        \node[maize] at (1, -0.5) {$k < 0$};
        \draw[-{Latex[length=1mm]}, maize] (0.5, -0.5) to [out=180,in=-90] (0.2, -0.1);
    \end{tikzpicture}
    \caption{The proposed ILW scattering equation spectrum in the $\zeta$ plane.
    The blue contours are the visible portion of the quadratrix of Hippias, solid for the quadratrix and dashed for the other components.
    The discrete spectrum shows up on the quadratrix and the continuous spectrum is labeled in solid yellow.}
    \label{fig:quadratrix}
\end{figure}

\subsection{Inverse Scattering for the ILW Equation}
\label{seq:ILWInverse}

Y. Kodama, M. J. Ablowitz and J. Satsuma continue in \cite{KodamaAblowitzSatsuma_1982}, formally dictating the inverse scattering problem via a Gelfan-Levitan-Marchenko (GLM) equation governing the reconstruction of the potential to the ILW scattering equation given a set of scattering data.
For the case of a zero reflection coefficient and $N \in \N$ distinct eigenvalues $\{ \zeta_n \}_{n = 1}^N \subset \quadplus{\delta}$ with their associated norming constants $\{ c_n \}_{n = 1}^N \subset \R_+$, the GLM equation can be solved explicitly; the solution being the $N$-soliton formula,
\begin{equation}
    u_N(x, t; \epsilon) = -\ii \epsilon \partial_x \log \det \left( \frac{\id{N} + \Delta_N(x + \ii\delta \epsilon, t; \epsilon)}{\id{N} + \Delta_N(x - \ii\delta \epsilon, t; \epsilon)} \right) \, ,
    \label{eq:ILWNSoliton}
\end{equation}
where $I_N$ is the $N \times N$ identity matrix and $\Delta_N(z, t; \epsilon)$\footnote{The definition of $\Delta_N(z, t; \epsilon)$ in \eqref{eq:DeltaDefComponents} differs from that reported in \cite{KodamaAblowitzSatsuma_1982} by a simple matrix similarity transformation. Because of the determinants in \eqref{eq:ILWNSoliton}, the two matrix definitions yield identical ILW solutions.} is an $N \times N$ matrix defined element-wise for $z \in \C$, $t \in \R$ and $\epsilon > 0$ by
\begin{equation}
    \left[ \Delta_N(z, t; \epsilon) \right]_{n \, \ell} := c_n c_\ell \ \ii \frac{\displaystyle \exp \left( - \im{\zeta_n + \zeta_\ell} \frac{z}{\epsilon} - \im{\left( \zeta_n - \frac{1}{2 \delta} \right)^2 + \left( \zeta_\ell - \frac{1}{2 \delta} \right)^2} \frac{t}{\epsilon} \right)}{\zeta_\ell - \zeta_n^*} \text{ for } 1 \leq n, \ell \leq N \, ,
    \label{eq:DeltaDefComponents}
\end{equation}
A small note on the notation: the matrix ``division'' in \eqref{eq:ILWNSoliton} is shorthand for the multiplication by the inverse and, due to the determinant operation, the ordering of this matrix multiplication is irrelevant.
Finally, it will be a useful fact that the $\Delta_N(z, t; \epsilon)$ matrix \eqref{eq:DeltaDefComponents} has a particular factorization
\begin{equation}
    \Delta_N(z, t; \epsilon) = D_N(z, t; \epsilon) \ C_N \ D_N(z, t; \epsilon) \, ,
    \label{eq:DeltaDecomp}
\end{equation}
where the $N \times N$ matrix $D_N(z, t; \epsilon)$ is diagonal and carries all of the $z$ and $t$ dependence,
\begin{equation}
    \left[ D_N(z,t; \epsilon) \right]_{n \, n} := c_n^{1/2} \exp \left( -\im{\zeta_n} \frac{z}{\epsilon} - \im{\left(\zeta_n - \frac{1}{2 \delta}\right)^2} \frac{t}{\epsilon} \right) \text{ for } 1 \leq n \leq N \, ,
\end{equation}
and $C_N$ is an $N \times N$ Cauchy matrix,
\begin{equation}
    \left[ C_N \right]_{n \, \ell} := \frac{\ii}{\zeta_\ell - \zeta_n^*} \text{ for } 1 \leq n, \ell \leq N \, .
\end{equation}
For a proof that \eqref{eq:ILWNSoliton} is indeed an exact solution of the ILW equation, see \cite[Chapter 5]{Matsuno_1984_BiTransMeth}.

\subsection{Results and Outline}
\label{sec:Outline}

The $N$-soliton formula \eqref{eq:ILWNSoliton} provides us with a rigorous basis for analysis of the asymptotics of the small-dispersion limit of the ILW equation as a semiclassical soliton ensemble.\footnote{The term ``semiclassical'' comes from quantum mechanics. The small-dispersion limit of the KdV equation corresponds under the scattering transform exactly to a semiclassical analysis of the Schr\"odinger equation. So in this work, ``small dispersion'' and ``semiclassical'' mean the same thing on the PDE and spectral sides of the IST, respectively.}
The use of semiclassical soliton ensembles is a common technique for simplifying the analysis of small-dispersion limits and has been employed for several nonlinear integrable PDE; Examples include the KdV \cite{LaxLevermoreI_1983, LaxLevermoreII_1983, Venakides_1990, DeiftVenakidesZhou_1998}, BO \cite{XuMiller_2010}, focusing Nonlinear Schr\"odinger \cite{KamvissisMcLaughlinMiller_2003_semiclass, JenkinsMcLaughlin_2014, BuckinghamJenkinsMiller_2025}, and sine-Gordon \cite{BuckinghamMiller_2013} equations.
Those familiar with the literature on small-dispersion limits will likely also know that they predominantly involve Riemann-Hilbert problems, employing the method of nonlinear steepest descent to achieve uniformly valid asymptotics of the solution on compact subsets of spacetime.
Examples of this include for the KdV equation \cite{ClaeysGrava_2009, ClaeysGrava_2010_SolEdge, ClaeysGrava_2010_DispEdge} as well as those cited above for the focusing Nonlinear Sch\"odinger and sine-Gordon equations.
Since there is currently no established Riemann-Hilbert problem which describes the ILW equation's inverse scattering, we are motivated to return to the methods of P. D. Lax and C. D. Levermore, prior to the adoption of Riemann-Hilbert.

In this paper, we follow the structure of and make use of theorems from \cite{LaxLevermoreI_1983, LaxLevermoreII_1983} where possible.
Like these works, we restrict to a particular class of initial conditions.
\begin{definition}[Admissible initial condition.] \label{def:AdmisInitCond}
    We say an initial condition to the ILW equation $u_0: \R \to \R$ is admissible iff it is a $\mathcal{C}^1(\R)$, positive, single lobed function where $u_0(x) \to 0$ as $x \to \pm \infty$ and the first derivative is nonvanishing except at $x_{\maxrm} \in \R$ where the maximum $u_0(x_\maxrm) = u_{\maxrm} > 0$ is attained.
    Additionally, we require
    \begin{equation}
        \int_\R (x^2 + 1) \, u_0(x) \, \dd x < \infty \, . \label{eq:DecayAssumInitCond}
    \end{equation}
\end{definition}

Section \ref{sec:Direct} gives a formal WKB analysis of the ILW scattering equation in the semiclassical limit $\epsilon \to 0^+$.
We extend the formal derivation of a Weyl law for eigenvalues in the ILW scattering equation first conducted in \cite{MinzoniMiloh_1994}, making their formulas explicit using classical functions.
Additionally, we also present an entirely new formula for the formal asymptotics of the norming constants.
This section culminates in the following conjecture:

\begin{conjecture}[Semiclassical asymptotics of the ILW scattering data] \label{conj:WeylLaw}
    Let $u_0$ be an admissible initial condition according to Definition \ref{def:AdmisInitCond} and let $R^\wyl: [1/2\delta, \zeta_\maxrm]_{\mathcal{Q}_\delta} \to \R_+$ be defined by
    \begin{equation}
        R^\wyl(\zeta) := \frac{1}{4 \delta} \int_{x_-(\zeta)}^{x_+(\zeta)} W_{0}\left(-2 \delta \zeta \ee^{-2\delta(\zeta + u_0(x))} \right) - W_{-1}\left(-2 \delta \zeta \ee^{-2\delta(\zeta + u_0(x))}\right) \, \dd x \, .
        \label{eq:RWeylLawDef}
    \end{equation}
    where $x_\pm : (1/2\delta, \zeta_\maxrm]_{\quadratrix{\delta}} \to \R$ are the turning point functions defined implicitly by
    \begin{equation}
        u_0(x_\pm(\zeta)) = E(\zeta) 
        \label{eq:TurningPointFunctDef}
    \end{equation}
    with $E$ as in \eqref{eq:TurningPointFromELevel} and such that $x_-(\zeta) \leq x_\maxrm \leq x_+(\zeta)$.
    Additionally, $W_n$ for $n \in \Z$ are the standard branches of the Lambert-$W$ function \cite[Equation 4.13.1\_1]{DLMF}.
    The asymptotics of the eigenvalues satisfy the ``ILW Weyl law''
    \begin{equation}
        R^\wyl(\zeta_n) = \left( n - \frac{1}{2} \right) \pi \epsilon + \bigO{\epsilon^2}
        \label{eq:WeylLaw}
    \end{equation}
    for $n \in \N$.
    Defining the tail integrals $\theta_\pm: [1/2\delta, \ \zeta_{\maxrm}]_{\quadratrix{\delta}} \to \R$
    \begin{equation}
        \theta_\pm(\zeta) := \im{\zeta} x_\pm(\zeta) + \int^{\pm\infty}_{x_\pm(\zeta)} \im{ \zeta + \frac{1}{2 \delta} W_{-1}\left(-2\delta \zeta \ee^{-2\delta(\zeta + u_0(x))}\right)} \, \dd x \, .
        \label{eq:TailIntegralsDef}
    \end{equation}
    Then the norming constants have the asymptotics
    \begin{align}
        c_n &= \exp \Bigg( \frac{2}{\epsilon} \theta_+(\zeta_n) \Bigg) \left| \frac{(2 \delta)^2 \zeta_n}{1 -2 \delta \zeta_n} \right| \nonumber \\
        &\Quad[2] \times \left( \int_{x_-(\zeta_n)}^{x_+(\zeta_n)} \frac{2\delta W_{-1}(-2\delta \zeta_n \ee^{-2\delta(\zeta_n + u_0(x))})}{1+W_{-1}(-2\delta \zeta_n \ee^{-2\delta(\zeta_n + u_0(x))})} - \frac{2\delta W_0(-2\delta \zeta_n \ee^{-2\delta(\zeta_n + u_0(x))})}{1+W_0(-2\delta \zeta_n \ee^{-2\delta(\zeta_n + u_0(x))})} \, \dd x \right)^{-1} \nonumber \\
        &\Quad[2] \times \left( 1 + \bigO{\epsilon^{1/6}} \right) \, . \label{eq:NormingConstAsymptoticFormula}
    \end{align}
    Finally, the reflection coefficient is small beyond all orders in $\epsilon$.
\end{conjecture}

In section \ref{sec:Inverse}, we define the ILW semiclassical soliton ensemble $u_N^\sse$ (Definition \ref{def:SemiclassSoliEnsem}) generated by a modified (in an asymptotically small way) set of scattering data (Definition \ref{def:ModScattData}) and which solves the ILW equation \eqref{eq:ILW} exactly for a particular sequence $\{ \epsN \}_{N \in \N}$ \eqref{eq:ModScattEpsN} so that as $N \to +\infty$, $\epsN \to 0^+$.
Analyzing $u_N^\sse$ as $N \to +\infty$, we prove several lemmas and theorems analogous to those by P. D. Lax and C. D. Levermore for the KdV equation, surmounting some additional difficulties that are unique to the ILW equation's inverse scattering problem.
Sections \ref{sec:ExpanDet} and \ref{sec:DiscrToContMeas} establish the distributional limit of the semiclassical soliton ensemble (Theorem \ref{thm:DistLimit}) is given in terms of an equilibrium measure which solves a particular Greens' energy minimization problem (Theorem \ref{thm:DescMinEnergyToContMinEnergy}).
This equilibrium measure is shown in Section \ref{sec:Variations} to be characterized by variational conditions (Theorem \ref{thm:VariationCondDetMin}).
We solve the variational conditions, under a particular simplifying assumption, by first solving differentiated variational conditions with respect to $x$ and $t$, consistently integrating up the solutions and then verifying the result solves the original variational conditions (Theorem \ref{thm:ConstLogPotentChargeDensSol}).
Finally, the resulting limit of the soliton ensemble solution is upgraded from distributional to convergence in $L^2$-norm for small time:
\begin{theorem} \label{thm:L2Limit}
    Let $u_0$ satisfy Definition \ref{def:AdmisInitCond} and $u_N^\sse(\diamond, t)$ be the associated semiclassical soliton ensemble according to Definition \ref{def:SemiclassSoliEnsem}.
    Furthermore, let $u^\Brm$ denote the solution to invicid Burgers' equation \eqref{eq:IB} with initial condition $u^\Brm(x, 0) = u_0(x)$.
    Recall that this is only globally well-defined up until the time of gradient catastrophe $t_\mathrm{c}$ \eqref{eq:CatTime}.
    Then as $N \to +\infty$, 
    \begin{equation}
        \lim_{N \to +\infty} u_N^\sse(\diamond, t) =  u^\Brm(\diamond, t)
    \end{equation}
    converging in $L^2(\R)$ norm uniformly for all $0 \leq t < t_\mathrm{c}$.
\end{theorem}

To conclude, we discuss in Section \ref{sec:Conclusion} the expectation for the DSW after the gradient catastrophe time $t_\mathrm{c}$ being described by modulated multi-periodic solutions to the ILW equation.

\section{Formal WKB for the ILW Scattering Equation}
\label{sec:Direct}

A. A. Minzoni and T. Miloh were the first to analyze the semiclassical asymptotics of the eigenvalues of the ILW scattering equation \eqref{eq:ILWScatteringEq1} for Klause-Shaw potentials where the number of eigenvalues increases as $\epsilon \to 0^+$ \cite{MinzoniMiloh_1994}.
Our goal in this section is to present their arguments with several key modifications.
\begin{enumerate}
    \item Their analysis reveals an infinite number of asymptotic solutions to the scattering equation of which they throw away all but two, whereas we categorize these solutions into those that can form normalizable wavefunctions (scattering solutions) and those that cannot (extraneous solutions).
    
    \item They expand the wavefunction along the lower boundary of the strip and neglect information about the solution interior to in the strip or on the upper boundary, whereas we expand the wavefunction around the center of the strip, account for the $\epsilon$-sized imaginary component of the strip, and give an expansion formally valid for the whole strip.
    
    \item Their asymptotic solutions for the wavefunction are given in terms of implicit solutions to transcendental equations, whereas we write our solutions explicitly in terms of classical functions.
    
    \item They only identify and make use of one solution to the proposed model turning point equation, whereas we provide asymptotics of the infinitely many solutions and demonstrate via a matching argument that indeed the only one which was identified in \cite{MinzoniMiloh_1994} participates in the construction of normalizable wavefunctions.
\end{enumerate}
Additionally, we also provide the entirely new result of a formal asymptotic formula for the norming constants associated to the eigenvalues in the ILW scattering data.

\subsection{Asymptotic Solutions Away from the Turning Points}
\label{sec:NonTurningPoint}

We first make the substitution in the ILW scattering equation \eqref{eq:ILWScatteringEq2}
\begin{equation}
    \phi(z) = \Phi(z) \exp \left( \ii \frac{\zeta}{\epsilon} z \right)
    \label{eq:CannonicalTransform}
\end{equation}
to put \eqref{eq:ILWScatteringEq2} into what we will call ``canonical form,''
\begin{equation}
    \ii \epsilon \Phi^+_x(x) + u(x) \Phi^+(x) = \zeta \ee^{-2 \delta \zeta} \Phi^-(x) \, .
    \label{eq:CannonFormScatteringEq}
\end{equation}
In this section, we take the potential in the ILW scattering equation $u$ to be an admissible initial condition, satisfying Definition \ref{def:AdmisInitCond}.
We wish to learn something about the solutions to \eqref{eq:CannonFormScatteringEq} as we take the semiclassical limit $\epsilon \to 0^+$.
This was done in \cite{MinzoniMiloh_1994} by conducting a WKB-style approximation similar to that for semiclassical limit of the time-independent Schr\"odinger equation.
Supposing that $\Phi(z)$ does not vanish in the strip, so that we can make the exponential substitution
\begin{equation}
    \Phi(z) = \exp \left( S(z) \right)
    \label{eq:ExponentialSForm}
\end{equation}
and, moving all the exponentials to one side, we obtain the new form of the scattering equation
\begin{equation}
    \ii \epsilon S_x^+ + u = \zeta \exp \left( -2 \delta \zeta + S^- - S^+ \right) \, .
    \label{eq:ScatEqSForm}
\end{equation}
If we begin looking for a dominant balance between the two terms on the left, this suggests a formal expansion in integer powers of $\epsilon$ with leading coefficient $1/\epsilon$,
\begin{equation}
    S(z) = -\frac{\ii}{\epsilon} \theta(z) + \alpha(z) + \ii \epsilon \beta(z) + \bigO{\epsilon^2} \, .
\end{equation}
Note though, that because the vertical width of the strip $\strip{\delta \epsilon}$ is shrinking as $\epsilon \to 0^+$, the real and imaginary parts of $z = x + \ii y$ should contribute at different orders to this series, i.e. since $-\delta \epsilon \leq y \leq \delta \epsilon$, $y$ is itself $\bigO{\epsilon}$.
If the functions are assumed analytic in the expansion, then we should Taylor expand them away from the center of the strip to be able to accurately order every term in the series:
\begin{equation}
    S(z) = -\frac{\ii}{\epsilon} \theta(x) + \left( \alpha(x) + \frac{y}{\epsilon} \theta_x(x) \right) + \ii \epsilon \left( \beta(x) + \frac{y}{\epsilon} \alpha_x(x) + \frac{y^2}{\epsilon^2} \theta_{xx}(x) \right) + \bigO{\epsilon^2} \, .
    \label{eq:AccurateOrderedSExpansion}
\end{equation}
Plugging this formal expansion into \eqref{eq:ScatEqSForm}, we note that this also brings the right side into the dominant balance,
\begin{align}
    \theta_x + u + \ii \epsilon \alpha_x - \ii \delta \epsilon \theta_{xx} + \bigO{\epsilon^2} &= \zeta \exp \left( -2 \delta \zeta +2 \delta \theta_x + \ii 2 \delta \epsilon \alpha_x + \bigO{\epsilon^2} \right) \\
    &= \zeta \ee^{-2 \delta (\zeta - \theta_x)} \left( 1 + \ii 2 \delta \epsilon \alpha_x + \bigO{\epsilon^2} \right) \, .
\end{align}
Equating order-by-order in $\epsilon$, we have the relations determining the formal functions $\theta$ and $\alpha$,
\begin{align}
    \theta_x + u &= \zeta \ee^{-2 \delta (\zeta - \theta_x)} \, , \label{eq:AsympThetaRelation} \\
   \alpha_x - \delta \theta_{xx} &= 2 \delta \zeta \ee^{-2 \delta (\zeta - \theta_x)} \alpha_x \, . \label{eq:AsympARelation}
\end{align}
We will define a function for $\eta, \zeta \in \C$ by
\begin{equation}
    G(\eta; \zeta) := -\eta + \zeta \ee^{-2 \delta (\zeta - \eta)} \, ,
    \label{eq:GDef}
\end{equation}
then the relationship \eqref{eq:AsympThetaRelation} determines $\theta_x$ by the equation
\begin{equation}
    G(\theta_x(x); \zeta) = u(x)
    \label{eq:ThetaImplicit}
\end{equation}
for $x \in \R$.
It is clear that this relationship is transcendental with infinitely many solutions due to the exponential function in $G$.
However, with some simple manipulations, we can solve this equation using the  Lambert $W$-function \cite[Section 4.13]{DLMF}, also known as the product-logarithm (See Section \ref{sec:LambertWExplained}).
We will let $U \in \R$ stand for the output of $G$ to define the branched inverse function for $G$ on its first variable:
\begin{equation}
    G^{-1}_n(U; \zeta) := -U - \frac{1}{2 \delta} W_n(-2 \delta \zeta \ee^{-2 \delta (\zeta + U)}) \, .
    \label{eq:GInvDef}
\end{equation}
where $n \in \Z$ denotes which of the standard branches of the Lambert $W$-function we use.
This means for each $n \in \Z$, there is a leading-order phase which is given by
\begin{equation}
    \theta_{x}(x) = G^{-1}_n(u(x); \zeta) \, .
    \label{eq:ThetaExplicit}
\end{equation}
Once we have a solution for $\theta_x$, now we look at the next-to-leading-order term, determined by \eqref{eq:AsympARelation}, and rearranging gives
\begin{equation}
    \alpha_x(x) = \frac{\delta \ \theta_{xx}(x)}{\displaystyle 1 - 2 \delta \zeta \ee^{-2 \delta (\zeta - \theta_x(x))}} = \partial_x  \int^{\theta_x(x)} \frac{\delta \, \dd \eta}{\displaystyle 1 - 2 \delta \zeta \ee^{-2 \delta (\zeta - \eta)}} \, .
\end{equation}
This integral can be evaluated with a change of variables $u = 2 \delta \zeta \ee^{-2 \delta (\zeta - \eta)}$ and then a partial-fraction decomposition. Using the implicit definition of $\theta_x$ \eqref{eq:ThetaImplicit} as well as its explicit equality \eqref{eq:ThetaExplicit} we find the result:
\begin{align}
    \alpha(x) = \int^{\theta_x(x)} \frac{\delta \, \dd \eta}{1 - 2 \delta \zeta \ee^{-2 \delta (\zeta - \eta)}} = \frac{1}{2} \log \left( \frac{\displaystyle - W_n \left( -2 \delta \zeta \ee^{-2 \delta (\zeta - u(x))} \right)}{\displaystyle 1 + W_n \left( -2 \delta \zeta \ee^{-2 \delta (\zeta - u(x))} \right)} \right) + C
    \label{eq:AlphaIntegralResult}
\end{align}
where $C \in \C$ is an integration constant.
The presence of the logarithm in \eqref{eq:AlphaIntegralResult} and the exponential in \eqref{eq:ExponentialSForm} motivates the definition of the amplitude functions for $n \in \Z$, $U \in \R$ and $\zeta \in \C$
\begin{equation}
    A_n(U; \zeta) := \left( \frac{\displaystyle -W_n \left( -2 \delta \zeta \ee^{-2 \delta (\zeta + U} \right)}{\displaystyle 1 + W_n \left( -2 \delta \zeta \ee^{-2 \delta (\zeta + U)} \right)} \right)^{1/2}
    \label{eq:FormalAsymptoticAmplitude}
\end{equation}
where we take the principal square-root for $U > 0$ near zero and then use the continuous branch from then on.
Should we run into a zero or singularity for $U \in \R$ inside of the square-root, we can always arbitrarily choose to extend the function around this by taking the branch that is continuous for $U$ with a small positive imaginary part.
Then, choosing a branch $n \in \Z$ of the Lambert $W$-function and taking the integration constant to be zero (a freedom we have because the scattering equation is linear), we have from \eqref{eq:AlphaIntegralResult} $\ee^{\alpha(x)} = A_n(u(x); \zeta)$.
Integrating $\theta_x$ from \eqref{eq:ThetaExplicit} in $x$, and assembling the remaining terms in \eqref{eq:AccurateOrderedSExpansion} to constant order in $\epsilon$, we can define the formal expansions for the solutions to the scattering equation for each $n \in \Z$, any spectral parameter value $\zeta \in \C$ and any initial point $x_0 \in \R$
\begin{align}
    \Phi_n(z; \zeta, x_0) &:= A_n \big( u(x); \zeta \big) \exp \Bigg( -\frac{\ii}{\epsilon} \int^x_{x_0} G^{-1}_n(u(x'); \zeta) \, \dd x' + \frac{y}{\epsilon} G^{-1}_n(u(x); \zeta) \Bigg) \big(1 + \bigO{\epsilon} \big) \, .
    \label{eq:FormalAsympScatteringSolution}
\end{align}

We note that this formal asymptotic solution \eqref{eq:FormalAsympScatteringSolution} shares many characteristics with the solution generated from WKB analysis of the time-independent Schr\"odinger equation. The fast instantaneous phase coefficient $\theta_x/\epsilon$ and the slowly varying amplitude $A = \ee^\alpha$ are both locally determined as functions of the potential.
There also appear to be quite a few differences, like the fact that there are infinitely many formal asymptotic solutions.
We also note that this limit appears to be singular, in the sense that we lose all information about the wavefunction's analyticity inside the strip: if $u$ is not analytic, $G^{-1}_n(u(x); \zeta)$, its integral, or $A_n\big(u(x); \zeta \big)$ are not either.
Of course, this was to be expected as the strip's vertical width shrinks to zero, and to keep terms well-aligned in our formal $\epsilon$-series, we had to break the analyticity at the outset.

\subsubsection{Explanation of the Standard Branches of the  Lambert $W$-function and Identifying Extraneous and Scattering Solutions}
\label{sec:LambertWExplained}

To understand the behavior of the solutions, we give a brief overview of the standard branch definitions of the Lambert $W$-function.
The Lambert $W$-function is the branched inverse of the complex analytic function $w \mapsto w\ee^w$, where the range of each branch in the $w$-plane is a maximal domain of univalence of the inverse.
The standard branch definitions have as the boundaries of their ranges the curves defined by $w \ee^w < 0$, excluding the line segment from $w = -1$ to $w = 0$.
One of the connected curves of these branch range boundaries is exactly the negative quadratrix $- \quadratrix{1} = \{ -z \, | \, z \in \quadratrix{1} \}$.
There are an infinite number of branches, due to the exponential in the mapping $w \mapsto w\ee^w$, falling into three qualitative categories:
\begin{enumerate}
    \item $\mathbf{n \neq \pm1, 0}$. These branches $W_n(z)$ all have one branch point at $z = 0$ with cuts taken along the negative real axis.
    The output on the cut axis are taken to be continuous with the values above the cut.
    For $n > 0$, $W_n(z)$ for $z$ just above the cut limit to a smooth curve from $z = + \infty + \ii(2n+3)\pi$ to $z = -\infty + \ii(2n+1)\pi$. From below, they limit to another smooth curve from $z = + \infty + \ii(2n+1)\pi$ to $z = -\infty + \ii2n\pi$. In this sense, these branch ranges behave very similar to the branch ranges of the logarithm, but right sides of the ranges have been sheared vertically by $\pi$. For $n < 0$, the right sides are sheared vertically by $-\pi$.

    \item $\mathbf{n = \pm 1}$. $W_{\pm 1}(z)$ both have two branch points, one at $z = 0$ which is logarithmic in nature, and one at $z = -\ee^{-1}$ which has a square-root behavior.
    Again, the outputs along the branch cuts along the negative reals are taken to be continuous with those just above. $W_1(z)$ for $z$ just above the cut maps to a smooth curve from $z = +\infty + \ii 3\pi$ to $z = -\infty + \ii2\pi$.
    For $z$ just below the cut, the behavior is different and depends on which side of the branch point $z = -\ee^{-1}$ we limit to: for $\re{z} < -\ee^{-1}$, this approaches the upper half of $- \quadratrix{1}$, and if $-\ee^{-1} < \re{z} < 0$, then this approaches the half line $(-\infty, -1)$.
    $W_{-1}(z)$ is similar but the behavior is flipped between the directions approaching the cut and its now the lower half of $- \quadratrix{1}$.

    \item $\mathbf{n=0}$. $W_0(z)$ has only one branch point at $z = -\ee^{-1}$ of a square-root nature and the cut extends out along the negative real axis.
    This branch maps the entire plane minus the cut to the interior of $- \quadratrix{1}$ in a one-to-one manner.
    $W_0(z)$ for $z$ just above the cut approach the upper half of $- \quadratrix{1}$ while those below map to the lower half. 
    
\end{enumerate}
For a graphical representation of the Lambert $W$-function, see Figure \ref{fig:LambertWBranches}.
\begin{figure}
    \centering
    \begin{tikzpicture}[scale=0.96]
        \node at (-8, 7) {\Large $\boxed{z}$};
        \draw[-{Latex[length=2mm]}] (0, -4) -- (0, 4);
        \draw[-{Latex[length=2mm]}] (-4, 0) -- (4, 0);
        \fill[fire, opacity=0.3, variable=\y] plot[scale=0.25, domain=-13.25:-14.96] ({-\y*cos(\y r)/sin(\y r)}, {\y}) -- plot[scale=0.25, domain=-8.92:-6.675] ({-\y*cos(\y r)/sin(\y r)}, {\y});
        \fill[navy, opacity=0.3, variable=\y] plot[scale=0.25, domain=-6.675:-8.92] ({-\y*cos(\y r)/sin(\y r)}, {\y}) -- plot[scale=0.25, domain=-2.96:-0.01] ({-\y*cos(\y r)/sin(\y r)}, {\y}) -- (-4, 0);
        \fill[maize, opacity=0.3, variable=\y] plot[scale=0.25, domain=-0.01:-2.96] ({-\y*cos(\y r)/sin(\y r)}, {\y}) -- plot[scale=0.25, domain=2.96:0.01] ({-\y*cos(\y r)/sin(\y r)}, {\y});
        \draw[scale=0.25, domain=-8.92:-6.675, smooth, variable=\y, fire, very thick]  plot ({-\y*cos(\y r)/sin(\y r)}, {\y});
        \draw[navy, very thick] (-0.2, 0) -- (-4, 0);
        \draw[scale=0.25, domain=-0.01:-2.96, smooth, variable=\y, navy, very thick]  plot ({-\y*cos(\y r)/sin(\y r)}, {\y});
        \draw[scale=0.25, domain=2.96:0.01, smooth, variable=\y, maize, very thick]  plot ({-\y*cos(\y r)/sin(\y r)}, {\y});
        \draw[scale=0.25, domain=8.92:6.675, smooth, variable=\y, black, very thick]  plot ({-\y*cos(\y r)/sin(\y r)}, {\y});
        \draw[scale=0.25, domain=13.25:14.96, smooth, variable=\y, black, very thick]  plot ({-\y*cos(\y r)/sin(\y r)}, {\y});
        \node at (3.2, -3.0) {$n = -2$};
        \node at (3.2, -1.5) {$n = -1$};
        \node at (3.1, 0.4) {$n = 0$};
        \node at (3.2, 1.5) {$n = +1$};
        \node at (3.2, 3.0) {$n = +2$};
        
        \node at (-8, 0) {\Large $\boxed{W_n(z)}$};
        \draw[-{Latex[length=2mm]}, thick] (5, 4.75) to [out=-90, in=0] (4.25, 0);
        \draw[-{Latex[length=2mm]}, thick] (0, 4.75) to [out=-135, in=90] (-2, 0.25);
        \draw[-{Latex[length=2mm]}, thick] (-5, 4.75) to [out=-90, in=180] (-4.25, -2.5);
        
        \draw[-{Latex[length=2mm]}] (5, 5) -- (5, 9);
        \draw[-{Latex[length=2mm]}] (4, 7) -- (7, 7);
        \draw[decorate, decoration={zigzag, segment length=4, amplitude=1}] (3, 7) -- (4, 7);
        \fill[maize, opacity=0.3] (3, 6.85) -- (4, 6.85) arc (-90:90:0.15) -- (3, 7.15) -- (3, 9) -- (7, 9) -- (7, 5) -- (3, 5);
        \draw[maize, very thick] (4.15, 7) arc (0:90:0.15) -- (3, 7.15);
        \draw[navy, very thick] (3, 6.85) -- (4, 6.85) arc (-90:0:0.15);
        \node at (4, 7) {\textbullet};
        
        \draw[-{Latex[length=2mm]}] (0, 5) -- (0, 9);
        \draw[-{Latex[length=2mm]}] (0, 7) -- (2, 7);
        \draw[decorate, decoration={zigzag, segment length=4, amplitude=1}] (-2, 7) -- (0, 7);
        \fill[navy, opacity=0.3] (-2, 6.85) -- (0, 6.85) arc (-90:90:0.15) -- (-2, 7.15) -- (-2, 9) -- (2, 9) -- (2, 5) -- (-2, 5);
        \draw[navy, very thick] (-2, 7.15) -- (0, 7.15) arc (90:0:0.15);
        \draw[fire, very thick] (-2, 6.85) -- (0, 6.85) arc (-90:0:0.15);
        \node at (-1, 7) {\textbullet};
        \node at (0, 7) {\textbullet};
        
        \draw[-{Latex[length=2mm]}] (-5, 5) -- (-5, 9);
        \draw[-{Latex[length=2mm]}] (-5, 7) -- (-3, 7);
        \draw[decorate, decoration={zigzag, segment length=4, amplitude=1}] (-7, 7) -- (-5, 7);
        \draw[fire, very thick] (-4.85, 7) arc (0:90:0.15) -- (-7, 7.15);
        \fill[fire, opacity=0.3] (-7, 6.85) -- (-5, 6.85) arc (-90:90:0.15) -- (-7, 7.15) -- (-7, 9) -- (-3, 9) -- (-3, 5) -- (-7, 5);
        \node at (-5, 7) {\textbullet};
        
    \end{tikzpicture}
    \caption{$z \mapsto W_n(z)$ demonstrated for three characteristic branches: $n = -2, -1,$ and $0$. The colored $z$-planes map into the same colored branch ranges indicated below. Additionally, the colored edges in the $z$-planes indicate the direction that the same colored range boundaries are approached from each side of the branch cut. The edge of the same color as the plane are the ones that are included in that branch definition. The $n = 1$ maps similar to the $n = -1$ branch while all others are similar to the $n = -2$ branch.}
    \label{fig:LambertWBranches}
\end{figure}
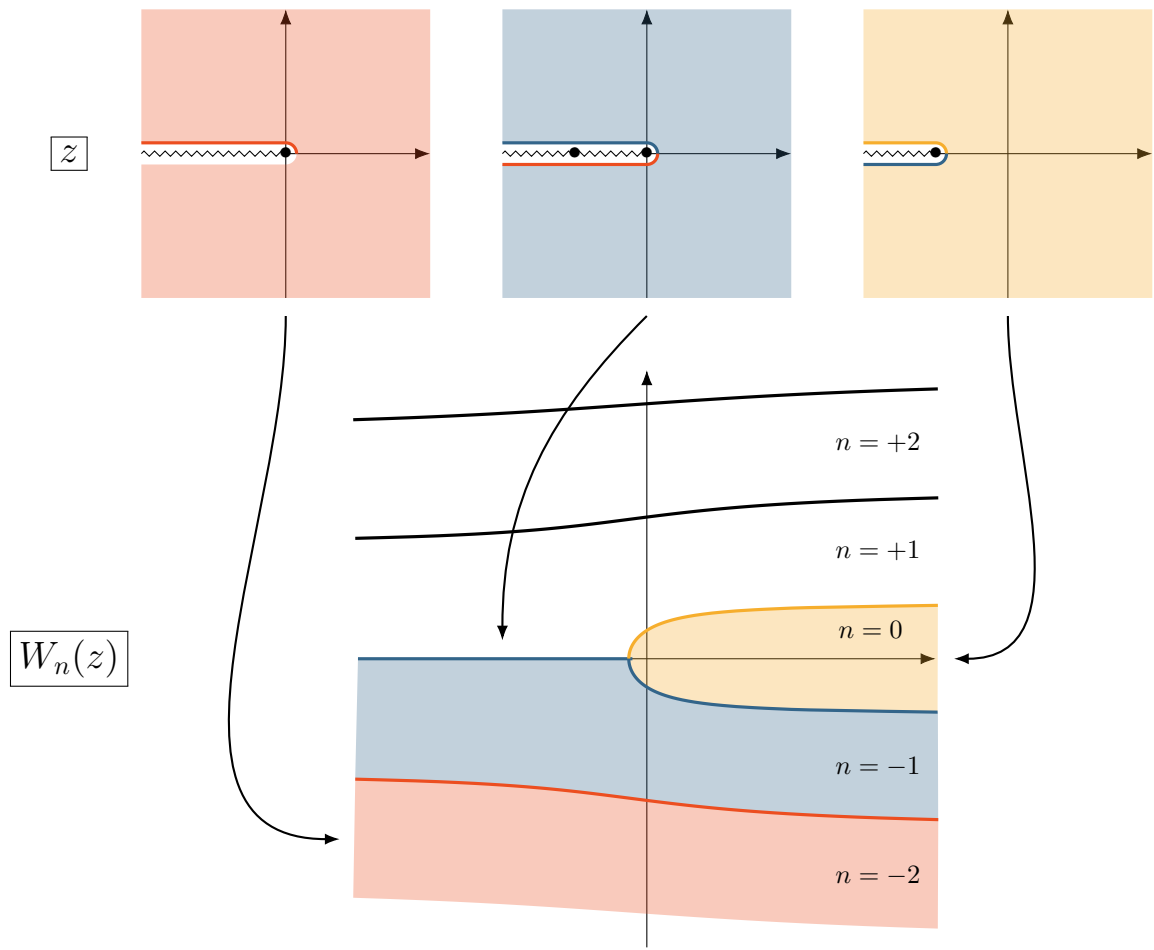
We now note that the argument of the Lambert $W$ terms in the asymptotic solution $\Phi_n(z; \zeta, x_0)$ \eqref{eq:FormalAsympScatteringSolution} are always $-2 \delta \zeta \ee^{-2 \delta (\zeta + u(x))}$.
Clearly, we should avoid $\zeta = 0$ since this makes all Lambert $W$ arguments zero for all $x \in \R$ which is a logarithmic branch point for all branches $n \neq 0$ and $W_0(0) = 0$.
Thus, for any $\zeta \in \C / \{ 0 \}$, the portion $-2\delta \zeta \ee^{-2\delta\zeta}$ is always nonzero.
The multiplying term $0 < \ee^{-2\delta u(x)} < 1$ for all $x \in \R$,
so changing $x \in \R$ just amounts to moving the argument of the Lambert $W$-function along a single complex ray, with a minimum and maximum possible magnitude since $u(x)$ is bounded,
\begin{equation} \label{eq:LambertWArgBounds}
    \left| -2 \delta \zeta \ee^{-2 \delta \zeta} \right| \ee^{-2 \delta u_\maxrm} \leq \left| -2 \delta \zeta \ee^{-2 \delta (\zeta + u(x))} \right| \leq \left| -2 \delta \zeta \ee^{-2 \delta \zeta} \right| \, .
\end{equation}
As $W_n(z)$ is continuous for any branch $n \in \Z$ along any ray bounded away from the origin, the only way the formal asymptotic solution $\Phi_n(z; \zeta, x_0)$ \eqref{eq:FormalAsympScatteringSolution} can be discontinuous is if the term in the denominator of the amplitude $A_n(x; \zeta)$ \eqref{eq:FormalAsymptoticAmplitude},
\begin{equation}
        1 + W_n(-2 \delta \zeta \ee^{-2 \delta (\zeta + u(x))}) \, ,
\end{equation}
is ever zero.
There are only two branches of the Lambert $W$-function which contain $-1$ in their ranges: $n = -1, 0$.
This motivates us to group the formal asymptotic solutions into two categories.
\begin{definition} \label{def:ScatteringCatagories}
    We refer to $\Phi_{-1}(z; \zeta, x_0)$ and $\Phi_0(z; \zeta, x_0)$ as the scattering solutions while $\Phi_n(x; \zeta, x_0)$ for any $n \in \Z$ such that $n \neq -1, 0$ will be referred to as the extraneous solutions.
\end{definition}
The names designate which solutions we expect to contribute to the WKB analysis and which ones should not.
The idea is that a blowup in an asymptotic solution for some $x_0 \in \R$ indicates a turning point where a matching procedure must be completed in order to construct a normalizable wavefunction.
If a solution does not blowup for any $x \in \R$, then this solution should not interact with any others through a matching procedure.
The conclusion drawn is that either the solutions corresponding to $n \neq -1,0$ are truly extraneous to the ILW scattering problem, or contribute beyond all orders in $\epsilon$.

In order for the amplitude term $A_n(u(x); \zeta)$ with $n = -1, 0$ to blow up at some point $x_0 \in \R$, by \eqref{eq:LambertWArgBounds} its required that
\begin{equation}
    -2 \delta \zeta \ee^{-2 \delta \zeta} \leq -e^{-1} \leq -2 \delta \zeta \ee^{-2 \delta \zeta} \ee^{-2 \delta u_\maxrm} \, .
\end{equation}
That is, $-2 \delta \zeta \ee^{-2 \delta \zeta} \in [-\ee^{2\delta u_\maxrm - 1}, -\ee^{-1}]$.
Miraculously, $-2 \delta \zeta \ee^{-2 \delta \zeta}$ is exactly the inverse of the Lambert $W$-function applied to $-2\delta\zeta$.
Thus, $-2\delta\zeta$ needs to be in any of the image sets $W_n([-\ee^{2\delta u_\maxrm - 1}, -\ee^{-1}])$.
To avoid redundancies, we can fix a branch range where $\zeta$-values are allowed to come from; for convenience, we'll fix the $n = -1$ branch range.
Then we can have blowups in the scattering solutions if and only if
\begin{equation}
    \zeta \in -\frac{1}{2\delta} W_{-1}([-\ee^{2\delta u_\maxrm - 1}, -\ee^{-1}]) \subset -\frac{1}{2\delta} W_{-1}([-\infty, -\ee^{-1}]) = \quadplus{\delta} \, ,
\end{equation}
which is in agreement with the description of \cite{KodamaAblowitzSatsuma_1982} for where eigenvalues of the ILW scattering equation should appear for the scattering variable $\zeta$.
Taking into account the maximum height $u_\maxrm$ and continuity of $u$, blowups must occur for some $x_0 \in \R$ if and only if $\zeta$ is between the quadratrix vertex and a maximal $\zeta$-value, i.e.
\begin{equation}
    \zeta \in \left( \frac{1}{2 \delta}, \zeta_\maxrm \right]_{\quadratrix{\delta}} \, ,
\end{equation}
where we are using the quadratrix interval notation \eqref{eq:QuadratrixIntervalNotationDef} and
\begin{equation}
    \zeta_\maxrm := -\frac{1}{2\delta} W_{-1}(-\ee^{2\delta u_\maxrm - 1}) \, .
\end{equation}
Blowup happens exactly when
\begin{equation}
    -2 \delta \zeta \ee^{-2\delta (\zeta + u(x_0))} = -\ee^{-1} \, .
    \label{eq:FirstTurningPointCondition}
\end{equation}
If $\zeta \in \quadplus{\delta}$, and we take a principal logarithm of \eqref{eq:FirstTurningPointCondition}, an equivalent condition is
\begin{equation}
    u(x_0) = E(\zeta) := -\zeta + \frac{1}{2\delta} \left( 1 + \log(2\delta\zeta) \right) \, .
    \label{eq:TurningPointFromELevel}
\end{equation}
When $u(x) < E(\zeta)$,
\begin{equation}
    -2 \delta \zeta \ee^{-2\delta (\zeta + u(x))} < -\ee^{-1} \, ,
    \label{eq:ExpCondition}
\end{equation}
and so the Lambert $W$ expressions with branches $n = -1, 0$ are on the quadratrix, specifically
\begin{align}
    &-\frac{1}{2 \delta} W_{-1}(-2 \delta \zeta \ee^{-2\delta (\zeta + u(x))}) \in \left( \frac{1}{2\delta}, \zeta \right)_{\mathcal{Q}_\delta} , \\
    &-\frac{1}{2 \delta} W_0(-2 \delta \zeta \ee^{-2\delta (\zeta + u(x))}) \in \left( \zeta^*,  \frac{1}{2\delta} \right)_{\mathcal{Q}_\delta} \, ,
\end{align}
Roughly, this means the formal asymptotic solutions $\Phi_{-1}(z; \zeta, x_0)$ and $\Phi_0(z; \zeta, x_0)$ \eqref{eq:FormalAsympScatteringSolution} have an exponentially growing or decaying magnitude due to the nonzero imaginary parts.

When $u(x) > E(\zeta)$,
\begin{equation}
    -2 \delta \zeta \ee^{-2\delta (\zeta + u(x))} > -\ee^{-1} \, ,
    \label{eq:OscilCondition}
\end{equation}
and so the Lambert $W$ expressions are real, specifically
\begin{align}
    &-\frac{1}{2 \delta} W_{-1}(-2 \delta \zeta \ee^{-2\delta (\zeta + u(x))}) \in \left( \frac{1}{2\delta}, +\infty \right) , \\
    &-\frac{1}{2 \delta} W_0(-2 \delta \zeta \ee^{-2\delta (\zeta + u(x))}) \in \left( 0, \frac{1}{2\delta} \right) \, .
\end{align}
Roughly, this means that the formal asymptotic solutions $\Phi_{-1}(z; \zeta, x_0)$ and $\Phi_0(z; \zeta, x_0)$ \eqref{eq:FormalAsympScatteringSolution} do not have any exponential behavior, just oscillatory.

In analogy with standard WKB analysis of the semiclassical time-independent Schr\"odinger equation, we make the following formal definitions:
\begin{definition}\label{def:TurnAllowedForbidded}
    Let $\zeta \in \quadplus{\delta}$. Points $x_0 \in \R$ where $u(x_0) = E(\zeta)$ are known as turning points. Regions of $x \in \R$ where $u(x) > E(\zeta)$ are known as classically allowed regions, and those where $u(x) < E(\zeta)$ are known as classically forbidden regions.
\end{definition}
For some specific examples of the values of $W_{n}(-2 \delta \zeta \ee^{-2\delta (\zeta + u(x))})$ in classically allowed and classically forbidden regions as well as at a turning point, see Figure \ref{fig:ExampleSolutionBehavior}.
From our work above, we see that not only do the solutions blow up at a turning point, but if $u(x)-E(\zeta)$ changes sign across the turning point, then the behavior of the scattering solutions changes so as to potentially build normalizable wavefunctions.

\begin{figure}
    \centering
    \begin{tikzpicture}[scale=0.8]
        \draw[-{Stealth[round]}, thick] (-7, 0) -- (7, 0);
        \draw[-{Stealth[round]}, thick] (-2, -1) -- (-2, 3);
        \node at (7.3, 0) {$x$};
        \node at (3.5, 3.1) {$u(x)$};

        \begin{scope}
            \clip (-7, -1) -- (7, -1) -- (7, 4) -- (-7, 4) -- cycle;
            \draw[very thick] (-7, 0.05) -- (-5, 0.1) to[out=4, in=-135] (0, 1.7) to[out=45, in=135] (5, 1.7) to[out=-45, in=-4] (10, 0.1);
        \end{scope}
        \draw[maize, very thick] (-7, 1.7) -- (7, 1.7);
        \node[maize] at (-7.5, 1.7) {$E(\zeta)$};
        
        \draw[thick, dashed] (-4.5, -1) -- (-4.5, 3);
        \draw[thick, dashed] (0, -1) -- (0, 3);
        \draw[thick, dashed] (4.5, -1) -- (4.5, 3);
        \node at (-7.1, -4) {$W_n$};
        \node at (-4.5, -1.5) {Classically Forbidden};
        \node at (0, -1.5) {Turning Point};
        \node at (4.5, -1.5) {Classically Allowed};

        \begin{scope}[shift={(-4.5, -4)}, scale=0.5]
            \draw[-{Latex[length=2mm]}] (0, -4) -- (0, 4);
            \draw[-{Latex[length=2mm]}] (-4, 0) -- (4, 0);
            \draw[scale=0.25, domain=-13.25:-14.96, smooth, variable=\y, very thick]  plot ({-\y*cos(\y r)/sin(\y r)}, {\y});
            \draw[scale=0.25, domain=-8.92:-6.675, smooth, variable=\y, very thick]  plot ({-\y*cos(\y r)/sin(\y r)}, {\y});
            \draw[very thick] (-0.2, 0) -- (-4, 0);
            \draw[scale=0.25, domain=-0.01:-2.96, smooth, variable=\y, very thick]  plot ({-\y*cos(\y r)/sin(\y r)}, {\y});
            \draw[scale=0.25, domain=2.96:0.01, smooth, variable=\y, very thick]  plot ({-\y*cos(\y r)/sin(\y r)}, {\y});
            \draw[scale=0.25, domain=8.92:6.675, smooth, variable=\y, very thick]  plot ({-\y*cos(\y r)/sin(\y r)}, {\y});
            \draw[scale=0.25, domain=13.25:14.96, smooth, variable=\y, very thick]  plot ({-\y*cos(\y r)/sin(\y r)}, {\y});
            \node at (0.75, 0.6) [circle, fill=seafoam, draw=none, inner sep=2]{};
            \node at (0.75, -0.6) [circle, fill=seafoam, draw=none, inner sep=2]{};
            \node at (0.2, 2) [circle, fill=fire, draw=none, inner sep=2]{};
            \node at (0.2, -2) [circle, fill=fire, draw=none, inner sep=2]{};
            \node at (-0.5, 3.5) [circle, fill=fire, draw=none, inner sep=2]{};
            \node at (-0.5, -3.5) [circle, fill=fire, draw=none, inner sep=2]{};
        \end{scope}

        \begin{scope}[shift={(0, -4)}, scale=0.5]
            \draw[-{Latex[length=2mm]}] (0, -4) -- (0, 4);
            \draw[-{Latex[length=2mm]}] (-4, 0) -- (4, 0);
            \draw[scale=0.25, domain=-13.25:-14.96, smooth, variable=\y, very thick]  plot ({-\y*cos(\y r)/sin(\y r)}, {\y});
            \draw[scale=0.25, domain=-8.92:-6.675, smooth, variable=\y, very thick]  plot ({-\y*cos(\y r)/sin(\y r)}, {\y});
            \draw[very thick] (-0.2, 0) -- (-4, 0);
            \draw[scale=0.25, domain=-0.01:-2.96, smooth, variable=\y, very thick]  plot ({-\y*cos(\y r)/sin(\y r)}, {\y});
            \draw[scale=0.25, domain=2.96:0.01, smooth, variable=\y, very thick]  plot ({-\y*cos(\y r)/sin(\y r)}, {\y});
            \draw[scale=0.25, domain=8.92:6.675, smooth, variable=\y, very thick]  plot ({-\y*cos(\y r)/sin(\y r)}, {\y});
            \draw[scale=0.25, domain=13.25:14.96, smooth, variable=\y, very thick]  plot ({-\y*cos(\y r)/sin(\y r)}, {\y});
            \node at (-0.25, 0) [circle, fill=seafoam, draw=none, inner sep=2]{};
            \node at (-0.6, 1.9) [circle, fill=fire, draw=none, inner sep=2]{};
            \node at (-0.6, -1.9) [circle, fill=fire, draw=none, inner sep=2]{};
            \node at (-1, 3.45) [circle, fill=fire, draw=none, inner sep=2]{};
            \node at (-1, -3.45) [circle, fill=fire, draw=none, inner sep=2]{};
        \end{scope}

        \begin{scope}[shift={(4.5, -4)}, scale=0.5]
            \draw[-{Latex[length=2mm]}] (0, -4) -- (0, 4);
            \draw[-{Latex[length=2mm]}] (-4, 0) -- (4, 0);
            \draw[scale=0.25, domain=-13.25:-14.96, smooth, variable=\y, very thick]  plot ({-\y*cos(\y r)/sin(\y r)}, {\y});
            \draw[scale=0.25, domain=-8.92:-6.675, smooth, variable=\y, very thick]  plot ({-\y*cos(\y r)/sin(\y r)}, {\y});
            \draw[very thick] (-0.2, 0) -- (-4, 0);
            \draw[scale=0.25, domain=-0.01:-2.96, smooth, variable=\y, very thick]  plot ({-\y*cos(\y r)/sin(\y r)}, {\y});
            \draw[scale=0.25, domain=2.96:0.01, smooth, variable=\y, very thick]  plot ({-\y*cos(\y r)/sin(\y r)}, {\y});
            \draw[scale=0.25, domain=8.92:6.675, smooth, variable=\y, very thick]  plot ({-\y*cos(\y r)/sin(\y r)}, {\y});
            \draw[scale=0.25, domain=13.25:14.96, smooth, variable=\y, very thick]  plot ({-\y*cos(\y r)/sin(\y r)}, {\y});
            \node at (-0.1, 0) [circle, fill=seafoam, draw=none, inner sep=2]{};
            \node at (-1.5, 0) [circle, fill=seafoam, draw=none, inner sep=2]{};
            \node at (-1.6, 1.9) [circle, fill=fire, draw=none, inner sep=2]{};
            \node at (-1.6, -1.9) [circle, fill=fire, draw=none, inner sep=2]{};
            \node at (-1.8, 3.45) [circle, fill=fire, draw=none, inner sep=2]{};
            \node at (-1.8, -3.45) [circle, fill=fire, draw=none, inner sep=2]{};
        \end{scope}
        
    \end{tikzpicture}
    \caption{An example potential $u(x)$ and level $E(\zeta)$ (top) with schematics (bottom) of the behaviors of $W_n(-2 \delta \zeta \ee^{-2\delta (\zeta + u(x))})$ at some chosen points in the classically forbidden region (left), at the turning point (middle) and in the classically allowed region (right). The values for the extraneous solutions ($n \neq -1, 0$) are shown in red and those for the scattering solutions ($n = -1, 0$) are shown in green.}
    \label{fig:ExampleSolutionBehavior}
\end{figure}
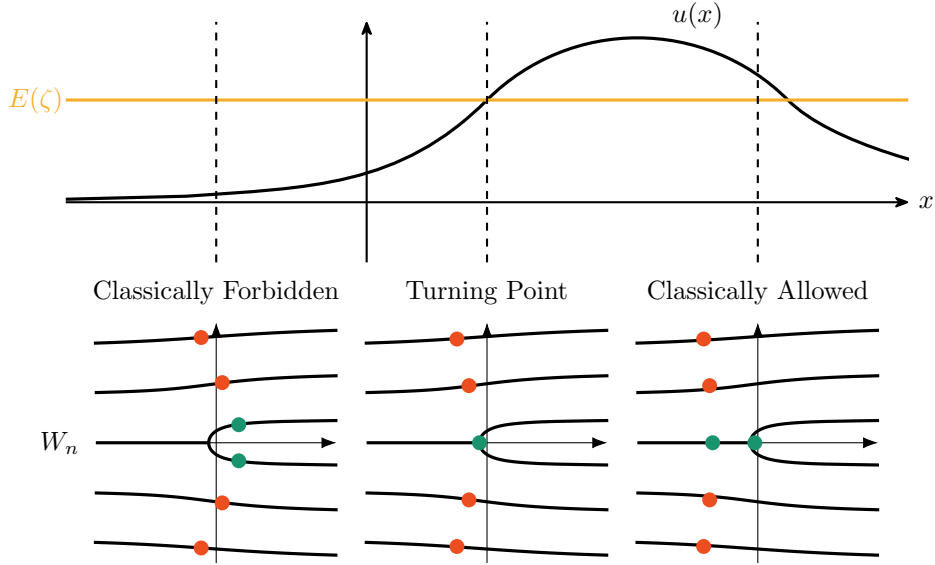

\subsection{The Model Turning Point Equation}
\label{sec:TurningPointEq}

The interpretation of the amplitude term of $\Phi_{-1}(z; \zeta, x_0)$ and $\Phi_0(z; \zeta, x_0)$ \eqref{eq:FormalAsympScatteringSolution} blowing up at a turning point is that the dominant balance which we considered in \eqref{eq:ScatEqSForm} is only appropriate for $u(x)$ away from the level $E(\zeta)$.
In the neighborhood of their equality, some terms in the scattering equation that were previously leading order, cancel, and a new dominant balance is formed.

We derive the model equation for a transition across a turning point by fixing $\zeta \in \quadratrix{\delta}$ and considering an $\epsilon$-shrinking interval in $x \in [a(\epsilon), b(\epsilon)] \subset \R$ where there is one and only one point $x_0 \in [a(\epsilon), b(\epsilon)]$ such that \eqref{eq:FirstTurningPointCondition}, or equivalently \eqref{eq:TurningPointFromELevel}, is satisfied. Then, by substituting 
\begin{equation}
    \Phi(z) = \Psi(z) \exp \left( \ii \frac{u(x_0)}{\epsilon} z \right)
\end{equation}
into the canonical scattering equation \eqref{eq:CannonFormScatteringEq}, we have after clearing exponential factors
\begin{equation}
    \ii \epsilon \Psi_x^+ + (u(x) - u(x_0)) \Psi^+ = \zeta \ee^{-2 \delta (\zeta + u(x_0))} \Psi^- = \frac{\ee^{-1}}{2 \delta} \Psi^- \, .
\end{equation}
If the rate at which the interval shrinks around the turning point is chosen appropriately, we expect that we may expand $u(x) - u(x_0) = u'(x_0) \ (x - x_0) + \bigO{(x - x_0)^2}$ and, as long as we have $u'(x_0) \neq 0$, then we may neglect the sub-leading terms in the Taylor expansion. Thus, we now have
\begin{equation}
    \ii \epsilon \Psi_x^+ + u'(x_0) \ (x - x_0) \Psi^+ = \frac{\ee^{-1}}{2 \delta} \Psi^- \, .
    \label{eq:TurningPointNeglectTerms}
\end{equation}
We may now want to remove as much dependence of the equation on the values of $u'(x_0), \epsilon, \delta$ and $x_0$ as we can via a shift and re-scaling of the independent variable.
For this, let $\psi(Z) := \Psi(z)$ where $Z := a (z - x_0)$ for some yet undetermined $a \in \R$ and $a \neq 0$.
Then the equation becomes for $X = \re{Z}$,
\begin{equation}
    \ii \epsilon a \ \psi_X^{+} + \frac{u'(x_0)}{a} X \psi^{+} = \frac{\ee^{-1}}{2 \delta} \psi^{-} \, .
    \label{eq:TurningPointExpansionNeglect}
\end{equation}
Additionally, due to the scaling $\psi$ is analytic on the strip $\strip{\delta \epsilon a}$. It is now apparent that $\epsilon$ can not be removed entirely from the equation: trying to equate any two terms, in an attempt to factor out an overall scale, will always leave the third term $\epsilon$-dependent. This simply indicates that the solutions to the scattering equation depend nontrivially on the semiclassical parameter $\epsilon$. Instead, we use this scaling to reduce the number of parameters in the equation from two, $u'(x_0)$ and $\delta$, to one. First, we will consider $u'(x_0) > 0$, then define $a$ and the new height parameter $h > 0$ by
\begin{equation}
    a = \frac{1}{2 \delta} \quad \text{and} \quad h = (2 \delta)^2 u'(x_0) \, .
\end{equation}
Then, after multiplying through by $2 \delta$ and bringing all of the terms to the same side, we have the model equation for the behavior near a turning point as
\begin{equation}
    0 = \ii \epsilon \psi_X^+ + h X \psi^+ - \ee^{-1} \psi^-
    \label{eq:ModelTuringPointEquation}
\end{equation}
where $\psi$ is analytic on the strip $\strip{\epsilon/2}$. 

For $u'(x) < 0$ we instead make the substitution $\psi(Z) := \Psi(-z^*)^*$ where $Z := (z - x_0)/2\delta$. Then
\begin{equation}
    \psi^\pm(X) = \psi\left( X \mp \ii \frac{\epsilon}{2} \right) = \Psi( -2 \delta X \mp \ii \delta)^* = \Psi^\pm( -2 \delta X)^* \, .
\end{equation}
Taking the conjugate of \eqref{eq:TurningPointExpansionNeglect}, using these boundary value identities and evaluating at $x = -2 \delta X + x_0$, we have
\begin{align}
    0 &= - \ii \epsilon \Psi^+_x(-2 \delta X)^* - 2 \delta u'(x_0) \, X \, \Psi^{+}(-2\delta X)^* - \frac{\ee^{-1}}{2 \delta} \Psi^{-}(-2\delta X)^* \\
    &= \ii \frac{\epsilon}{2 \delta} \psi^+_X(X) - 2 \delta u'(x_0) \, X \, \psi^{+}(X) - \frac{\ee^{-1}}{2 \delta} \psi^{-}(X) \, .
\end{align}
After multiplying through by $2 \delta$ and defining
\begin{equation}
    h = -(2 \delta)^2 u'(x_0) > 0
\end{equation}
we have the same form as \eqref{eq:ModelTuringPointEquation}. Thus, we can handle the analysis at both kinds of turning points in a uniform manner.

\subsubsection{Solving the Model Turning Point Equation by Laplace Transform}
\label{sec:SolvModelTurningPointEq}

Suppose there is an analytic function $\hat{\psi}: \C \to \C$ such that a solution to the model turning point equation \eqref{eq:ModelTuringPointEquation} could be formulated as
\begin{equation}
    \psi(Z) = \int_C \hat{\psi}(k) \ee^{\ii k Z / \epsilon} \dd k
\end{equation}
where $C$ is a contour on which the integral above converges rapidly enough for any $Z \in \strip{\epsilon/2}$ so that we may differentiate under the integral,
\begin{equation}
    \psi_X^+(X) = \int_C \hat{\psi}(k) \partial_X \ee^{\ii k (X / \epsilon - \ii/2)} \dd k = \int_C \ii \frac{k}{\epsilon} \ \hat{\psi}(k) \ee^{\ii k (X / \epsilon - \ii/2)} \dd k \, ;
\end{equation}
that we may integrate-by-parts with zero boundary term,
\begin{align}
    X \psi^+(X) &= \int_C \hat{\psi}(k) \left( \ii \frac{\epsilon}{2} \ee^{\ii k (X / \epsilon - \ii/2)} - \ii \epsilon \partial_k \ee^{\ii k (X / \epsilon - \ii/2)} \right) \dd k \\
    &= \int_C \ii \epsilon \left( \frac{1}{2} \hat{\psi}(k) + \hat{\psi}'(k) \right) \ee^{\ii k (X / \epsilon - \ii/2)} \dd k ;
\end{align}
and that evaluating at both the upper and lower boundaries are simultaneously well-defined
\begin{equation}
    \psi^-(X) = \int_C \hat{\psi}(k) \ee^{\ii k (X / \epsilon + \ii/2)} \, \dd k = \int_C \hat{\psi}(k) \, \ee^{-k} \, \ee^{\ii k (X / \epsilon - \ii/2)} \, \dd k \, .
\end{equation}
Plugging this form into the model turning point equation \eqref{eq:ModelTuringPointEquation},
\begin{equation}
    0 = \int_C \left( -k \hat{\psi}(k) + \ii h^2 \epsilon \left( \frac{1}{2} \hat{\psi}(k) + \hat{\psi}'(k) \right) - \ee^{-k-1} \hat{\psi}(k) \right) \ee^{\ii k (X / \epsilon - \ii/2)} \, \dd k \, ,
\end{equation}
and simply requiring the integrand be zero identically, we get a first-order differential equation for $\hat{\psi}$. This has the solution 
\begin{equation}
    \hat{\psi}(k) = A \exp \left( -\frac{\ii}{h \epsilon} \left[ \frac{k^2}{2} - \ee^{-k-1} \right] - \frac{k}{2} \right)
\end{equation}
for any $A \in \C$ independent of the variables $Z$ and $k$. This gives our solution as
\begin{equation} \label{eq:MTPEqLaplaceSolution}
    \psi(Z) = A \int_C \exp \left( -\frac{\ii}{h \epsilon} \left[ \frac{k^2}{2} - \ee^{-k-1} \right] + \ii k \left(\frac{Z}{\epsilon} +\frac{\ii}{2} \right) \right) \, \dd k \, .
\end{equation}
For large $k$ and fixed $Z \in \C$, the real part of the exponent is dominated by the behavior of the imaginary part of the function
\begin{equation}
    f(k) := \frac{k^2}{2} - \ee^{-k-1} \, .
\end{equation}
The landscape of $\im{f(k)} = \im{k} \re{k} + \ee^{-\re{k}-1} \sin(\im{k})$ determines where the contour $C$ may end in order to force rapid convergence of the integral.
In analogy with geographical topography, we'll refer to unbounded regions where $\im{f} \to -\infty$ as $k \to \infty$ ``valleys'' and unbounded regions where $\im{f} \to \infty$ as $k \to \infty$ ``mountains.''
Then finding a suitable integration contour $C$ becomes the problem of connecting two valleys.
In the west ($\re{k} < 0$), the exponential term $\ee^{-k-1}$ dominates $\im{f}$ and so we see there is an infinite alternating series of mountains and valleys spreading out in the north-south direction.
The valleys lie around the lines $\im{k} = \pi \nu$ where $\nu \in 2 \Z - 1/2$.
The mountains occur in between, around the lines $\im{k} = \pi \nu$ where $\nu \in 2 \Z + 1/2$.
In the east ($\re{k} > 0$), the quadratic term $k^2 / 2$ dominates, and we get just one large valley in the south around the argument $-\pi/4$ and one large mountain in the north around the argument $\pi/4$.
These features are easily seen in the landscape of Figure \ref{fig:ImPartFTopo}.

\subsubsection{A Basis of Solutions to the Model Turning Point Equation}

Its clear that there are infinitely many homotopy classes of integration contours $C$ which connect valleys in the landscape of $\im{f}$.
Towards constructing a basis of solutions, we define a basis of contours.
Let $C_\nu$ for $\nu \in 2\Z - 1/2$ be any piecewise-smooth, simple contour such that
\begin{equation}
    C_\nu \quad \text{begins at} \quad k = -\infty + \ii \pi \nu \quad \text{and ends at} \quad k = -\infty + \ii \pi \left( \nu + 2 \right) \, .
    \label{eq:DefNuBasisContours}
\end{equation}
These are contours which connect pairs of nearest-neighbor valleys in the west. Additionally, we define $C_{-\infty}$ to be any piecewise-smooth, simple contour such that
\begin{equation}
    C_{-\infty} \quad \text{begins at} \quad k = -\infty - \ii \frac{\pi}{2} \quad \text{and ends at} \quad k = \infty \ \ee^{-\ii \pi/4} \, .
    \label{eq:DefInftyBasisContour}
\end{equation}
An arbitrary contour $C$ which connects any valley to any other is Cauchy-equivalent to a finite sum of these basis contours $\{ C_\nu \}_{\nu \in \{-\infty\}\cup(2\Z - 1/2)}$
Hence, the solution \eqref{eq:MTPEqLaplaceSolution} associated to $C$ can be written as the sum of the fundamental solutions
\begin{align}
    \psi_\nu(Z; \epsilon, h) &:= \int_{C_\nu} \exp \left( -\frac{\ii}{h \epsilon} f(k) + \ii k \left( \frac{Z}{\epsilon} +\frac{\ii}{2} \right) \right) \dd k \quad \text{for all} \ \nu \in \{ -\infty \} \cup \left( 2\Z - \frac{1}{2} \right) \, .
    \label{eq:PsiFundamentalDef}
\end{align}

\begin{figure}
    \centering
    \includegraphics[width=0.6\linewidth]{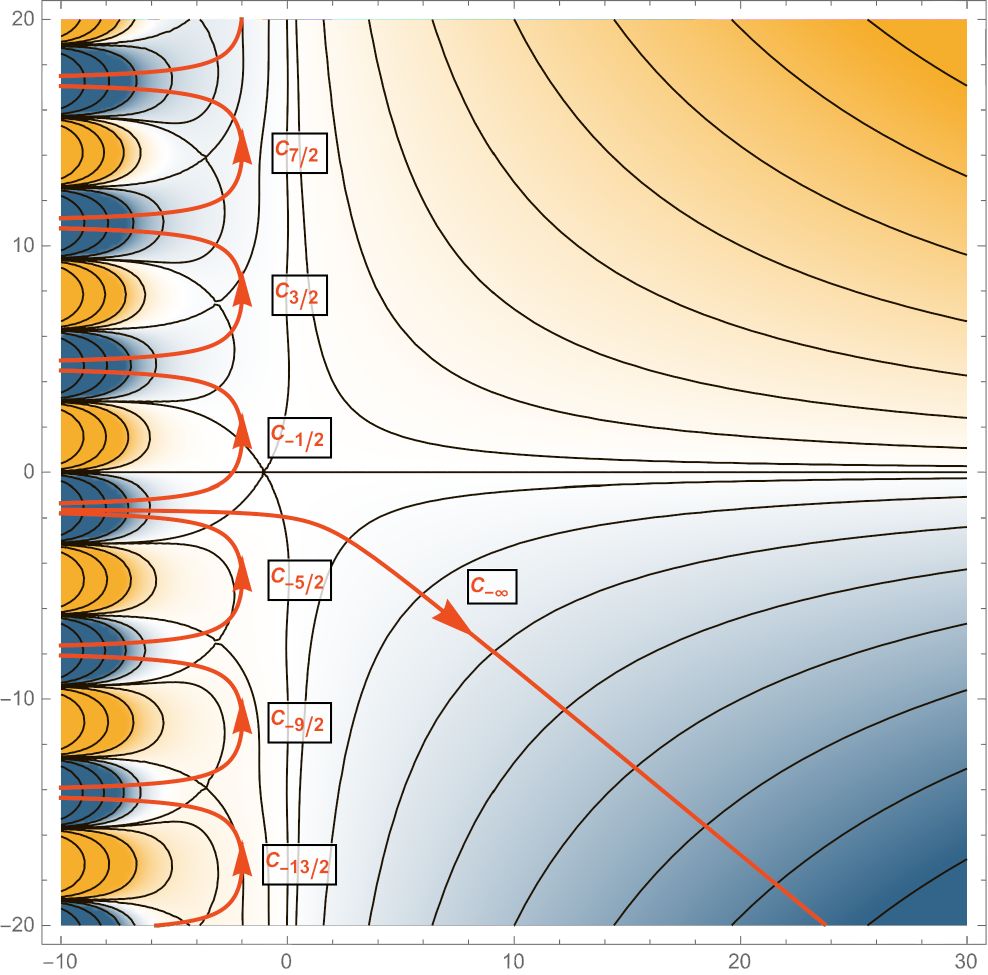}
    \caption{The landscape of $\im{f}$. Yellow shading corresponds to a positive value while blue shading corresponds to a negative value.
    Additionally, the fine black lines are some particular level contours.
    The alternating mountain and valley structure can be clearly seen on the left of the image with the valleys in blue and the mountains in yellow.
    On the right, we also see the northern mountain (yellow) and the southern valley (blue) in the first and fourth quadrants respectively.
    Some representative choices of basis contours $C_\nu$ as defined above are depicted in red with their labels.}
    \label{fig:ImPartFTopo}
\end{figure}

As $\epsilon \to 0^+$, the asymptotics of these fundamental solutions can be computed using the method of steepest descent.
For this reason, we need to understand the dependence of the saddle points of the function in the exponent of \eqref{eq:PsiFundamentalDef} on $Z$ and $h$.

\subsubsection{The Behavior of the Saddle Points}

Consider $Z = X + \ii \epsilon (\tilde{Y}-1/2) \in \strip{\epsilon/2}$ with $\tilde{Y} \in [0, 1]$, then the contribution of $\tilde{Y}$ to the exponent in the integrand of \eqref{eq:PsiFundamentalDef} is next-too-leading-order in $\epsilon$,
\begin{equation}
    \exp \left( -\frac{\ii}{h \epsilon} f(k) + \ii k \left( \frac{Z}{\epsilon} +\frac{\ii}{2} \right) \right) = \exp \left( -\frac{\ii}{h\epsilon} f(k; \tilde{X}) - k \tilde{Y} \right)
\end{equation}
where we've introduced 
\begin{equation}
    f(k; \tilde{X}) := f(k) - k \tilde{X}= \frac{k^2}{2} - \ee^{-k-1} - k\tilde{X} \quad \text{and} \quad \tilde{X} = hX
\end{equation}
to simplify notation. This means the saddle points that we are interested in are defined by
\begin{equation}
    0 = \partial_k f(k; \tilde{X}) = k + \ee^{-k - 1} - \tilde{X} \, ,
    \label{eq:SaddlePointDefintion}
\end{equation}
which can be expressed in terms of the Lambert $W$-function
\begin{equation}
    k_n(\tilde{X}) = \tilde{X} + W_n(-\ee^{-\tilde{X}-1}) \quad \text{for all} \quad n \in \Z \, .
\end{equation}
There are three qualitatively distinct configurations of these saddle points as $X \in \R$ is allowed to vary:
\begin{enumerate}
    \item For $\tilde{X} < 0$ large, we see from the asymptotics of the Lambert-W function \cite[Equation 4.13.1\_1]{DLMF}, 
    \begin{align}
        k_n(\tilde{X}) &= \tilde{X} + \big(-\tilde{X}-1 - \log \big| \tilde{X}+1 \big| + \ii (2n+1)\pi + o(1) \big) \\
        &= - \big( \log \big| \tilde{X}+1 \big| + 1 \big) + \ii (2n+1) \pi + o(1) \, .
    \end{align}
    That is, they are about evenly spaced in the imaginary direction and with real part going like a negative log of $|\tilde{X}|$. The saddle points $k_n(\tilde{X})$ for $n \in \Z$ are all simple.
    
    \item As $\tilde{X}$ increases towards zero, the saddle points below the real-line move upwards and the saddle points above the real-line move down. All of the saddle points maintain their distance from each other except for $k_0(\tilde{X})$ and $k_{-1}(\tilde{X})$, which coincide at $\tilde{X} = 0$. As $\tilde{X}$ increases past zero, the saddle points $k_0(\tilde{X})$ and $k_{-1}(\tilde{X})$ move apart again but now on the real line, in the character of a fold coalescing of saddle points.

    \item For $\tilde{X} > 0$ large, we note the asymptotics of the Lambert-W function are distinct between the 0-branch and the others.
    \begin{equation}
        k_0(\tilde{X}) = \tilde{X} + \bigO{\ee^{-\tilde{X}-1}}
    \end{equation}
    while
    \begin{align}
        k_n(\tilde{X}) &= \tilde{X} + \big(-\tilde{X}-1 - \log \big| \tilde{X}+1 \big| + \ii 2n_+\pi + o(1) \big) \quad \text{for all} \quad n \neq 0\\
        &= -\big( \log \big| \tilde{X}+1 \big| + 1 \big) + \ii 2n_+\pi + o(1) \, .
    \end{align}
    with $n_+ = n$ if $n > 0$ and $n_+ = n+1$ if $n < 0$. So there is again a vertical tower of about evenly spaced saddle points in the imaginary direction and with real part going like a negative $\log |\tilde{X}|$. However, this tower of saddle points is now centered on the real line instead of symmetric about it and there is one saddle point which has split off and is at about $\tilde{X}$ on the real-line. Once again, the saddle points $k_n(\tilde{X})$ for $n \in \Z$ are all simple.
\end{enumerate}

\subsection{Matching at the Turning Points}
\label{sec:WKB}

During the process of deriving the model turning point equation, we formally neglected the higher order terms $\bigO{(x - x_0)^2}$ from the Taylor expansion of the potential in \eqref{eq:TurningPointNeglectTerms}.
This is reasonable only if we are working in a shrinking $x$ window around a turning point.
By the transformation $\tilde{X} = h (x - x_0) / 2 \delta$, this would be a shrinking window around $\tilde{X} = 0$, exactly where the $\im{f}$ landscape has two colliding saddle points.
A careful steepest decent analysis (see \cite[Chapter 2]{Mitchell_2024}) in the style of C. Chester, B. Friedman and F. Ursell \cite{ChesterFriedmanUrsell_1957} of the definitions of the fundamental solutions \eqref{eq:PsiFundamentalDef} gives the following result.\footnote{The careful reader may note that the overall multiplicative coefficients presented here differ slightly from those presented in \cite{Mitchell_2024}.
Unfortunately, it seems these extra multiplicative factors were neglected in the cited work and have been corrected in this work.}
\begin{theorem}
    \label{thm:AsympTurnPointEqBasis}
    Let $Z = \epsilon^\alpha \chi + \ii \epsilon (\tilde{Y} -1/2)$ with $1/2 < \alpha < 2/3$ and the shorthand $W_n := W_n(-\ee^{-1})$. Then for $\chi$ in compact subsets and $\nu \in (2\Z - 1/2)\cup\{ -\infty \}$ in the following cases, we have the approximations with uniform error bounds:
    \begin{enumerate}
        \item $\nu = -\infty$,
        \begin{align}
            \psi_{-\infty}(Z; \epsilon, h) &= 2 \pi (2h\epsilon)^{1/3} \exp \left( \frac{\ii}{2h\epsilon} -\frac{\ii}{\epsilon^{1-\alpha}} \chi + \tilde{Y} \right) \nonumber \\
            &\Quad[10] \times \Bigg[ \Ai \left( - \frac{(2h)^{1/3}}{\epsilon^{2/3-\alpha}} \chi \big( 1 + \bigO{\epsilon^{\alpha}} \big) \right) + \bigO{\epsilon^{1-3\alpha/2}} \Bigg] \, ;
            \label{eq:PsiInftyFullApprox}
        \end{align}
        
        \item $\nu \in 2\Z - 1/2$ with $\nu \leq -5/2$ and let the integer $n = (\nu-3/2)/2 \leq -2$,
        \begin{align}
            \psi_{\nu}(Z; \epsilon, h) &= \left( \frac{2 \pi h \epsilon}{1 + W_n} \right)^{1/2} \exp \left( \ii \frac{1 - (1 + W_n)^2}{2h\epsilon} + \ii \frac{W_n}{\epsilon^{1-\alpha}} \chi - W_n \tilde{Y} - \ii \frac{\pi}{4} \right) \left[ 1 + \bigO{\epsilon^{\alpha}} \right] \, ;
            \label{eq:PsiNuNegFinalApprox}
        \end{align}

        \item $\nu = -1/2$,
        \begin{align}
            \psi_{-1/2}(Z; \epsilon, h) &= 2 \pi (2h\epsilon)^{1/3} \exp \left( \frac{\ii}{2h\epsilon} + \ii \frac{\pi}{3} -\frac{\ii}{\epsilon^{1-\alpha}} \chi + \tilde{Y} \right) \nonumber \\
            &\Quad[10] \times \Bigg[ \Ai \left( \frac{(2h)^{1/3}\ee^{\ii \pi/3}}{\epsilon^{2/3-\alpha}} \chi \big( 1 + \bigO{\epsilon^{\alpha}} \big) \right) + \bigO{\epsilon^{1-3\alpha/2}} \Bigg] \, ;
            \label{eq:PsiOneHalfFullApprox}
        \end{align}

        \item $\nu \in 2\Z - 1/2$ with $\nu \geq 3/2$ and let the integer $n = (\nu+1/2)/2 \geq 1$,
        \begin{align}
            \psi_{\nu}(Z; \epsilon, h) &= \left( \frac{2 \pi h \epsilon}{1 + W_n} \right)^{1/2} \exp \left( \ii \frac{1 - (1 + W_n)^2}{2h\epsilon} + \ii \frac{W_n}{\epsilon^{1-\alpha}} \chi - W_n \tilde{Y} + \ii \frac{3\pi}{4} \right) \left[ 1 + \bigO{\epsilon^{\alpha}} \right] \, .
            \label{eq:PsiNuPosFinalApprox}
        \end{align}

    \end{enumerate}
\end{theorem}

This confirms the appearance of the Airy functions, just like for the model turning point equation for the time-independent Schr\"odinger equation, which was asserted in \cite{MinzoniMiloh_1994}. 

We now heuristically justify the construction of normalizable eigenfunctions $\psi$ of the ILW scattering equation \eqref{eq:CannonFormScatteringEq} by matching the formal asymptotic solutions \eqref{eq:FormalAsympScatteringSolution} across turning points by way of the rigorous asymptotics of our basis of solutions to the model turning point equation \eqref{eq:ModelTuringPointEquation}. 
Assuming the potential $u_0$ is an admissible initial condition satisfying Definition \ref{def:AdmisInitCond}, then for a particular level $\zeta \in (1/2\delta, \zeta_\maxrm)_{\mathcal{Q}_\delta}$ there are two distinct turning points $x_{\pm}(\zeta)$ satisfying \eqref{eq:TurningPointFromELevel} and ordered $x_{-}(\zeta) < x_{+}(\zeta)$.
These split $\R$ into three disjoint regions, the ``left classically forbidden'' region $F_{-}(\zeta)$, the ``classically allowed'' region $A(\zeta)$ and the ``right classically forbidden'' region $F_{+}(\zeta)$ define by
\begin{align}
    F_{-}(\zeta) &:= \big( -\infty, x_-(\zeta) \big) \, , \\
    A(\zeta) &:= \big( x_-(\zeta), x_+(\zeta) \big) \, , \\
    F_{+}(\zeta) &:= \big( x_+(\zeta), + \infty \big) \, .
\end{align}
In each of these three regions, we consider an infinite linear combination of asymptotic solutions \eqref{eq:FormalAsympScatteringSolution}.
\begin{align}
    \label{eq:LeftForbidRegSol}
    \Phi_{F-}(z, \zeta) &:= \sum_{n \in \Z} f^{(n)}_{-} \Phi_n(z; \zeta, x_-(\zeta)) \text{ for } x \in F_{-}(\zeta) , \\
    \label{eq:AllowedRegSol}
    \Phi_{A}(z, \zeta) &:= \sum_{n \in \Z} a^{(n)}_{-} \Phi_n(z; \zeta, x_-(\zeta)) = \sum_{n \in \Z} a^{(n)}_{+} \Phi_n(z; \zeta, x_+(\zeta)) \text{ for } x \in A(\zeta) , \\
    \label{eq:RightForbidRegSol}
    \Phi_{F+}(z, \zeta) &:= \sum_{n \in \Z} f^{(n)}_{+} \Phi_n(z; \zeta, x_+(\zeta)) \text{ for } x \in F_{+}(\zeta) \, ,
\end{align}
with $\{ f^{(n)}_{\pm}, a^{(n)}_{\pm} \}_{n \in Z} \subset \C$ are some possibly $\epsilon$-dependent but definitely $z$-independent constants.
Additionally, we wrote the same expansion in the classically allowed region two different ways, one integrating in from the left turning point and one from the right turning point which are equivalent when the coefficients satisfy the relationship
\begin{equation}
    \frac{a^{(n)}_{+}}{a^{(n)}_{-}} = \exp \left( \frac{\ii}{\epsilon} \int_{x_-(\zeta)}^{x_+(\zeta)} G^{-1}_n(u(x'); \zeta) \, \dd x' \right) \, .\label{eq:ApmRelationship}
\end{equation}
which can be required as this is a $z$-independent constant.

Now, we also introduce the $\epsilon$-shrinking transition regions around the turning points
\begin{equation}
    T_{\pm}(\zeta, \epsilon) = \left( x_{\pm}(\zeta) - \epsilon^\alpha 2\delta M^\chi_{\pm}, \  x_{\pm}(\zeta) + \epsilon^\alpha 2\delta M^\chi_{\pm} \right) \, ,
\end{equation}
for some $M^\chi_{\pm} > 0$ and $1/2 < \alpha < 2/3$.
We write $x \in T_{\pm}(\zeta, \epsilon)$ as $x = x_{\pm}(\zeta) \mp \epsilon^\alpha 2\delta \chi_{\pm}$ where $|\chi_{\pm}| \leq M^\chi_{\pm}$. Then we write the asymptotic solution in the transition regions as
\begin{equation}
    \label{eq:TransRegionSols}
    \Phi_{T\pm}(z; \zeta) = \exp \left( \ii \frac{u(x_{\pm}(\zeta))}{\epsilon} z \right) \sum_{\nu \in (2\Z-1/2)\cup\{-\infty\}} t_{\pm}^{(\nu)} \psi_{\nu}^{(*)_+}(Z_{\pm}; \epsilon, h_{\pm}) \quad \text{ for } x \in T_{\pm}(\zeta) 
\end{equation}
where $\{ t_{\pm}^{(\nu)} \}_{\nu \in (2\Z-1/2)\cup\{-\infty\}} \subset \C$ are, again, some $z$-independent constants, $Z_{\pm} = \pm\big( x - x_{\pm}(\zeta) \pm \ii y \big) /2\delta = \mp \epsilon^\alpha \chi_{\mp} + \ii y/2\delta$, $h_{\pm} = \mp 2 \delta u_0'\big( x_{\pm}(\zeta) \big)$ and the $\psi^{(*)_+}$ notation is meant to indicate a conjugate should only be present for the right turning point expansion.
With these coordinate definitions, $\chi_{\pm} > 0$ is always in the classically allowed region while $\chi_{\pm} < 0$ is always in one of the classically forbidden regions. 

First we begin in the left classical forbidden region, $x \in F_{-}(\zeta)$.
Here all of the values $G^{-1}_n(u_0(x); \zeta)$ from its definition \eqref{eq:GInvDef} have nonzero imaginary parts.
Moreover, as $x \to -\infty$ in $F_{-}(\zeta)$, $G^{-1}_n(u_0(x); \zeta) \to G^{-1}_n(0; \zeta)$.
For $n = -1$ this returns $G^{-1}_{-1}(0; \zeta) = \zeta$, and for $n = 0$, $G^{-1}_0(0; \zeta) = \zeta^*$.
In the other cases of $n$, we have for $n \leq -2$, $\im{G_n(0; \zeta)} > \im{\zeta}$ and for $n \geq 1$, $\im{G_n(0; \zeta)} < -\im{\zeta}$.
This means the integral in the exponent of the asymptotic solutions $\Phi_n(x; \zeta, x_-(\zeta))$ \eqref{eq:FormalAsympScatteringSolution} has a real part that increases without bound as $x \to - \infty$ if $n \geq -1$, that is, the magnitude of $\Phi_n(x; \zeta, x_-(\zeta))$ will increase without bound, and the overall $\Phi_{F-}(x, \zeta)$ will not be normalizable.
Thus, we immediately conclude that all 
\begin{equation}
    f^{(n)}_{-} = 0 \text{ for } n \geq 0 \, .
    \label{eq:GrowingExpSolAreKilled}
\end{equation}
Conversely, if $n \leq -1$ in the aforementioned integral, the real part will decrease without bound and there will be no problem integrating the norm-squared of the solution as $x \to -\infty$ (of course, subject to convergent coefficients $f^{(n)}_{-}$).

Now we wish to match with the solution expansion in the transition region around the turning point $x_-(\zeta)$.
We will expand the constituent parts of the asymptotic solutions $\Phi_n(x; \zeta, x_-(\zeta))$ \eqref{eq:FormalAsympScatteringSolution} with the change of variables $x = x_-(\zeta) + \epsilon^\alpha 2\delta \chi_-$, $y = 2 \delta \epsilon(\tilde{Y}-1/2)$.
Because of the different behavior for the solutions at a turning point, we will have to handle these expansions in two cases:
\begin{enumerate}
    \item \textbf{Extraneous solutions ($\boldsymbol{n \in \Z}$ but $\boldsymbol{n \neq -1, 0}$).} $G^{-1}_n(u_0(x); \zeta)$ is differentiable at the turning point so we have the expansion
    \begin{align}
        \Phi_n(z; \zeta, x_-(\zeta)) &= \exp \Bigg( \ii \left( 2 \delta u_0(x_-(\zeta)) + W_n(-\ee^{-1}) \right) \left( \frac{\chi_-}{\epsilon^{1-\alpha}} + \ii \tilde{Y} -\frac{\ii}{2} \right)  \Bigg) \nonumber \\
        &\Quad[6] \times \left[ \left( \frac{-2 \delta W_n(-\ee^{-1})}{1 + W_n(-\ee^{-1})} \right)^{1/2} + \bigO{\epsilon^{2\alpha-1}} \right] \, .
        \label{eq:FormalAsympSolExpanAllBut-10}
    \end{align}
    Because these solutions are regular across the turning point, this expansion is valid in the entire transition region $T_{-}(\zeta)$.

    \item \textbf{Scattering solutions ($\boldsymbol{n = -1, 0}$).} The difficulty in this case is handling the fact that the Lambert $W$ expression is not differentiable at the turning point; it vanishes like a square-root there (i.e. these solutions behave like the WKB solutions to the time-independent Schr\"odinger equation).
    Thus, we will have different asymptotics depending on which side of the transition region we're expanding on.
    First, let us look at the case for $\chi_- < 0$. Expanding the argument of the Lambert $W$-function,
    \begin{align}
        -2\delta \zeta \ee^{-2\delta(\zeta + u_0(x))} &= -2\delta \zeta \exp \left( -2\delta \left[ \zeta + u_0(x_-(\zeta)) + u_0'(x_-(\zeta)) 2 \delta \epsilon^\alpha \chi_- + \bigO{\epsilon^{2\alpha}} \right] \right) \nonumber \\
        &= -\exp \left( -1 - h_- \epsilon^\alpha \chi_- + \bigO{\epsilon^{2\alpha}} \right) \, .
    \end{align}
    We make use of the series expansion of the Lambert $W$-function \cite[Equation 4.13.6]{DLMF}, where we pull a negative sign out of the principle square root as a factor of $\ii$ and plug the result into the formula for the asymptotic solution \eqref{eq:FormalAsympScatteringSolution}:
    \begin{align}
        \Phi_{-1}(z; \zeta, x_-(\zeta)) &= \exp \Bigg( \ii \left( 2 \delta u_0(x_-(\zeta)) + 1 \right) \left( \frac{\chi_-}{\epsilon^{1-\alpha}} + \ii \tilde{Y} -\frac{\ii}{2} \right) - \frac{h_-^{1/2}}{3 \epsilon^{1-3\alpha/2}} (- 2\chi_-)^{3/2} \Bigg) \nonumber \\
        &\Quad[4] \times \frac{(2 \delta)^{1/2} \ee^{\ii \pi/ 4}}{(2 h_-)^{1/4} \epsilon^{\alpha/4} \left( - \chi_- \right)^{1/4}} \left[ 1 + \bigO{\epsilon^{2\alpha-1}} \right] \, , \label{eq:FormalAsympSolExpanForbid-1}\\
        \Phi_{0}(z; \zeta, x_-(\zeta)) &= \exp \Bigg( \ii \left( 2 \delta u_0(x_-(\zeta)) + 1 \right) \left( \frac{\chi_-}{\epsilon^{1-\alpha}} + \ii \tilde{Y} -\frac{\ii}{2} \right) + \frac{h_-^{1/2}}{3 \epsilon^{1-3\alpha/2}} (-2\chi_-)^{3/2} \Bigg) \nonumber \\
        &\Quad[4] \times \frac{(2 \delta)^{1/2} \ee^{-\ii \pi/ 4}}{(2 h_-)^{1/4} \epsilon^{\alpha/4} \left( -\chi_- \right)^{1/4}} \left[ 1 + \bigO{\epsilon^{2\alpha-1}} \right] \, . \label{eq:FormalAsympSolExpanForbid0}
    \end{align}
    For $\chi_- > 0$, we do not have to pull a negative out of the principle square root in \cite[Equation 4.13.6]{DLMF}, thus we don't have an extra factor of $\ii$ and corresponding algebra reveals
    \begin{align}
        \Phi_{-1}(z; \zeta, x_-(\zeta)) &= \exp \Bigg( \ii \left( 2 \delta u_0(x_-(\zeta)) + 1 \right) \left( \frac{\chi_-}{\epsilon^{1-\alpha}} + \ii \tilde{Y} -\frac{\ii}{2} \right) + \ii \frac{h_-^{1/2}}{3 \epsilon^{1-3\alpha/2}} ( 2\chi_-)^{3/2} \Bigg) \nonumber \\
        &\Quad[4] \times \frac{\ii (2 \delta)^{1/2}}{(2 h_-)^{1/4} \epsilon^{\alpha/4} \left( \chi_- \right)^{1/4}} \left[ 1 + \bigO{\epsilon^{2\alpha-1}} \right] \, , \label{eq:FormalAsympSolExpanAllow-1} \\
        \Phi_{0}(z; \zeta, x_-(\zeta)) &= \exp \Bigg( \ii \left( 2 \delta u_0(x_-(\zeta)) + 1 \right) \left( \frac{\chi_-}{\epsilon^{1-\alpha}} + \ii \tilde{Y} -\frac{\ii}{2} \right) - \ii \frac{h_-^{1/2}}{3 \epsilon^{1-3\alpha/2}} ( 2\chi_-)^{3/2} \Bigg) \nonumber \\
        &\Quad[4] \times \frac{(2 \delta)^{1/2}}{(2 h_-)^{1/4} \epsilon^{\alpha/4} \left( \chi_- \right)^{1/4}} \left[ 1 + \bigO{\epsilon^{2\alpha-1}} \right] \, . \label{eq:FormalAsympSolExpanAllow0}
    \end{align}

\end{enumerate}

Now we are finally ready to start matching in the left transition region.
The asymptotics of $\Phi_{F-}$ \eqref{eq:LeftForbidRegSol} near the turning point $x_{-}(\zeta)$ are computed using \eqref{eq:FormalAsympSolExpanAllBut-10}, \eqref{eq:FormalAsympSolExpanForbid-1} and \eqref{eq:FormalAsympSolExpanForbid0};
likewise, those of $\Phi_{A}$ \eqref{eq:AllowedRegSol} are computed using \eqref{eq:FormalAsympSolExpanAllBut-10}, \eqref{eq:FormalAsympSolExpanAllow-1}, and \eqref{eq:FormalAsympSolExpanAllow0}; and finally, those of $\Phi_{T-}$ \eqref{eq:TransRegionSols} are computed using the results of Theorem \ref{thm:AsympTurnPointEqBasis} along with the standard asymptotics of the Airy function \cite[Section 9.7(ii)]{DLMF} and the fact that under our change of variables to the model turning point equation 
\begin{equation}
    \exp \left( \ii \frac{u_0(x_-(\zeta))}{\epsilon} z \right) = \exp \left( \ii 2 \delta u_0(x_-(\zeta)) \left( \frac{\chi_-}{\epsilon^{1-\alpha}} + \ii \tilde{Y} -\frac{\ii}{2} + \frac{x_-(\zeta)}{2 \delta \epsilon} \right) \right) \, .
\end{equation}
These asymptotics are linear combinations of exponentials in $x$ with differing exponent coefficients.
The heuristic principal for matching solutions on the two sides of the transition region is that exponentials of the same exponential coefficient in the two expansions to be matched should have the same linear coefficient in their respective sums.
From this, we conclude that, to leading order in $\epsilon$, we have
\begin{align}
    f^{(n)}_{-} &= C_n(\epsilon) t_{-}^{(2n+3/2)} = a^{(n)}_{-} \text{  for  } n \leq -2 \, , \label{eq:MatchingFnNegTA} \\
    f_{-}^{(-1)} &= C_{-1}(\epsN) t_{-}^{(-\infty)} = \ii a^{(-1)}_{-} \, , \label{eq:MatchingF-1TA} \\
    0 = f_{-}^{(0)} &= C_{0}(\epsilon) t^{(-1/2)}_{-} = \ii a^{(-1)}_{-} + \ii a^{(0)}_{-} \, , \label{eq:MatchingF0TA} \\
    0 = f^{-}_{n} &= C_n(\epsilon) t^{(2n-1/2)}_{-} = a^{n}_{-} \text{  for  } n \geq 1 \, , \label{eq:MatchingFnPosTA}
\end{align}
For some explicit $\epsilon$-dependent, but $z$-independent, coefficients $C_{n}(\epsilon)$ for $n \in \Z$.

For matching at the right endpoint, we first note that in the classically forbidden region
\begin{align}
    \Phi_n(z; \zeta, x_0) &= \Phi_{-1-n}^*(-z^*; \zeta, -x_0) \text{ for } n \in \Z \, ,
\end{align}
where $u(x) = u_0(x)$ is taken in the asymptotic solutions on the left and $u(x) = u_0(-x)$ is taken in the asymptotic solutions on the right.
This is due to the Lambert $W$-function terms coming in conjugate pairs identified by $n \leftrightarrow -1-n$.
Likewise, in the classically allowed region,
\begin{align}
    \Phi_n(z; \zeta, x_0) &= \phantom{-}\Phi_{-1-n}^*(-z^*; \zeta, -x_0) \text{ for } n \neq -1, 0 \, , \\
    \Phi_{-1}(z; \zeta, x_0) &= -\Phi_{-1}^*(-z^*; \zeta, -x_0) \, , \\
    \Phi_0(z; \zeta, x_0) &= \phantom{-}\Phi_{0}^*(-z^*; \zeta, -x_0) \, ,
\end{align}
again where $u(x) = u_0(x)$ is taken in the asymptotic solutions on the left and $u(x) = u_0(-x)$ is taken in the asymptotic solutions on the right.
That is, the two real Lambert $W$-function values for $n = -1, 0$ do not swap under conjugation and, additionally, the amplitude term of $\Phi_{-1}$ \eqref{eq:FormalAsymptoticAmplitude} is purely imaginary while that of $\Phi_{0}$ is positive.
If we construct a new tilded system of approximate solutions to the potential $\tilde{u}_0(x) = u_0(-x)$, with new turning points $\tilde{x}_\pm(\zeta) = -x_\mp(\zeta)$ then, using the expansions \eqref{eq:LeftForbidRegSol}, \eqref{eq:AllowedRegSol}, \eqref{eq:RightForbidRegSol}, and \eqref{eq:TransRegionSols}, the new tilded coefficients can be related to those of the previous coefficients using
\begin{align}
    \tilde{\Phi}_{F\pm}(z, \zeta) = \Phi_{F\mp}^*(-z^*, \zeta) \quad \implies &\quad \tilde{f}_{\pm}^{(n)} = (f_{\mp}^{(-1-n)})^* \text{ for } n \in \Z \, , \\
    \tilde{\Phi}_{A}(z, \zeta) = \Phi_{A}^*(-z^*, \zeta)  \quad \implies &\quad \left\{ \begin{array}{ll} 
        &\tilde{a}_{\pm}^{(n)} = (a_{\mp}^{(-1-n)})^* \text{ for } n \neq -1, 0 \\
        & \tilde{a}_{\pm}^{(-1)} = -(a_{\mp}^{(-1)})^* \\
        & \tilde{a}_{\pm}^{(0)} = (a_{\mp}^{(0)})^*
    \end{array} \right. \, , \\
    \tilde{\Phi}_{T\pm}(z, \zeta) = \Phi_{T\mp}^*(-z^*, \zeta) \quad \implies &\quad \tilde{t}_{\pm}^{(\nu)} = (t_{\mp}^{(\nu)})^* \text{ for } \nu \in 2\Z - 1/2 \cup \{ \infty \} \, .
\end{align}
Matching at the left turning point of the tilded system $\tilde{x}_-(\zeta)$ yields conditions \eqref{eq:GrowingExpSolAreKilled}, \eqref{eq:MatchingFnNegTA}, \eqref{eq:MatchingF-1TA}, \eqref{eq:MatchingF0TA} and \eqref{eq:MatchingFnPosTA} with tilded coefficients.
From the above identifications, these become new matching conditions on the original coefficients at the right turning point $x_+(\zeta)$.
Accounting for the relation between right and left coefficients in the classically allowed region \eqref{eq:ApmRelationship}, we have the full system of coefficients matched:
\begin{align}
    f_{-}^{(n)} = a_{-}^{(n)} = a_{+}^{(n)} = f_{+}^{(-1-n)} &= 0 \quad \text{for} \quad n \leq -2 \, , \\
    f_{+}^{(-1)} &= 0 \, ,
\end{align}
\begin{equation}
    f_{-}^{(-1)} = \left\{ \begin{array}{ll|l}
        \phantom{-}\ii a_{-}^{(-1)} &= \displaystyle \phantom{-}\ii \exp \left( -\frac{\ii}{\epsilon} \int_{x_-(\zeta)}^{x_+(\zeta)} G^{-1}_{-1}(u(x'); \zeta) \, \dd x' \right) a_{+}^{(-1)} & \ii a_{+}^{(-1)} \\
        - \ii a_{-}^{(0)} &= \displaystyle - \ii \exp \left( -\frac{\ii}{\epsilon} \int_{x_-(\zeta)}^{x_+(\zeta)} G^{-1}_0(u(x'); \zeta) \, \dd x' \right) a_{+}^{(0)} & \ii a_{+}^{(0)}
    \end{array} \right\} = f_{+}^{(0)} \, ,
    \label{eq:MatchedScatteringCoefs}
\end{equation}
\begin{align}
    0 &= f_{-}^{(0)} \, , \\
    0 &= f_{-}^{(n)} = a_{-}^{(n)} = a_{+}^{(n)} = f_{+}^{(-1-n)} \quad \text{for} \quad n \geq 1 \, .
\end{align}
Here, we have visually represented the matching conditions as an infinite hierarchy.
Those that are set to zero are exactly the coefficients associated to the extraneous formal asymptotic solutions, and hence do not contribute to a normalizable solution under the matching arguments.
These coefficients were set to zero at one side or the other (indicated above by which side the zero appears on) due to exponential growth.
The argument required that their coefficients match all the way through the turning point transition regions to the other side where they exponentially decay.

Starting at $f_{-}^{(-1)}$ in \eqref{eq:MatchedScatteringCoefs} and following a circuit around the equalities across the top and then back along the bottom, we see that we end up with $f_{-}^{(-1)}$ times a phase.
In order for the solution to be nonzero, it must be that this phase is one:
\begin{equation}
    1 = -\exp \left( \frac{\ii}{\epsilon} \int_{x_-(\zeta)}^{x_+(\zeta)} G^{-1}_0(u(x); \zeta) - G^{-1}_{-1}(u(x); \zeta) \, \dd x \right) \, .
\end{equation}
We simplify the integrand and define the function $R^\wyl$ by \eqref{eq:RWeylLawDef}.
The condition above is satisfied for those $\zeta = \zeta_n$ which satisfy the ILW Weyl Law \eqref{eq:WeylLaw}.
Note, $R^\wyl$ is positive because the integral is take over the classically allowed region, and here $W_0$ is real and greater than $-1$ while $W_{-1}$ is real and less than $-1$.

\begin{lemma}[Some properties of the ILW Weyl Law] \label{lem:WylLaw}
For $u_0$ an admissible initial condition satisfying Definition \ref{def:AdmisInitCond}, the following properties of $R^\wyl$ hold.
    \begin{enumerate}
        \item $R^\wyl(1/2\delta)$ is finite if 
            \begin{equation}
                \int_{-\infty}^{+\infty} \sqrt{u(x)} \, \dd x < \infty \, .
                \label{eq:FiniteRCond}
            \end{equation}

        \item $R^\wyl(\zeta_\maxrm) = 0$.
        
        \item $R^\wyl(\zeta)$ is continuous and decreasing along $[1/2\delta, \zeta_\maxrm]_{\mathcal{Q}_\delta}$ and differentiable on $(1/2\delta, \zeta_\maxrm)_{\mathcal{Q}_\delta}$ with negative parameterized derivative 
        \begin{equation}
            \rho^\wyl(\kappa) := -\frac{\dd}{\dd \kappa} R^\wyl(\zeta(\kappa)) = \frac{1}{2} \zeta'(\kappa) \, E'(\zeta(\kappa)) \int_{x_-(\zeta(\kappa))}^{x_+(\zeta(\kappa))} r(\zeta(\kappa); u_0(x)) \, \dd x
            \label{eq:RhoWylDef}
        \end{equation}
        where
        \begin{equation}
            r (\zeta; u) := \re{\frac{W_{0}\left(-2 \delta \zeta \ee^{-2\delta(\zeta + u)} \right)}{1 + W_{0}\left(-2 \delta \zeta \ee^{-2\delta(\zeta + u)} \right)} - \frac{W_{-1}\left(-2 \delta \zeta \ee^{-2\delta(\zeta + u)} \right)}{1 + W_{-1}\left(-2 \delta \zeta \ee^{-2\delta(\zeta + u)} \right)}}
            \label{eq:WyelLawDerivIntegrand}
        \end{equation}
        defined for all $\zeta \in \quadplus{\delta}$ and $u \geq 0$.
    \end{enumerate}
\end{lemma}

\begin{proof}
    The first property follows from noting that as $\zeta \to 1/2\delta$, $x_\pm(\zeta) \to \pm\infty$. Then plugging in $\zeta = 1/2\delta$ into the integrand
    \begin{equation}
        W_{0}\left( -\ee^{-1 - 2\delta u_0(x)} \right) - W_{-1}\left( -\ee^{-1 - 2\delta u_0(x)} \right) \, .
    \end{equation}
    Once more using the series \cite[Equation 4.13.6]{DLMF}, we see that for large enough $|x|$, this integrand behaves like $\sqrt{u_0(x)}$, so the condition \eqref{eq:FiniteRCond} guarantees that this integrand is integrable over $\R$, and thus $R^\wyl(1/2\delta)$ is finite.
    
    On the other hand, for $\zeta = \zeta_\maxrm$, $x_\pm(\zeta) = x_\maxrm$. Thus, the integral returns zero here.
    Let $\zeta < \xi \in [1/2\delta, \zeta_\maxrm]_{\mathcal{Q}_\delta}$, then $E(\zeta) < E(\xi)$.
    Thus, we have after adding $u_0(x)$, multiplying by $-2\delta$, exponentiating and then negating,
    \begin{equation}
        -2 \delta \zeta \ee^{-2\delta(\zeta + u_0(x))} < -2 \delta \xi \ee^{-2\delta(\xi + u_0(x))} \, .
    \end{equation}
    We also have $x_-(\zeta) < x_-(\xi) < x_+(\xi) < x_+(\zeta)$. Using the fact that $W_{-1}$ and $W_0$ are monotone on $(-\ee^{-1}, 0)$,
    \begin{align}
        W_{-1}\left( -2 \delta \zeta \ee^{-2\delta(\zeta + u_0(x))} \right) &>  W_{-1}\left( -2 \delta \xi \ee^{-2\delta(\xi + u_0(x))} \right) \, , \\
        W_0\left( -2 \delta \zeta \ee^{-2\delta(\zeta + u_0(x))} \right) &<  W_0\left( -2 \delta \xi \ee^{-2\delta(\xi + u_0(x))} \right) \, ,
    \end{align}
    for $x_-(\xi) < x < x_+(\xi)$. Additionally, the integrand of $R^\wyl(\zeta)$ is positive outside of this, we have $R^\wyl(\zeta) > R^\wyl(\xi)$.
    Moreover, everything here is continuous, along with the endpoints of the integration, and we've already shown that this is integrable for all $\zeta \in [1/2\delta, \zeta_\maxrm]_{\mathcal{Q}_\delta}$, so continuity of $R$ follows trivially.

    Lastly, differentiability is inherited from the formulas involved.
    The derivative of $R$ contains a boundary term from differentiating the endpoints and an integral term from differentiating the integrand.
    The boundary term is easily shown to be zero, due to the fact that both the $n = -1, 0$ branches return the same value at the turning points.
    What remains is the integral term.
    After plugging in the parameterization $\zeta(\kappa)$ and taking the derivative with respect to $\kappa$, we have the terms present in $r(\zeta(\kappa); u_0(x))$, along with some extra terms from chain rule and a term from the derivative of Lambert $W$:
    \begin{align}
        \frac{-2\delta + (2\delta)^2\zeta(\kappa)}{-2\delta\zeta(\kappa)} \zeta'(\kappa) &= \left( \frac{1}{\zeta(\kappa)} - 2\delta \right) \zeta'(\kappa) \\
        &= 2\delta E'(\zeta(\kappa)) \zeta'(\kappa) \, .
    \end{align}
    Since the terms in $r(\zeta(\kappa); u_0(x))$ have no worse than inverse square-roots at the end points, we collect these terms and get the formula \eqref{eq:RhoWylDef}.\footnote{Note, at this moment the real part evaluation in \eqref{eq:RhoWylDef} seems unnecessary, since for $x_-(\zeta(\kappa)) \leq x \leq x_+(\zeta(\kappa))$, the expression inside is manifestly real.
    However, for $x$ outside of this range, the quantity inside of the real part is manifestly imaginary from conjugate symmetry, and the real part results in $r$ being identically zero.
    We reserve this property for use later, and so define it this way.}
\end{proof}

Since the maximal value $R^\wyl(1/2\delta) < \infty$, for each $\epsilon > 0$, there is a maximal $N \in \Z$ such that the Weyl Law can be satisfied. This means there are always a finite number of asymptotic approximate eigenvalues, but the total number $N = \bigO{1/\epsilon}$ as $\epsilon \to 0^+$.

As mentioned at the start of this section, \cite{MinzoniMiloh_1994} first provided this Weyl Law for the eigenvalues of the ILW scattering equation by a similar matching argument much that the one proved above.
Here we will go a step further and use our constructed approximate formal normalizable eigenfunction to compute the asymptotic form of the norming constant $c_n$ associated to an approximated eigenvalue $\zeta_n$ with eigenfunction $\Phi$.
Using the definition \eqref{eq:AKSNormingConstDef}, accounting for the transformations \eqref{eq:CannonicalTransform} and \eqref{eq:KappaToZetaWaveFuncTransform}, we have
\begin{equation}
    c_n = \left( \int_{-\infty}^{+\infty} |\Phi(x)|^2 \, \dd x \right)^{-1} \, .
    \label{eq:NormingConstDef}
\end{equation}
where the normalization of $\Phi$ is chosen so that
\begin{equation}
    \lim_{x \to +\infty} \Phi(z) \ee^{\ii \zeta^* z/\epsilon} = 1 \, . \label{eq:NormingConstDefNormaliz}
\end{equation}
We can use this normalization condition to fix the final undetermined constant in our matched solution. As $x \to +\infty$, we eventually pass into the right classically forbidden region, and forever after, can use the formal asymptotic solution $\Phi_0$, which was the only one determined to contribute in this region.
\begin{align}
    &\lim_{x \to +\infty} \Phi(x+\ii y; \zeta_n, x_+(\zeta_n)) \ee^{\ii \zeta^* (x+\ii y)/\epsilon} \nonumber \\
    &\Quad[2] = f_{+}^0 \left( \frac{(2 \delta)^2 \zeta^*_n}{1 -2 \delta \zeta^*_n} \right)^{1/2} \lim_{x \to +\infty}  \exp \Bigg( \frac{\ii}{\epsilon} \left[ \zeta^*_n x_+(\zeta) + \int^x_{x_+(\zeta_n)} \zeta^*_n - G^{-1}_0(u(x'); \zeta_n) \, \dd x' \right] \Bigg) \, ,
\end{align}
where we've used $G^{-1}_0(0; \zeta_n) = \zeta^*_n$.
Because $u_0(x) \to 0$ as $x \to +\infty$, $G^{-1}_0(U; \zeta_n)$ is differentiable at $U = 0$ for $\zeta_n \in (1/2\delta, \ \zeta_\maxrm)$, it follows that the integral above is integrable in the limit $x \to +\infty$. In order for the normalization \eqref{eq:NormingConstDefNormaliz} to be satisfied for our matched solution, it must be that
\begin{equation}
    f_{+}^0 = \left( \frac{1 -2 \delta \zeta^*_n}{(2 \delta)^2 \zeta^*_n} \right)^{1/2} \exp \Bigg( -\frac{\ii}{\epsilon} \left[ \zeta^*_n x_+(\zeta_n) + \int^{+\infty}_{x_+(\zeta_n)} \zeta^*_n - G^{-1}_0(u(x'); \zeta_n) \, \dd x' \right] \Bigg)  \left( 1 + \bigO{\epsilon} \right) \, .
\end{equation}
Now, to compute the associated norming constant $c_n$ to an eigenvalue $\zeta_n$, we need to integrate the norm-squared of our matched solution. We do this by focusing on in each of the regions separately:
\begin{enumerate}
    
    \item \textbf{The classically forbidden regions.} Clearly as $\epsilon$ shrinks, the exponential in \eqref{eq:FormalAsympScatteringSolution} here goes to zero point-wise in $x$.
    The contribution to the integral of the norm-squared in \eqref{eq:NormingConstDef} out here can be bounded by integrating all the way into the turning points and applying Watson's lemma,
    \begin{align}
        \int_{-\infty}^{x_-(\zeta_n) - \epsilon^\alpha 2 \delta M^-_{\chi}} \left| \Phi_{f-}(x; \zeta_n) \right|^2 \, \dd x &= |f^{(0)}_+|^2 \bigO{\epsilon^{1/3}} \text{ and } \label{eq:LeftForbidRegionNormConstInt}\\
        \int_{x_+(\zeta_n) + \epsilon^\alpha 2 \delta M^+_{\chi}}^{+\infty} \left| \Phi_{f+}(x; \zeta_n) \right|^2 \, \dd x &= |f_+^{(0)}|^2 \bigO{\epsilon^{1/3}} \label{eq:RightForbidRegionNormConstInt}
    \end{align}
    for the left region and right regions respectively. Here, in the first equation we used the result from matching $|f^{(-1)}_{-}| = |f^{(0)}_{+}|$ \eqref{eq:MatchedScatteringCoefs} since $G^{-1}_{-1}(u(x);\zeta_n)$ and $G^{-1}_{0}(u(x);\zeta_n)$ are purely real for $x_-(\zeta_n) \leq x \leq x_+(\zeta_n)$.
    
    \item \textbf{The transition regions.} Again we expect this contribution to be small since the region itself shrinks like $\epsilon^\alpha$ and after using the expansion \eqref{eq:PsiInftyFullApprox} with the fact that $\Ai$ for a real argument is uniformly bounded, we have
    \begin{equation}
        \int_{x_\pm(\zeta_n) - \epsilon^\alpha 2 \delta M^\pm_{\chi}}^{x_\pm(\zeta_n) + \epsilon^\alpha 2 \delta M^\pm_{\chi}} \left| \Phi_{t\pm}(x; \zeta_n) \right|^2 \, \dd x = |t_{-}^{(-\infty)}|^2 \bigO{\epsilon^{2/3+\alpha}} = |f_{+}^{(0)}|^2 \bigO{\epsilon^{1/6}} \, . \label{eq:TransRegionsNormConstInt}
    \end{equation}
    Here, we also used the result from matching \eqref{eq:MatchingF-1TA} where we have explicitly computed $|C_{-1}(\epsN)|^2 = 2 h \pi \epsilon$ and used \eqref{eq:MatchedScatteringCoefs}.

    \item \textbf{The classically allowed region.} Here we have two contributing solutions. Using the matching conditions \eqref{eq:MatchedScatteringCoefs}, we have
    \begin{align}
        \int_{a(\epsilon)}^{b(\epsilon)} \left| \Phi_{a}(x; \zeta_n) \right|^2 \, \dd x = |f_{+}^{0}|^2 \int_{a(\epsilon)}^{b(\epsilon)} \left|  \Phi_{-1}(x; \zeta_n, x_+(\zeta_n)) - \Phi_0(x; \zeta_n, x_+(\zeta_n)) \right|^2 \, \dd x
    \end{align}
    where $a(\epsilon) = x_-(\zeta_n) + \epsilon^\alpha 2 \delta M^-_\chi$ and $b(\epsilon) = x_+(\zeta_n) - \epsilon^\alpha 2 \delta M^+_\chi$. We expand the norm-squared and plug in the formula \eqref{eq:FormalAsympScatteringSolution}
    \begin{align}
        &\left| \Phi_{-1}(x; \zeta_n, x_+(\zeta_n)) - \Phi_0(x; \zeta_n, x_+(\zeta_n)) \right|^2 \\
        &\Quad[2] = \frac{2\delta W_{-1}(-2\delta \zeta_n \ee^{-2\delta(\zeta_n + u_0(x))})}{1+W_{-1}(-2\delta \zeta_n \ee^{-2\delta(\zeta_n - u_0(x))})} - \frac{2\delta W_0(-2\delta \zeta_n \ee^{-2\delta(\zeta_n + u_0(x))})}{1+W_0(-2\delta \zeta_n \ee^{-2\delta(\zeta_n + u_0(x))})} \nonumber \\
        &\Quad[6] - 2 \sqrt{\frac{2\delta W_{-1}(-2\delta \zeta_n \ee^{-2\delta(\zeta_n + u_0(x))})}{1+W_{-1}(-2\delta \zeta_n \ee^{-2\delta(\zeta_n + u_0(x))})} \frac{-2\delta W_0(-2\delta \zeta_n \ee^{-2\delta(\zeta_n + u_0(x))})}{1+W_0(-2\delta \zeta_n \ee^{-2\delta(\zeta_n + u_0(x))})} } \nonumber \\
        &\Quad[9] \times \sin \left( \frac{1}{\epsilon} \int^{x}_{x_+(\zeta)} G^{-1}_{-1}(u(x'); \zeta_n) - G^{-1}_0(u(x'); \zeta_n) \, \dd x' \right) 
    \end{align}
    where we used the facts that $-1 \leq W_0 < 0$ and $W_{-1} \leq -1$, with these Lambert $W$'s evaluated as above in the classically allowed region.
    From replacing the sine term with one, we see that the integrand is bounded by an inverse square-root at the turning points.
    This means that extending the integration domain all the way into the turning points, the function is integrable and the error is $\bigO{\epsilon^{\alpha/2}} = \bigO{\epsilon^{1/4}}$.
    Additionally, the term containing the sine can be bounded with the method of stationary phase similar to the exponential term for the classically forbidden regions, yielding another error term of $\bigO{\epsilon^{1/3}} = \bigO{\epsilon^{1/4}}$. What remains is integrable and constant in $\epsilon$, so we have
    \begin{align}
        &\int_{a(\epsilon)}^{b(\epsilon)} \left| \Phi_{a}(x; \zeta_n) \right|^2 \, \dd x \nonumber \\
        &\quad = |f_{+}^{0}|^2 \left( \int_{x_-(\zeta_n)}^{x_+(\zeta_n)} \frac{2\delta W_{-1}(-2\delta \zeta_n \ee^{-2\delta(\zeta_n + u_0(x))})}{1+W_{-1}(-2\delta \zeta_n \ee^{-2\delta(\zeta_n + u_0(x))})} - \frac{2\delta W_0(-2\delta \zeta_n \ee^{-2\delta(\zeta_n + u_0(x))})}{1+W_0(-2\delta \zeta_n \ee^{-2\delta(\zeta_n + u_0(x))})} \, \dd x + \bigO{\epsilon^{1/4}} \right) \, . \label{eq:AllowRegionNormConstInt}
    \end{align}

\end{enumerate}
Assembling the results \eqref{eq:LeftForbidRegionNormConstInt}, \eqref{eq:RightForbidRegionNormConstInt} and \eqref{eq:AllowRegionNormConstInt} allow us to compute the asymptotic form of the norming constants from \eqref{eq:NormingConstDef} which results in the formula \eqref{eq:NormingConstAsymptoticFormula}.

\subsection{The Modified Scattering Data}
\label{sec:ModScattData}

After the asymptotic formulas for the eigenvalues and the norming constants have been determined, what remains is the reflection coefficient.
This is defined for $\zeta > 1/2\delta$. Since this means that $-2\delta \zeta \ee^{-2\delta \zeta} < -\ee^{-1}$, then the formal asymptotic solution for all $n \in \Z$ have regular amplitudes for all $x \in \R$ because, as $u_0(x) > 0$
\begin{equation}
    -2\delta \zeta \ee^{-2\delta (\zeta + u_0(x))} < -\ee^{-1} \, .
\end{equation}
Thus there are no turning points indicating the reflection coefficient goes to zero.
All of these heuristic results motivate the following definition:
\begin{definition}[The Modified Scattering Data] \label{def:ModScattData}
    Let $u_0: \R \to \R^+$ be an admissible initial condition satisfying Definition \ref{def:AdmisInitCond}.
    We define the sequence $\{ \epsilon_N \}_{N \in \N}$ by
    \begin{equation}
        \epsilon_N = \frac{1}{\pi N} R\left( \frac{1}{2 \delta} \right) \, .
        \label{eq:ModScattEpsN}
    \end{equation}
    Then we define the set of distinct eigenvalues $\{ \zeta_n \in \left[ \frac{1}{2\delta}, \zeta_\maxrm \right]_{\quadratrix{\delta}} \}_{n \in \Z_N}$ and associated set of norming constants $\{ c_n > 0\}_{n \in \Z_N}$ for $\epsilon_N$ to be given by
    \begin{align}
        R^\wyl(\zeta_n) &= \pi \left( n - \frac{1}{2} \right) \epsilon_N \label{eq:ModScattEigenvalues} \, , \\
        c_n &= \exp \left( -\frac{2}{\epsilon} \theta_+(\zeta_n) \right)
        \label{eq:ModScattNormingConst}
    \end{align}
    for all $1 \leq n \leq N$.
    Using the parametrization of the quadratrix by the imaginary part \eqref{eq:QuadZetaDef}, we may identify the eigenvalues $\zeta_n$ by their imaginary parts $\im{\zeta_n} = \kappa_n \in [0, \kappa_\maxrm]$ where $\kappa_\maxrm := \im{\zeta_\maxrm}$. 
    Lastly, we take the reflection coefficient to be identically zero.
\end{definition}

Note that we have dropped everything but the leading order behavior in $\epsilon$ in the definition of our modified scattering data. Because the matching argument did not handle errors in a rigorous way, we cannot make any statements about the small-dispersion limit for an exact initial condition. Instead, we content ourselves to study the inverse problem for this modified scattering data, which we expect to be close to the true scattering data as we take the small-dispersion limit. 
The sense in which we mean that the semiclassical soliton ensemble is ``close'' to the original initial condition is that the initial condition is recovered as $\epsilon \to 0^+$ for $t = 0$, as will be demonstrated in the proof of Theorem \ref{thm:L2Limit}.

\section{Inverse Scattering for the Semiclassical Soliton Ensemble}
\label{sec:Inverse}

We now turn our attention to the process of constructing a solution to the ILW equation using inverse scattering.
Since the scattering data that we derived in the previous section through the direct scattering problem is not rigorous, we resign ourselves to study an inverse scattering problem which we expect to be asymptotically close as $\epsilon \to 0^+$ to the exact inverse scattering problem associated to a given initial condition $u_0$. 
\begin{definition}[The ILW semiclassical soliton ensemble] \label{def:SemiclassSoliEnsem}
For an admissible initial condition $u_0$ satisfying Definition \ref{def:AdmisInitCond}, we define the associated semiclassical soliton ensemble for the ILW equation to be the sequence for $N \in \N$ of $N$-soliton solutions \eqref{eq:ILWNSoliton} where we take $\epsilon = \epsN$ and plug in the modified scattering data from Definition \ref{def:ModScattData}.
It is convenient to express the $N$-soliton solution for the semiclassical soliton ensemble as 
\begin{equation}
    u^\sse_N(x, t) = \partial_x D_{\ii \delta \epsN} F_N(x, t) \, ,
    \label{eq:ILWSemiclassSolEnsemb}
\end{equation}
where $D_{\ii \delta \epsN}$ is the symmetric, finite-difference operator
\begin{equation}
    D_{\ii \delta \epsN} F_N(x, t) := \frac{F_N(x + \ii \delta \epsN, t) - F_N(x - \ii \delta \epsN, t)}{\ii 2 \delta \epsN} \label{eq:FinDifOppDef}
\end{equation}
and
\begin{align}
    Z_N(z, t) &:= \det \left( \id{N} + \Delta_N( z, t; \epsN ) \right) \, ,
    \label{eq:PartitionFunctionDef} \\
    F_N(z, t) &:= 2 \delta \epsN^2 \log Z_N(z, t) \, . \label{eq:FreeEnergyFunctionDef}
\end{align}  
\end{definition}
We remind the reader that the $N$-soliton formula is one of only a few aspects of the ILW IST that rests on rigorous foundation since (see \cite{Matsuno_1984_BiTransMeth}[Chapter 5] for a proof)
Analyzing the semiclassical soliton ensemble \eqref{eq:ILWSemiclassSolEnsemb} amounts to studying the behavior of the determinant in \eqref{eq:PartitionFunctionDef} as the matrix size grows without bound.

\subsection{Expanding the Determinant}
\label{sec:ExpanDet}

We begin our analysis by conducting a ``Fredholm-type expansion'' of the determinant that appears in \eqref{eq:PartitionFunctionDef}.
To demonstrate this expansion, let $M \in \N$ and $A \in \C_{M \times M}$ and consider
\begin{equation}
    \det \left( \id{M} + A \right) = \begin{vmatrix}
        1 + A_{1 \, 1} & A_{1 \, 2} & \cdots & A_{1 \, M} \\
        A_{2 \, 1} & 1 + A_{2 \, 2} & \cdots & A_{2 \, M} \\
        \vdots & \vdots & \ddots & \vdots \\
        A_{M \, 1} & A_{M \, 2} & \cdots & 1 + A_{M \, M}
    \end{vmatrix} \, .
\end{equation}
We use linearity of the determinant in the first column to expand our determinant into a sum of two: one with the contribution from $\id{M}$ in the first column and another with the contribution from $A$ in the first column.
\begin{equation}
    \det \left( \id{M} + A \right) = \begin{vmatrix}
        1 & A_{1 \, 2} & \cdots & A_{1 \, M} \\
        0 & 1 + A_{2 \, 2} & \cdots & A_{2 \, M} \\
        \vdots & \vdots & \ddots & \vdots \\
        0 & A_{M \, 2} & \cdots & 1 + A_{M \, M}
    \end{vmatrix} + \begin{vmatrix}
        A_{1 \, 1} & A_{1 \, 2} & \cdots & A_{1 \, M} \\
        A_{2 \, 1} & 1 + A_{2 \, 2} & \cdots & A_{2 \, M} \\
        \vdots & \vdots & \ddots & \vdots \\
        A_{M \, 1} & A_{M \, 2} & \cdots & 1 + A_{M \, M}
    \end{vmatrix} \, .
\end{equation}
Now we expand on the second column of each of these determinants yielding a sum of four matrix determinants, one for each possible ways to select the first two columns independently from either $\id{M}$ or $A$.
\begin{align}
    \det \left( \id{M} + A \right) &= \begin{vmatrix}
        1 & 0 & \cdots & A_{1 \, M} \\
        0 & 1 & \cdots & A_{2 \, M} \\
        \vdots & \vdots & \ddots & \vdots \\
        0 & 0 & \cdots & 1 + A_{M \, M}
    \end{vmatrix} + \begin{vmatrix}
        1 & A_{1 \, 2} & \cdots & A_{1 \, M} \\
        0 & A_{2 \, 2} & \cdots & A_{2 \, M} \\
        \vdots & \vdots & \ddots & \vdots \\
        0 & A_{M \, 2} & \cdots & 1 + A_{M \, M}
    \end{vmatrix} \\
    &\qquad + \begin{vmatrix}
        A_{1 \, 1} & 0 & \cdots & A_{1 \, M} \\
        A_{2 \, 1} & 1 & \cdots & A_{2 \, M} \\
        \vdots & \vdots & \ddots & \vdots \\
        A_{M \, 1} & 0 & \cdots & 1 + A_{M \, M}
    \end{vmatrix} + \begin{vmatrix}
        A_{1 \, 1} & A_{1 \, 2} & \cdots & A_{1 \, M} \\
        A_{2 \, 1} & A_{2 \, 2} & \cdots & A_{2 \, M} \\
        \vdots & \vdots & \ddots & \vdots \\
        A_{M \, 1} & A_{M \, 2} & \cdots & 1 + A_{M \, M}
    \end{vmatrix} \, .
\end{align}
Continuing this procedure by expanding each resulting matrix determinant on each next column, it is clear that the result is a sum over all matrices formed by selecting every column independently from either $\id{M}$ or $A$.
Each of these matrices can be associated with a unique subset of $\Z_M$ which is the set of indices of the columns which are inherited from the matrix $A$.
The sum is then conducted over the power set of $\Z_M$,
\begin{equation}
    \det \left( \id{M} + A \right) = \sum_{S \subset \Z_M} \det(A^S)
\end{equation}
where $A^S$ is the matrix that has for each index in $S$, the associated column of $A$ and all other columns from $\id{M}$.

Now, we simplify some of these determinants once more by noting that performing a Laplace expansion of $\det(A^S)$ on a column which comes from $\id{M}$, only one principle minor contributes, which is the one where the column that we expanded on and the associated row are removed.
Iteratively Laplace expanding on the columns of $A^S$ which come from $\id{M}$, the only principle minor that survives is the minor which only contains the index elements of the set $S$, which we'll notate $\minor{S}{A}$.
Note that $\minor{\Z_M}{A} = A$ and that $\minor{\emptyset}{A}$ would be empty, but this minor would result from expanding the matrix determinant $\det(A^\emptyset) = \det(\id{M}) = 1$, so we set $\minor{\emptyset}{A} = [ \, 1 \, ]$ to write our final formula simply as:
\begin{equation}
    \det \left( \id{M} + A \right) = \sum_{S \subset \Z_M} \det \left( \minor{S}{A} \right) \, .
    \label{eq:FredholmExpan2}
\end{equation}

We may now apply this formula to $\det(\id{N} + \Delta_N( z, t; \epsN ))$ from \eqref{eq:ILWSemiclassSolEnsemb} for $z = x + \ii y \in \C$ and $t \in \R$.
We begin by noting that the decomposition of \eqref{eq:DeltaDecomp} extends to the principal minors $\minor{S}{\Delta_N(z,t;\epsN)}$ since the left and right matrices in the decomposition are diagonal.
For any $S \subset \Z_N$,
\begin{align}
    \det \minor{S}{\Delta_N( z, t; \epsN )} &= \det \minor{S}{D_N ( z, t; \epsN )}^2 \det \minor{S}{C_N} \\
    &= \left( \prod_{n \in S} c_n \exp \left( -2\im{\zeta_n} \frac{z}{\epsN} - 2\im{\left(\zeta_n - \frac{1}{2 \delta}\right)^2} \frac{t}{\epsN} \right) \right) \nonumber \\
    &\Quad[2] \times \left( \prod_{n \in S} \frac{1}{2 \im{\zeta_n}} \prod_{n < \ell \in S} \frac{\left| \zeta_\ell - \zeta_n \right|^2}{\left| \zeta_\ell - \zeta_n^* \right|^2} \right) \label{eq:DetPrincipleMinor}
\end{align}
where it is understood that a product over $S = \emptyset$ is one to be consistent with $\det (\minor{\emptyset}{\diamond}) = 1$.
Bringing the positive definite terms from the Cauchy determinant into the exponential and plugging in the modified norming constants \eqref{eq:ModScattNormingConst}, we get
\begin{align}
    \det \minor{S}{\Delta_N( z, t; \epsN )} &= \exp \left( - 2 \sum_{n \in S} \im{\zeta_n} \frac{z}{\epsN} + \im{\left(\zeta_n - \frac{1}{2 \delta}\right)^2} \frac{t}{\epsN} \right. \nonumber \\
    &\Quad[4] \left.+ 2 \sum_{n \in S} \frac{\theta_+(\zeta_n)}{\epsN} - \sum_{n, \ell \in S} g(\zeta_\ell, \zeta_n) \right) \, ,
    \label{eq:DeltaMinor2} 
\end{align}
where $g: \C^2 \to \R$ is the ``regulated'' Green's function for the Laplacian on the upper half-plane
\begin{equation}
    g(\zeta, \lambda) = \begin{cases}
        \displaystyle \log \left| \frac{\zeta - \lambda^*}{\zeta - \lambda\phantom{^*}} \right| & \text{if } \zeta \neq \lambda \\
        \displaystyle \log \left| \zeta - \lambda^* \right| & \text{if } \zeta = \lambda
    \end{cases} \, . \label{eq:DescLogKernel}
\end{equation}
Again, we note that sums over $S = \emptyset$ should be taken to be zero.
The full determinant $\det(\id{N} + \Delta_N(z/\epsN, t/\epsN))$ is a sum of terms \eqref{eq:DeltaMinor2} over over all possible $S \subset \Z_N$.
We search for a dominant balance in $\epsN$ in the exponent of the summand \eqref{eq:DeltaMinor2} and find it by collecting a factor of $\epsN$ with every sum over $S$ (i.e. $\epsN^2$ with the double sum).
As we take $N \to \infty$, we may then expect that these sums over $S$ behave like Riemann sums, and limit to integrals.
Towards this end, we define for $S \in \Z_N$, the associated point measures in one variable, which we parameterize using the imaginary part along the quadratrix,
\begin{equation}
    \dd \mu_S(\kappa) = \epsN \pi \sum_{n \in S} \delta(\kappa - \im{\zeta_n}) \, \dd \kappa
\end{equation}
for $0 \leq \kappa \leq \kappa_\maxrm < \pi/2\delta$ where 
\begin{equation}
    \kappa_\maxrm := \im{\zeta_\maxrm} \, ,
    \label{eq:KappaMaxDef}
\end{equation}
$\delta$ here is the Dirac-delta unit point measure and $\dd \kappa$ is the Lebesgue measure on the line.
Now, for ease of notation, we will also define several quantities, splitting the real from imaginary contributions. For $z \in \strip{\delta \epsN}$ and $t \in \R$
\begin{align}
    V(\zeta; z, t) &:= z \ \im{\zeta} + t \ \im{\left( \zeta - \frac{1}{2 \delta} \right)^2} - \theta_+(\zeta) \, , \label{eq:ElecExtPotential}\\
    E_{N, S}(z, t) &:= \int V(\zeta(\kappa); x, t) \, \dd \mu_S(\kappa) + \frac{1}{2 \pi} \iint g\big(\zeta(\kappa), \zeta(\eta)\big) \, \dd \mu_S(\kappa) \, \dd \mu_S(\eta) \, , 
    \label{eq:ElecEnergy} \\
    P_S &:= \int \kappa \, \dd \mu_S(\kappa) \, , 
    \label{eq:ElecDipole} \\
    Q_S &:= \int \im{\left( \zeta - \frac{1}{2\delta} \right)^2} \, \dd \mu_S(\kappa) \, .
    \label{eq:ElecQuadpole}
\end{align}
Thus, using the notation \eqref{eq:ElecExtPotential}, \eqref{eq:ElecEnergy} and \eqref{eq:ElecDipole} to simplify the result \eqref{eq:DeltaMinor2} and then plug this into \eqref{eq:FredholmExpan2}, we get a new expression for the full determinant and define a new compact notation for the result:
\begin{align}
    Z_N(z,t) = \det \left( \id{N} + \Delta_{N}\left( z, t; \epsN \right) \right) &= \sum_{S \subset \Z_N} \exp \left( -\frac{2}{\epsN^2 \pi} E_{N,S}(z, t) \right) \, .
\end{align}

Looking at \eqref{eq:PartitionFunctionDef} and briefly setting $z \in \R$, you may notice some similarities with the standard real exponential integral with large parameter $\lambda := \frac{2}{\epsN^2 \pi} > 0$, $I(\lambda) := \int_a^b \ee^{-\lambda f(x)} \, \dd x$ where $f$ is a $\mathcal{C}^2(\R)$ function on $[a, b] \subset \R$.
In the case where $f$ has a unique global minimum $x_0 \in (a,b)$, it is well-known that in the limit $\lambda \to \infty$, $\log I(\lambda) \approx -\lambda f(x_0) + \frac{1}{2}\log \left(\frac{2\pi}{\lambda f^{''}(x_0)}\right)$ by Laplace's method. 
Our sum \eqref{eq:PartitionFunctionDef} for fixed $N$ is, first off, a discrete sum over particular sets of sample measures.
As $N \to \infty$, the size of the set of sample measures grows as $2^N$, and therefore we can think of the sum in \eqref{eq:PartitionFunctionDef} as a Riemann sum-like discretization of a functional integral which is integrated over some set of admissible densities $\mathcal{A}$, $\mathcal{I}(\lambda) := \int_\mathcal{A} \ee^{- \lambda \mathpzc{E}[\rho]} \mathcal{D}\rho$.
It is then reasonable to expected that if there is a unique minimizer $\rho_0$ of the resulting functional $E$ over $\mathcal{A}$, that $\log \mathcal{I}(\lambda) \approx - \lambda \mathpzc{E}[\rho_0] + \cdots$ by a functional integral generalization of Laplace's method.
In effect, this is exactly what P. D. Lax and C. D. Levermore proved for the KdV equation in \cite{LaxLevermoreI_1983}, and what we will show in section \ref{sec:DiscrToContMeas}.

In the case of the ILW equation, we have the addition complication that the exponentials are complex, since in \eqref{eq:ILWSemiclassSolEnsemb} we need to plug $z = x \pm \ii \delta \epsN$ into \eqref{eq:PartitionFunctionDef}.
However, since the imaginary part of the exponent in \eqref{eq:PartitionFunctionDef} is $\bigO{\epsN}$, the upshot is that this does not present too much difficulty.
In preparation for the main result, we establish the following bounds on the magnitude and argument of $Z_N(z, t) \in \C$.

\begin{theorem}
    \label{thm:BoundsOnDet}
    Let $N \in \N$, $z = x + \ii y \in \strip{\delta \epsilon}$ and $t \in \R$.
    Define the minimum value
    \begin{equation}
        E^\minrm_N(x, t) := \min_{S \subset \Z_N} E_S(x, t) \, .
        \label{eq:EnergyMinDef}
    \end{equation}
    Then, $Z_N(z, t)$ has the bound in magnitude
    \begin{equation}
        0 < \left[ d \left( 2 \kappa_\maxrm \frac{|y|}{\epsN} \right) \right]^N \leq \left| Z_N(z, t) \right| \leq 2^N \exp \left( -\frac{2}{\epsN^2 \pi} E^{\minrm}_N(x, t) \right) \, ,
        \label{eq:DetIPlusDeltaMagBounds}
    \end{equation}
    where
    \begin{equation}
        d(\theta) := \left\{ \begin{array}{cc}
            \displaystyle 1 & 0 \leq \theta \leq \frac{\pi}{2} \\
            \displaystyle \sin \theta & \frac{\pi}{2} < \theta \leq \pi
        \end{array} \right. ,
        \label{eq:dDefinition}
    \end{equation}
    and the bound in argument
    \begin{align}
        0 &\leq -\sgn(y) \arg\left( Z_N(z, t) \right) \leq 2 N \kappa_{\maxrm} \frac{|y|}{\epsN} \, ,
        \label{eq:ArgDetBounds}
    \end{align}
    where the argument here is defined to be zero when $z \in \R$ and continuous for the rest of $z \in \strip{\epsN \delta}$.
\end{theorem}

\begin{proof}
    The right side of \eqref{eq:DetIPlusDeltaMagBounds} follows from an application of the triangle inequality to \eqref{eq:PartitionFunctionDef} and the fact that $E^\minrm_N(x, t)$ bounds all of the $E_{N,S}(x, t)$ from below.
    The left side of \eqref{eq:DetIPlusDeltaMagBounds} and the argument bounds \eqref{eq:ArgDetBounds} are quite a bit more involved and require an a analysis of the eigenvalues of the matrix $\Delta_N( z, t; \epsN )$.
    
    We begin by using the factorization \eqref{eq:DeltaDecomp} to peel off the matrix multiplicative contributions dependent on the imaginary part $y$,
    \begin{equation}
        \Delta_N( z, t; \epsN ) = \exp \left( -\ii K y \right) \ \Delta_N( x, t; \epsN ) \ \exp \left( -\ii K y \right) \, ,
    \end{equation}
    where $K$ is the $N \times N$ diagonal matrix with $K_{nn} := \im{\zeta_n} = \kappa_n$ for all $1 \leq n \leq N$.
    Let $\lambda$ be an eigenvalue of $\Delta_N( z, t; \epsN )$, then it follows by matrix similarity that it is also an eigenvalue of the matrix
    \begin{equation}
        A := \exp \left( -\ii K \frac{y}{\epsN} \right) \ \Delta_N( z, t; \epsN ) \ \exp \left( \ii K \frac{y}{\epsN} \right) = \exp \left( -\ii 2 K \frac{y}{\epsN} \right) \ \Delta_N( x, t; \epsN ) \, .
        \label{eq:SimilarDelta}
    \end{equation}
    It is clear that the matrix $\exp \left( -\ii 2 K \frac{y}{\epsN} \right)$ is unitary.
    Now we want to show $\Delta_N( x, t; \epsN )$ is positive-definite.
    Looking back at \eqref{eq:DetPrincipleMinor} with $z = x \in \R$, we see that the principle minors $\Delta_N( x, t; \epsN )$ are easily seen to be manifestly positive, so positive-definiteness follows immediately by Sylvester's criterion.
    
    For each distinct eigenvalue $\lambda$, we have at least one eigenvector $v \in \C^N$ of the matrix $A$.
    By the calculation,
    \begin{align}
         |\lambda| \, \norm{v} &= \norm{A \ v} \\
         &= \norm{\exp \left( -\ii 2 K \frac{y}{\epsN} \right) \ \Delta_N( x, t; \epsN ) \ v} \\
         &= \norm{\Delta_N( x, t; \epsN ) \ v} \, ,
         \label{eq:LambdaNormRelation}
    \end{align}
    we have that
    \begin{equation}
        |\lambda| = \frac{\norm{\Delta_N(x, t; \epsN) \ v}}{\norm{v}} > 0
    \end{equation}
    since $\norm{v} \neq 0$.
    So all of the eigenvalues of $\Delta_N( z, t; \epsN )$ are nonzero and their arguments are well-defined.
    
    Now, for simplicity of calculation, let $u = \Delta_N( x, t; \epsN ) \, v$ and consider
    \begin{align}
         \lambda \, u^\dagger v &= u^\dagger \, A \, v  \\
         &= u^\dagger \, \exp \left( -\ii 2 K \frac{y}{\epsN} \right) \, \Delta_N( x, t; \epsN ) \ v \\
         &= u^\dagger \, \exp \left( -\ii 2 K \frac{y}{\epsN} \right) \, u \, .\label{eq:LambdaPhaseRelation}
    \end{align}
    Since $\Delta_N(x, t; \epsN)$ is, again, positive-definite,
    \begin{equation}
        u^\dagger v = v^\dagger \ \Delta_N( x, t; \epsN ) \ v > 0 \, .
    \end{equation}
    Thus, the principal argument of $\lambda$ must coincide with the principal argument of the right side of \eqref{eq:LambdaPhaseRelation}.
    We divide by the positive $\norm{u}^2$ and then can simplify this using the diagonal form of $K$ and the components of $u$, $u_n \in \C$ for $1 \leq n \leq N$, 
    \begin{equation}
        \arg ( \lambda ) = \arg \left( \frac{\displaystyle u^\dagger \ \exp \left( -\ii 2 K \frac{y}{\epsN} \right) \ u}{\norm{u}^2} \right) = \arg \left( \sum_{n = 1}^N \exp \left( -\ii 2 \kappa_n \frac{y}{\epsN} \right) \frac{|u_n|^2}{\norm{u}^2} \right) \, .
        \label{eq:LambdaArgument1}
    \end{equation}
    We note from the conjugate symmetry of the phases across $y = 0$, we have
    \begin{equation}
        \arg ( \lambda ) = - \sgn(y) \arg \left( \sum_{n = 1}^N \exp \left( \ii 2 \kappa_n \frac{|y|}{\epsN} \right) \frac{|u_n|^2}{\norm{u}^2} \right) \, .
        \label{eq:LambdaArgument2}
    \end{equation}
    
    From the definition of the modified scattering data \ref{def:ModScattData}, $\kappa_n \leq \kappa_\maxrm$ for all $1 \leq n \leq N$,
    \begin{equation}
        0 \leq 2 \kappa_n \frac{|y|}{\epsN} \leq 2 \kappa_\maxrm \frac{|y|}{\epsN} \leq 2 \delta \kappa_\maxrm < \pi \, .
    \end{equation}
    This confines all of the phases of the eigenvalues of $\Delta_N( z, t; \epsN )$ to a sector of angle less than $\pi$.
    Since the sum in \eqref{eq:LambdaArgument2} is a weighted average of these phases, by convexity of sectors with angle less than $\pi$, it follows that the weighted average of the phases is also contained in the sector and therefore shares the principal argument bounds.
    Finally, relating this to the principal argument of $\lambda$ by \eqref{eq:LambdaArgument2}, we have arrived at the uniform bound for any eigenvalue $\lambda$ of $\Delta_N( z, t; \epsN )$,
    \begin{equation}
        \label{eq:LambdaBounds}
        0 \leq - \sgn(y) \arg(\lambda) \leq 2 \kappa_\maxrm \frac{|y|}{\epsN} < \pi \, .
    \end{equation}

    Now we are ready to resolve the lower bound for $\left| Z_N(z, t) \right|$ in \eqref{eq:DetIPlusDeltaMagBounds}.
    Noting that this is the characteristic polynomial of the matrix $\Delta_{N}( z, t; \epsN )$ evaluated at one, we can let $\{ \lambda_n \}_{n = 1}^N$ be the set of eigenvalues (not necessarily distinct, duplicates accounting for algebraic multiplicity of the root) and we have 
    \begin{equation}
        |Z_N(z, t)| = \left| \det \left( \id{N} + \Delta_{N}( z, t; \epsN ) \right) \right| = \prod_{n = 1}^N \left| 1 + \lambda_n \right| \, .
        \label{eq:DetMagFromCharPoly}
    \end{equation}
    For each $1 \leq n \leq N$,
    \begin{align}
        \left| 1 + \lambda_n \right|^2 &= \Big| 1 + |\lambda_n| \ee^{ \ii \arg(\lambda_n)} \Big|^2 = 1 + 2 |\lambda_n| \cos \left(\arg(\lambda_n)\right) + |\lambda_n|^2.
    \end{align}
    Subject to \eqref{eq:LambdaBounds}, the expression is always minimized for fixed $|\lambda_n| \geq 0$ when $|\arg(\lambda_n)|$ is greatest:
    \begin{align}
        \left| 1 + \lambda_n \right|^2 &\geq 1 + 2 |\lambda_n| \cos \left( 2 \kappa_\maxrm \frac{|y|}{\epsN} \right) + |\lambda_n|^2 \\
        &= \left( |\lambda_n| + \cos \left( 2 \kappa_\maxrm \frac{|y|}{\epsN} \right) \right)^2 + \sin^2 \left( 2 \kappa_\maxrm \frac{|y|}{\epsN} \right) \, .
    \end{align}
    If $2 \kappa_\maxrm |y| / \epsN \leq \pi/2$, then the cosine term is non-negative and the quadratic is strictly increasing for $|\lambda_n| \geq 0$, thus the minimum is obtained at $|\lambda_n| = 0$, which has the value one.
    If instead $\pi/2 < 2 \kappa_\maxrm |y| / \epsN < \pi$, then the cosine term is negative and the quadratic has its minimum at for some $|\lambda_n| > 0$, which has the value $\sin^2 \left( 2 \kappa_\maxrm |y|/\epsN \right)$. 
    Putting these cases together and taking the positive square-root, we have from the definition \eqref{eq:dDefinition}
    \begin{equation}
        \left| 1 + \lambda_n \right| \geq d \left( 2 \kappa_\maxrm \frac{|y|}{\epsN} \right) \, .
    \end{equation}
    Using this bound for each factor in \eqref{eq:DetMagFromCharPoly}, we get the desired lower bound in \eqref{eq:DetIPlusDeltaMagBounds}.
    The positivity of this lower bound (left-most inequality in \eqref{eq:DetIPlusDeltaMagBounds}) follows from the fact that $d(\theta) > 0$ for all $0 \leq \theta < \pi$ and $2 \kappa_\maxrm |y| / \epsN \leq 2 \delta \kappa_\maxrm < \pi$.
    \\

    Lastly, because of the positivity of this lower bound argument of $Z_N(z, t)$ can be defined to be continuous throughout the strip $\strip{\delta \epsN}$.
    Moreover, since $\Delta_{N}( x, t; \epsN )$ is positive definite, $Z_N(x, t) > 0$, and so the argument can be taken to be zero on the real line.
    Because the elements of $\Delta_{N}( z, t; \epsN )$ are continuous functions of $z$ and $t$, there exists an indexing where each of the eigenvalues $\lambda_n$ are continuous in $z$ and $t$.
    With this enumeration, we can compute the continuous argument of $Z_N(z, t)$ according to 
    \begin{equation}
        \arg \big( Z_N(z, t) \big) = \arg \left( \det \left( \id{N} + \Delta_{N}( z, t; \epsN ) \right) \right) = \sum_{n = 1}^N \arg(1 + \lambda_n) \, .
        \label{eq:DetArgCalc}
    \end{equation}
    From a simple geometric argument involving sector where the phases of each $\lambda_n$ are restricted to \eqref{eq:LambdaBounds}, we can demonstrate that
    \begin{equation}
        0 \leq -\sgn(y) \arg(1 + \lambda_n) \leq -\sgn(y) \arg(\lambda_n) \leq 2 \kappa_\maxrm \frac{|y|}{\epsN} \, ,
    \end{equation}
    which, with \eqref{eq:DetArgCalc}, gives the bound \eqref{eq:ArgDetBounds}.
\end{proof}

\subsection{Characterizing the Minimizer}
\label{sec:DiscrToContMeas}

We're now going to focus on characterizing the minimal measure which we expect to dominate the determinant \eqref{eq:PartitionFunctionDef}.
This will be done by relating $E^\minrm_N(x,t)$, the minimizer over the discrete point-mass measures $\{ \dd \mu_S \}_{S \subset \Z_N}$, to a minimizer over a larger set of admissible measure densities $\mathcal{A}$.
The set $\mathcal{A}$ is exactly the set of densities associate to measures that are well-approximated by $\{ \dd \mu_S \}_{S \subset \Z_N}$ as $N \to \infty$ in the weak sense.
We will use the very same mathematical machinery that was developed by P. D. Lax and C. D. Levermore for working out the limit of the the KdV equation \cite[Section 2]{LaxLevermoreI_1983} with a few small modifications, for the ILW equation's small-dispersion limit. 

\begin{theorem} \label{thm:BoundedVarWeakConv}
    The measures $\dd \mu_S(\kappa)$ for all $S \subset \Z_N$ and $N \in \N$ have uniform bounded variation. Moreover $\{ \dd \mu_{\Z_N}(\kappa) \}_{N=1}^\infty$, i.e. the measures will all indices included, have the weak limit
    \begin{equation}
        \wlim{N \to +\infty} \dd \mu_{\Z_N}(\kappa) = \rho^\wyl(\kappa) \, \dd \kappa
        \label{eq:WeakLimitAllSitesPop}
    \end{equation}
    where $\rho^\wyl$ is defined by the ILW Weyl law \eqref{eq:RhoWylDef}.
\end{theorem}

\begin{proof}
    First, the measures $\dd \mu_S(\kappa)$ are nonnegative, and we have the clear bound
    \begin{equation}
        \int \dd \mu_S(\kappa) \leq \int \dd \mu_{\Z_N}(\kappa) = \epsN \pi N \leq R\left( \frac{1}{2\delta} \right) < \infty \, ,
    \end{equation}
    where we used the definition of the modified scattering data \eqref{eq:ModScattEigenvalues} and the properties of the Weyl law in Lemma \ref{lem:WylLaw}.
    Now let $f:[0, \kappa_\maxrm] \to \R$ be a continuous function.
    We integrate $f$ against $\dd \mu_{\Z_N}$ and change variables with $r = R^\wyl(\zeta(\kappa))$, letting $R^{-1}(r) = \kappa$ denote the inverse.
    \begin{align}
        \int f(\kappa) \, \dd \mu_{\Z_N}(\kappa) &= \epsN \pi \sum_{n = 1}^N f(\kappa_n) \\
        &= \epsN \pi \sum_{n = 1}^N f(R^{-1}(\epsN \pi n)) \, .
    \end{align}
    This is exactly the right endpoint Riemann sum with uniform $N^{\textrm{th}}$ partition approximate to the integral
    \begin{equation}
        \int_0^{R^\wyl(1/2\delta)} f(R^{-1}(r)) \dd r = - \int_0^{\kappa_\maxrm} f(\kappa) \frac{\dd}{\dd \kappa} R^\wyl(\zeta(\kappa)) \, \dd \kappa = \int_0^{\kappa_\maxrm} f(\kappa) \rho^\wyl(\kappa) \, \dd \kappa \, .
    \end{equation}
    Because this is integrable and the integrand is continuous, we know that the right endpoint Riemann sum limits to the integral, and thus,
    \begin{equation}
        \lim_{N \to +\infty} \int f(\kappa) \, \dd \mu_{\Z_N}(\kappa) = \int_0^{\kappa_\maxrm} f(\kappa) \rho^\wyl(\kappa) \, \dd \kappa \, .
    \end{equation}
    This demonstrates \eqref{eq:WeakLimitAllSitesPop}.
\end{proof}

Knowing that the discrete measures $\dd \mu_S$ have lower and upper bounds which have weak limits as $N \to +\infty$, we expect the measures $\dd \mu_S$ to well-approximate in the limit measures with densities belonging to the admissible set of densities
\begin{equation}
    \mathcal{A} := \left\{ \ \rho \ \left| \ \rho \text{ is a real, measurable function on $[0, \kappa_\maxrm]$ and } 0 \leq \rho \leq \rho^\wyl \ \right. \right\} \, .
\end{equation}
Since $\rho^\wyl \, \dd \kappa$ as a measure has trivially bounded variation
\begin{equation}
    \norm{\rho^\wyl}_{L^1} = \int_0^{\kappa_\maxrm} \rho^\wyl(\kappa) \, \dd \kappa = R\left( \frac{1}{2 \delta} \right) < \infty \, ,
\end{equation}
each $\rho \in \mathcal{A}$ has an associated measure $\dd \mu = \rho \, \dd \kappa$ which have uniformly bounded variation.
We can define now functionals which are versions of the quantities $E_{N,S}(x,t)$ \eqref{eq:ElecEnergy}, $P_S$ \eqref{eq:ElecDipole}, and $Q_S$ \eqref{eq:ElecQuadpole} but which can be computed using the densities $\rho \in \mathcal{A}$,
\begin{align}
    \mathpzc{E}[\rho](x, t) &:= \int_0^{\kappa_\maxrm} \left( V(\zeta(\kappa); x, t) \ \rho(\kappa) + \frac{1}{2} \mathpzc{L}[\rho](\zeta(\kappa)) \right) \ \rho(\kappa) \, \dd \kappa \, ,
    \label{eq:EContDef} \\
    \mathpzc{L}[ \rho ](\lambda) &:= \frac{1}{\pi} \int_0^{\kappa_\maxrm} \log \left| \frac{\lambda - \zeta(\eta)^*}{\lambda - \zeta(\eta)\phantom{^*}} \right| \rho(\eta) \, \dd \eta \, ,
    \label{eq:LContDef} \\
    \mathpzc{P}[\rho] &:= \int_0^{\kappa_\maxrm} \kappa \ \rho(\kappa) \, \dd \kappa \, , \label{eq:PContDef} \\
    \mathpzc{Q}[\rho] &:= \int_0^{\kappa_\maxrm} \im{\left( \zeta - \frac{1}{2\delta} \right)^2} \ \rho(\kappa) \, \dd \kappa \, , \label{eq:QContDef}
\end{align}
and where $V$ is still the same as in \eqref{eq:ElecExtPotential}.

We note that for the discrete version \eqref{eq:DescLogKernel} of the logarithmic kernel of $\mathpzc{L}$ \eqref{eq:LContDef} had to have a modified diagonal, i.e. when $\kappa = \eta$, for finiteness.
For the functional version, the logarithmic kernel of $\mathpzc{L}$ \eqref{eq:LContDef} is positive and has no worse than a logarithmic singularity on the diagonal $\kappa = \eta$ and that
\begin{theorem} \label{thm:LContAndBound}
    $\mathpzc{L}[\rho^\wyl]$ is continuous and takes the value
    \begin{equation}
        \mathpzc{L}[\rho^\wyl](\zeta(\kappa)) = \theta_+(\kappa) + \theta_-(\kappa)
    \end{equation}
    for $\kappa \in [0, \kappa_\maxrm]$, i.e. on the interval along the quadratrix $[1/2\delta, \zeta_\maxrm]_{\quadratrix{\delta}}$.
    Consequently, $\mathpzc{L}[\rho]$ is continuous and satisfies
    \begin{equation}
        0 \leq \mathpzc{L}[\rho](\zeta) \leq \mathpzc{L}[\rho^\wyl](\zeta)
    \end{equation}
    for all $\rho \in \mathcal{A}$ and $\zeta \in \C_+$.
\end{theorem}
The proof of the first statement is presented with the proof of Corollary \ref{cor:LRhoWyl}.
The bounds in the second statement follow trivially from noticing that the kernel of $\mathpzc{L}$ is positive on the upper half-plane and using the bounds from the definition of $\rho \in \mathcal{A}$.
Continuity of $\mathpzc{L}[\rho]$ follows from considering a sequence $\{ \zeta_m \}_{m \in \N}$ in the upper half-plane with a limit and applying the generalized dominated convergence theorem to the integrands of $\mathpzc{L}[\rho](\zeta_m)$ and $\mathpzc{L}[\rho^\wyl](\zeta_m)$.
The continuity of $\mathpzc{L}[\rho]$ then extends to the whole plane by odd symmetry.
From these properties, making the appropriate adjustments to the proof of Theorem 2.2 from \cite{LaxLevermoreI_1983}, we have
\begin{theorem}[The limit of the discrete minimizers $E_N^\minrm$.] \label{thm:DescMinEnergyToContMinEnergy}
    For all $(x, t) \in \R^2$, there exists a minimum of the functional $\mathpzc{E}[\diamond](x,t)$ \eqref{eq:EContDef} attained on the admissible set $\mathcal{A}$, 
    \begin{equation}
        E^\minrm_\infty(x, t) := \min_{\rho \in \mathcal{A}} \mathpzc{E}[\rho](x, t) \, .
        \label{eq:MinEnergyInfDef}
    \end{equation}
    Furthermore
    \begin{equation}
        \lim_{N \to +\infty} E_N^\minrm(x, t) = E^\minrm_\infty(x, t)
        \label{eq:LimitOfDescMinEnergy}
    \end{equation}
    uniformly on compact subsets of $(x,t) \in \R^2$.
\end{theorem}
That is to say, the limit as $N \to +\infty$ of the discrete minima $E^\minrm_N(x,t)$ is equal to the minimum of the functional $\mathpzc{E}[\diamond](x,t)$ \eqref{eq:EContDef} over the admissible set of densities $\mathcal{A}$.
Should the reader be interested in the content of these proofs, we direct their attention first to \cite{LaxLevermoreI_1983} and then, for a more in-depth explanation, to C. D. Levermore's dissertation \cite[Chapter 2]{Levermore_1982_Diss}.

\subsection{The Distributional Limit of the Semiclassical Soliton Ensemble Solution}
\label{sec:DistLimSol}

\begin{theorem} \label{thm:DistLimit}
    For any $t \in \R$, the distributional limit of \eqref{eq:ILWSemiclassSolEnsemb} is given by
    \begin{equation}
        \dd-\lim_{N \to \infty} u^\sse_N(\diamond, t) = -\frac{4 \delta}{ \pi} \partial_x^2 E^\minrm_\infty(\diamond, t)
        \label{eq:DistribLimitOfSol}
    \end{equation}
    where $E^\minrm_\infty$ is defined in \eqref{eq:MinEnergyInfDef}.
\end{theorem}
\begin{proof}
    We will prove the distributional limit from definition.
    Let $\chi \in \mathcal{C}^\infty_0(\R)$ and let $K \subset \R$ be a compact support of $\chi$, and consider the solution $u^\sse_N(x, t)$ integrated against $\chi$ in $x$:
    \begin{equation}
        \int_{\R} \chi(x) \ u^\sse_N(x, t) \, \dd x = \int_{\R} \chi(x) \ \partial_x D_{\ii \delta \epsN} F_N(x, t) \, \dd x \, .
    \end{equation}
    Moving the derivative onto $\chi$ is a standard invoking of integration-by-parts:
    \begin{align}
        \int_{\R} \chi(x) \ u^\sse_N(x, t) \, \dd x = - \int_{\R} \chi_x(x) \ D_{\ii \delta \epsN} F_N(x, t) \, \dd x \, ,
        \label{eq:IntegrationByParts}
    \end{align}
    where the boundary term disappears due to the compact support of $\chi$.
    Dealing with the finite-difference operator $D_{\ii \delta \epsN}$ is much trickier.
    We will exploit the fact that the difference is shrinking with $\epsN$ and using a near-analytic extension of $\chi$ to get our result.
    
    First, define $\phi_+: \strip{\delta \epsN} \to \C$ by
    \begin{equation}
        \phi_+(z) := \chi_x(x) + \ii(y - \delta \epsN) \chi_{xx}(x) \, ,
        \label{eq:PhiNearAnalyExt}
    \end{equation}
    where again we always identify $z = x + \ii y$.
    Then, from the $\dbar$-generalization of Cauchy's theorem,
    \begin{equation}
        \int_{C} \phi_+(z) \ F_N(z, t) \, \dd z = \int_a^b \int_0^{\delta \epsN} \dbar_z \big( \phi_+(z) \ F_N(z, t) \big) \, \dd x \, \dd y 
        \label{eq:GeneralizedCauchyThm}
    \end{equation}
    with $a < b$ such that $K \subset [a, b]$ and $C$ is the boundary of the rectangle $(x,y) \in [a, b] \times [0, \delta \epsN]$, i.e. the integration region on the left-hand side, oriented positively.
    Since $F_N$ is analytic, which follows from Theorem \ref{thm:BoundsOnDet} where the determinant was shown to be nonzero for all $|y| \leq \delta \epsN$, we have that $\dbar_z F_N(z, t) = 0$, so product rule gives
    \begin{align}
        \dbar_{z} \big( \phi_+(z) \, F_N(z, t) \big) &= F_N(z, t) \ \left( \frac{\partial}{\partial x} + \ii \frac{\partial}{\partial y} \right) \phi_+(x + \ii y) \\
        &= \ii (y - \delta \epsN) \chi_{xxx}(x) \, F_N(z, t) \, .
        \label{eq:DBarIntegrand}
    \end{align}
    Since $\chi$ is $\mathcal{C}^\infty$, there is a maximum magnitude of $\chi_{xxx}$ on $[a,b]$, call it $M_{\chi,3} \geq 0$.
    Let us also introduce notation for the maximum modulus of $F_N$ on a $y$-symmetric rectangular region 
    \begin{equation}
        M_{F_N} := \max_{\substack{x \in [a,b] \\ y \in [-\delta \epsN, \delta \epsN]}} |F_N(z, t)| \, .
    \end{equation}
    Thus, inserting \eqref{eq:DBarIntegrand} into the integrand of the double integral in \eqref{eq:GeneralizedCauchyThm}, we have the bound
    \begin{equation}
        \left| \int_{a}^{b} \int_0^{\delta \epsN} \dbar_{z} \big( \phi_+(z) \ F_N(z, t) \big) \, \dd x \, \dd y \right| \leq \frac{\delta^2 \epsN^2}{2} (b-a) M_{\chi,3} M_{F_N} \, .
        \label{eq:RectangleIntegralBound}
    \end{equation}
    Now let us consider the contour integral on the left side of \eqref{eq:GeneralizedCauchyThm} by parameterization:
    \begin{align}
        \int_{C} \phi_+(z) \, F_N(z, t) \, \dd z &= \int_{a}^{b} \phi_+(x) \, F_N(x, t) \, \dd x + \int_0^{\delta \epsN} \phi_+(b + \ii y) \, F_N(b + \ii y, t) \, \ii \, \dd y \nonumber \\
        &\qquad + \int_{b}^{a} \phi_+(x + \ii \delta \epsN) \, F_N(x + \ii \delta \epsN, t) \, \dd x \nonumber \\
        &\qquad + \int_{\delta \epsN}^0 \phi_+(a + \ii y) \, F_N(a + \ii y, t) \, \ii \, \dd y
    \end{align}
    We plug in the form of $\phi_+$ \eqref{eq:PhiNearAnalyExt} and simplify,
    \begin{align}
        \int_{C} \phi_+(z) \, F_N(z, t) \, \dd z &= \int_{a}^{b} \big( \chi_x(x) - \ii \delta \epsN \chi_{xx}(x) \big) \, F_N(x, t) \, \dd x  - \int^0_{\delta \epsN} (y - \delta \epsN) \chi_{xx}(b) \ F_N(b + \ii y, t) \, \dd y \nonumber \\
        &\qquad - \int_{a}^{b} \chi_x(x) \, F_N(x + \ii \delta \epsN, t) \, \dd x \nonumber \\
        &\qquad + \int_0^{\delta \epsN} (y - \delta \epsN) \chi_{xx}(a) \, F_N(a + \ii y, t) \, \dd y
        \label{eq:SecondRectangularBoundaryIntegral}
    \end{align}
    where the $\chi_x$ terms disappeared in the vertical segments due the endpoints being outside of its support: $\chi_x(a, t) = \chi_x(b, t) = 0$.
    Because of this vanishing, the integrals over the vertical segments can be bounded easily by $\frac{\delta^2 \epsN^2}{2} M_{\chi, 2} M_{F_N}$, 
    where $M_{\chi, 2}$ is a uniform bound on $|\chi_{xx}|$ over $x \in [a, b]$.
    We have succeeded in bounding the difference between the integrals along the horizontal segments of $C$ which involve $F_N(x,t)$ and $F_N(x + \ii \delta \epsN, t)$.
    This forms half of the finite-difference operator \eqref{eq:FinDifOppDef} on $F_N$.
    For the other half, we do the exact same exercise but for a rectangle in the lower half-plane and for a function $\phi_-: \strip{\delta \epsN} \to \C$ defined by
    \begin{equation}
        \phi_-(z) := \chi_x(x) + \ii(y + \delta \epsN) \chi_{xx}(x) \, .
    \end{equation}
    This gives identical bounds since the previous bounds on $\chi$ were $y$-independent and the bound on $F_N$ was chosen to be on all of $\strip{\delta \epsN}$.
    
    Subtracting the resulting lower half-plane horizontal integrals from the those of the upper half-plane in \eqref{eq:SecondRectangularBoundaryIntegral} and dividing by $\ii 2 \delta \epsN$ lets us assemble the finite-difference in \eqref{eq:IntegrationByParts} along with a term integrated on the real axis:
    \begin{align}
        &\frac{1}{\ii 2 \delta \epsN} \left( \int_{a}^{b} \chi_x(x) \, F_N(x + \ii \delta \epsN, t) \, \dd x - \int_{a}^{b} \Big( \chi_x(x) - \ii \delta \epsN \chi_{xx}(x) \Big) \, F_N(x, t) \, \dd x \right. \nonumber \\
        &\left. \qquad \qquad - \int_{a}^{b} \chi_x(x) \, F_N(x - \ii \delta \epsN, t) \, \dd x + \int_{a}^{b} \Big(\chi_x(x) + \ii \delta \epsN \chi_{xx}(x)\Big) \, F_N(x, t) \, \dd x \right) \nonumber \\
        &\qquad \qquad \qquad \qquad = \int_{a}^{b} \chi_x(x) \, D_{\ii \delta \epsN} F_N(x, t) \, \dd x + \int_{a}^{b} \chi_{xx}(x) \,  F_N(x, t) \, \dd x \, .
        \label{eq:FiniteDiffApprox}
    \end{align}
    The bounds on the double integral \eqref{eq:RectangleIntegralBound} and those of the vertical components along with the rearranging the equalities \eqref{eq:GeneralizedCauchyThm} and \eqref{eq:SecondRectangularBoundaryIntegral} lets us assemble a bound on the left-hand side of equality \eqref{eq:FiniteDiffApprox} by na\"ive triangle identity.
    Nearly there, we combine our first integration-by-parts result \eqref{eq:IntegrationByParts} with \eqref{eq:FiniteDiffApprox} and its bound to get our first approximation to the integral of $u^\sse_N(\diamond, t)$ against $\chi$:
    \begin{align}
        &\left| \int_{a}^{b} \chi(x) \, u^\sse_N(x, t) \, \dd x - \int_{a}^{b} \chi_{xx}(x) \, F_N(x, t) \, \dd x \right| \nonumber \\
        &\Quad[8] = \left| \int_{a}^{b} \chi_x(x) \, D_{\ii \delta \epsN} F_N(x, t) \, \dd x + \int_{a}^{b} \chi_{xx}(x) \, F_N(x, t) \, \dd x \right| \label{eq:FirstApproxIntSolChi} \\
        &\Quad[8] \leq \frac{\delta \epsN}{2} \Big( (b-a) M_{\chi,3} + M_{\chi,2} \Big) M_{F_N} \, .
        \label{eq:FirstApproxIntSolChiBound}
    \end{align}
    
    To demonstrate the limit, we need to show that the right-hand side of \eqref{eq:FirstApproxIntSolChi} indeed goes to zero.
    For this, we split $F_N$ into real and imaginary parts
    \begin{align}
        \re{F_N(z, t)} &= 2 \delta \epsN^2 \ln \left| Z_N(z, t) \right| \, , \\
        \im{F_N(z, t)} &=  2 \delta \epsN^2 \arg \left( Z_N(z, t) \right) \, ,
    \end{align}
    and then applying Theorem \ref{thm:BoundsOnDet} to the right sides of the above equalities, we have the bounds on the real and imaginary parts of $F_N$
    \begin{align}
        2 \delta \epsN^2 N \ln d( 2 \delta \kappa_\maxrm) &\leq \re{F_N(z, t)} \leq - \frac{4 \delta}{\pi} E^{\minrm}_N(x, t) + 2 \delta \epsN^2 N \ln 2 \, , \label{eq:ReFBounds} \\
        0 &\leq \left| \im{F_N(z, t)} \right| \leq 4 \delta^2 \epsN^2 N \kappa_\maxrm \, . \label{eq:ImFBounds}
    \end{align}
    As $N \to +\infty$, all of terms in these bounds limit to zero independently of $x, t$ except for the term \\ $- 4 \delta/\pi \, E^{\minrm}_N(x, t)$.
    Indeed, the terms $\epsN^2 N$ tend to zero, by the construction of $\epsN$ \eqref{eq:ModScattEpsN}.
    Additionally, $E^{\minrm}_N(x, t)$, due to Theorem \ref{thm:DescMinEnergyToContMinEnergy}, has a uniform limit on $[a, b]$.
    Since its real and imaginary parts have uniform bounds, so too does $|F_N|$ and hence $M_{F_N}$.
    This shows that the bound in \eqref{eq:FirstApproxIntSolChiBound} will tend to zero as $N \to +\infty$.
    
    Lastly, since we now only have $F_N$ evaluated at $x \in \R$ at the beginning of \eqref{eq:FirstApproxIntSolChiBound}, we can establish a tighter lower bound than \eqref{eq:ReFBounds}, relying directly on the fact that $Z_N(x, t)$ \eqref{eq:PartitionFunctionDef} is a sum of positive terms, which is greater than any one of its terms:
    \begin{align}
        F_N(x, t) &= 2 \delta \epsN^2 \log \big( Z_N(x, t) \big) \\
        &= 2 \delta \epsN^2 \log \left[ \sum_{S \subset \Z_N} \exp \left( -\frac{2}{\epsN^2 \pi} E_S(x, t) \right) \right] \\
        &\geq 2 \delta \epsN^2 \log \left[ \exp \left( -\frac{2}{\epsN^2 \pi} E_N^\minrm(x, t) \right) \right] \\
        &\geq -\frac{4 \delta}{\pi} E_N^\minrm(x, t) \, .
    \end{align}
    This with the upper bound from \eqref{eq:ReFBounds} yields the three-sided inequality
    \begin{align}
        0 \leq F_N(x, t) + \frac{4 \delta}{\pi} E_N^\minrm(x, t) &\leq 2 \delta \epsN^2 N \ln 2 \, .
        \label{eq:TightFReBounds}
    \end{align}
    By the sandwich theorem, we have that $F_N(x, t) \to -4 \delta/\pi \ E_N^\minrm(x, t)$ uniformly in $x$ and $t$ as $N \to +\infty$. Considering
    \begin{align}
        &\left| \int_{a}^{b} \chi(x) \, u^\sse_N(x, t) \, \dd x + \int_{a}^{b} \chi_{xx}(x) \, \frac{4 \delta}{\pi} E^\minrm_\infty(x, t) \, \dd x \right| \nonumber \\
        &\qquad \qquad \leq \left| \int_{a}^{b} \chi(x) \, u^\sse_N(x, t) \, \dd x - \int_{a}^{b} \chi_{xx}(x) \, F_N(x, t) \, \dd x \right| + \left| \int_{a}^{b} \chi_{xx}(x) \, \left( F_N(x, t) + \frac{4 \delta}{\pi} E^\minrm_N(x, t) \right) \, \dd x \right| \nonumber \\
        &\qquad \qquad \qquad \qquad + \left| \frac{4 \delta}{\pi} \int_{a}^{b} \chi_{xx}(x, t) \ \big( E^\minrm_N(x, t) - E^\minrm_\infty(x, t) \big) \, \dd x \right| \label{eq:FinalApproxIntSolChi} \\
        &\qquad \qquad \leq \frac{\delta \epsN}{2} \Big( R M_{\chi,3} + M_{\chi,2} \Big) M_{F_N} + (b-a) M_{\chi,2} 2 \delta \epsN^2 N \ln 2 \nonumber \\
        &\qquad \qquad \qquad \qquad + (b-a) M_{\chi,2} \frac{4 \delta}{\pi} \max_{x \in [a,b]} \big| E^\minrm_N(x, t) - E^\minrm_\infty(x, t) \big| \, , 
        \label{eq:FinalApproxIntSolChiBound}
    \end{align}
    where in the first step we added and subtracted terms and split using the triangle identity, then in the second step we bounded the first term on the right in \eqref{eq:FinalApproxIntSolChi} using \eqref{eq:FirstApproxIntSolChiBound} and the second two by the na\"ive maximum on the integration range.
    Lastly, we have that the first and second term in \eqref{eq:FinalApproxIntSolChiBound} tends to zero from the arguments above while the third term tends to zero by the uniformity of this limit in Theorem \ref{thm:DescMinEnergyToContMinEnergy}.
    Since $\chi \in \mathcal{C}^\infty_0$ was arbitrary, we have proved that the limit of $u^\sse_N(x, t)$ in a distributional sense coincides with $-4 \delta/\pi \ \partial_{x}^2 E^\minrm_\infty(x, t)$.
\end{proof}

Theorem \ref{thm:DistLimit} demonstrates two important qualities of the ILW small dispersion solution.
First, that the finite difference in the $N$-soliton solution between the two edges of the strip $\strip{\delta \epsN}$ can be approximated in a distributional sense by $\epsN$ times a derivative on the center of the strip as $\epsN \to 0^+$, as one might expect.
And second, the limit of $F_N(x,t)$ as a large exponential sum is behaving exactly as described after \eqref{eq:PartitionFunctionDef}, that is that thinking of this as a Riemann-sum discretizations of a functional integral, we have just demonstrated a generalization of Laplace's method for functional exponential integrals.
As it turns out, this minimal value $E^\minrm_\infty(x,t)$ directly determines the solution of the ILW semiclassical soliton ensemble.

\subsection{Characterizing the Minimizing Density}
\label{sec:MinMeas}

Theorem \ref{thm:DescMinEnergyToContMinEnergy} guarantees that the that there exists a minimizer of $\mathpzc{E}[\diamond](x, t)$ in $\mathcal{A}$, now we wish to characterize and construct it.
Towards this end, we will show that this minimizer is uniquely determined by a set of variational conditions, the result of Theorem \ref{thm:ConstLogPotentChargeDensSol}. This will follow from

\begin{lemma} \label{lem:PositiveDefL}
    For any $\rho_1, \rho_2 \in \mathcal{A}$ let $\rho = \rho_1 - \rho_2$, then
    \begin{equation}
        \int_0^{\kappa_\maxrm} \mathpzc{L}[ \rho ](\zeta(\kappa)) \, \rho(\kappa) \, \dd \kappa \geq 0 \, ,
        \label{eq:PositiveDefL}
    \end{equation}
    with equality iff $\rho_1 = \rho_2$.
\end{lemma}

\begin{proof}
    Let $\rho$ be as in the hypothesis, then $\rho \in L^1[0, \kappa_\maxrm]$, so we can define the two-dimensional Fourier transform for measures restricted along the quadratrix by
    \begin{equation}
        \hat{\rho}(\ell, k) := \int_K \ee^{ -\ii \ell \xi(\kappa) - \ii k \kappa } \ \rho(\kappa) \, \dd \kappa
        \label{eq:RhoFourier}
    \end{equation}
    where $K = [-\kappa_\maxrm, \kappa_\maxrm]$ and we extend $\rho$ to $[-\kappa_\maxrm, 0)$ with odd symmetry.
    Then we define the Gaussian-mollified sequence of two-dimensional densities $\{ \rho_n:\C \to \R \}_{n = 1}^\infty$,
    \begin{equation}
        \rho_{n}(\zeta) := \int_K G_n(\zeta-\zeta(\kappa')) \, \rho(\kappa') \, \dd \kappa' \, , \text{ where } G_n(\zeta) := \frac{n^2}{2 \pi} \ee^{-n^2 \left| \zeta \right|^2 / 2} \label{eq:MollifiedRhoDef}
    \end{equation}
    The net result of this mollification is that the charge density $\rho$, which can be thought of as a distribution supported on the quadratrix, is spread out into the plane by convolution with a radially symmetric Gaussian of width $1/n$. 
    
    We exploit the fact that $\rho_{n}$ are Schwartz class and $\rho_{n}(\zeta) = -\rho_{n}(\zeta^*)$ to establish an equality between the ``spatial'' integral \eqref{eq:PositiveDefL} and a Fourier integral using Parseval's theorem and using the tempered distribution for the Fourier transform of the $\log$-kernel
    \begin{align}
        \ell_{n}^{\textrm{Spat}} :=& -\frac{1}{2} \int_{\C^2} \log \left| \zeta - \lambda \right| \rho_{n}(\zeta) \rho_{n}(\lambda) \, \dd^2 \zeta \, \dd^2 \lambda \label{eq:LSpatial} \\
        &= \frac{1}{4\pi} \int_{\R^2} \frac{\left| \hat{\rho}_{n}(\ell, k) \right|^2}{\ell^2 + k^2} \, \dd \ell \, \dd k =: \ell_{n}^{\textrm{Four}} \, , \label{eq:LFourier}
     \end{align}
    where we have the shorthand for the real two-dimensional Lebesgue measure on the complex plane $\dd^2 z = \dd \re{z} \dd \im{z}$ and $\hat{\rho}_{n}$ is the standard two dimensional Fourier transform of $\rho_{n}$.
    Because of the form of \eqref{eq:LFourier}, the lemma statement applied instead to $\rho_n$ is now apparent.
    We now wish to show that as $n \to \infty$, that is as the width of the Gaussian in \eqref{eq:MollifiedRhoDef}, $\ell_{n}^{\textrm{Spat}}$ tends to the integral in the lemma statement.
    
    Towards this end, we note that the entire integrand of \eqref{eq:LSpatial} with \eqref{eq:MollifiedRhoDef} plugged in is absolutely integrable, use $\rho \in L^1[-\kappa_\maxrm, \kappa_\maxrm]$ and $\displaystyle \max_{\kappa' \in [-\kappa_\maxrm, \kappa_\maxrm]} G_n(\zeta - \zeta(\kappa'))$ is $C^2(\C)$ which decays beyond all orders.
    From Fubini, we rearrange the order of integration and isolate the terms which constitute pure convolutions on the plane.
    Then we use the commutation of convolutions and conduct the convolutions of the two Gaussians explicitly:
    \begin{align}
        \ell_{n}^{\textrm{Spat}} &= -\frac{1}{2} \int_{K^2} \rho(\kappa) \rho(\kappa') \int_{\C^2} G_{n}(\zeta(\kappa') - \zeta) \log \left| \zeta - \lambda \right| G_{n}(\lambda-\zeta(\kappa)) \, \dd^2 \zeta \, \dd^2 \lambda \, \dd \kappa' \, \dd \kappa \\
        &= -\frac{1}{2} \int_{K^2} \rho(\kappa) \rho(\kappa') \int_{\C^2} \log \left| \zeta(\kappa') - \lambda \right| G_{n}(\lambda - \zeta') G_{n}(\zeta'-\zeta(\kappa)) \, \dd^2 \zeta' \, \dd^2 \lambda \, \dd \kappa' \, \dd \kappa \\
        &= -\frac{1}{2} \int_{K^2} \rho(\kappa) \rho(\kappa') \int_{\C} \log \left| \zeta(\kappa') - \lambda \right| G_{n/\sqrt{2}}(\lambda-\zeta(\kappa)) \, \dd^2 \lambda \, \dd \kappa' \, \dd \kappa \, .
    \end{align}
    Now we split the $\kappa'$ integral over zero and use the oddness of $\rho$ to write the integral over just $[0, \kappa_\maxrm]$,
    \begin{equation}
        \ell_{n}^{\textrm{Spat}} = \frac{1}{2} \int_{K} \rho(\kappa) \int_{0}^{\kappa_\maxrm} \rho(\kappa') \int_{\C} \log \left| \frac{\zeta(\kappa')^* - \lambda}{\zeta(\kappa')^{\phantom{*}} - \lambda} \right| G_{n/\sqrt{2}}(\lambda-\zeta(\kappa)) \, \dd^2 \lambda \, \dd \kappa' \, \dd \kappa \, .
    \end{equation}
    where we combined the inner-most integral with $\kappa' \to -\kappa'$ and an overall sign difference between them.
    Fubini's theorem again let's us reorder the inner two integrals since the integrand is dominated by an integrable function: use $|\rho(\kappa')| \leq \rho^\wyl(\kappa')$ and that the new $\log$-kernel is positive everywhere on the integration domain.
    The integral over $\kappa'$ is then exactly the integral computing $\mathpzc{L}[\rho](\lambda)$, so we choose to compute this one first, 
    \begin{align}
        \ell_{n}^{\textrm{Spat}} &= \frac{1}{2} \int_{K} \rho(\kappa) \int_{\C} G_{n/\sqrt{2}}(\zeta(\kappa)-\lambda) \mathpzc{L}[\rho](\lambda) \, \dd^2 \lambda \, \dd \kappa \\
        &= \int_{0}^{\kappa_\maxrm} \rho(\kappa) (G_{n/\sqrt{2}} * \mathpzc{L}[\rho])(\zeta(\kappa)) \, \dd \kappa \, ,
    \end{align}
    where we introduced a notation for the convolution and used the even symmetry of the integrand in $\kappa$ so as to show that we have almost recovered exactly the integral \eqref{eq:PositiveDefL}.
    The hope is that as the width of the Gaussian in the convolution goes to zero ($n \to \infty$), we recover just $\mathpzc{L}[\rho]$, and thus the integral in \eqref{eq:PositiveDefL}.
    We set this up by computing the norm of the difference
    \begin{align}
        \left| \int_0^{\kappa_\maxrm} \rho(\kappa) \mathpzc{L}[ \rho ](\zeta(\kappa)) \, \dd \kappa - \ell_{n}^{\textrm{Spat}} \right| &= \left| \int_0^{\kappa_\maxrm} \rho(\kappa) \Big( \mathpzc{L}[ \rho ](\zeta(\kappa)) - (G_{n/\sqrt{2}} * \mathpzc{L}[\rho])(\zeta(\kappa)) \Big) \, \dd \kappa \right| \\
        &\leq \norm{\rho}_{L^1[0,\kappa_\maxrm]} \norm{\mathpzc{L}[ \rho ](\zeta(\diamond)) - (G_{n/\sqrt{2}} * \mathpzc{L}[\rho])(\zeta(\diamond))}_{L^\infty[0,\kappa_\maxrm]} \\
        &\leq \norm{\rho^\wyl}_{L^1} \norm{\mathpzc{L}[ \rho ] - (G_{n/\sqrt{2}} * \mathpzc{L}[\rho])}_{L^\infty(\C)} \, .\label{eq:SpatL1LInftyNormDiffBound}
    \end{align}
    
    It suffices to show that $G_{n/\sqrt{2}} * \mathpzc{L}[\rho] \to \mathpzc{L}[\rho]$ uniformly as $n \to \infty$.
    For this, we need a partition of unity $\chi: \C \to [0,1]$
    \begin{equation}
        \chi(\zeta) := \begin{cases}
            1 &\text{if } |\zeta| \leq \frac{3}{2}|\zeta_\maxrm| \\
            0 &\text{if } |\zeta| \geq 2|\zeta_\maxrm|
        \end{cases}
    \end{equation}
    and which is $C^2(\C)$ with bounded second partial derivatives.
    Employing the partition and linearity of the the convolution operation
    \begin{equation}
        G_{n/\sqrt{2}} * \mathpzc{L}[\rho] = G_{n/\sqrt{2}} * (\chi \mathpzc{L}[\rho]) + G_{n/\sqrt{2}} * \big( (1-\chi) \mathpzc{L}[\rho]) \big) \, . \label{eq:GaussConvPartOfUnit}
    \end{equation}
    In the first term of \eqref{eq:GaussConvPartOfUnit}, $\chi \mathpzc{L}[\rho]$ is continuous by Theorem \ref{thm:LContAndBound} and compactly supported, by Heine-Cantor it is uniformly continuous.
    A simple argument using the definitions of uniform continuity and the convolution gives that $\{ G_{n/\sqrt{2}} * (\chi \mathpzc{L}[\rho]) \}_{n=1}^\infty$ is a uniformly equi-continuous sequence.
    By the Arzel\`a-Ascoli theorem, there is a subsequence $\{n_m\}_{m=1}^\infty$ for which $G_{n_m/\sqrt{2}} * (\chi \mathpzc{L}[\rho]) \to \chi \mathpzc{L}[\rho]$ uniformly as $m \to \infty$.

    Now we work on the second term of \eqref{eq:GaussConvPartOfUnit}.
    Noticing that there is a non-zero distance between the support of $1-\chi$ and $[\zeta_\maxrm^*, \zeta_\maxrm]_{\quadratrix{\delta}}$, we see immediately from the differentiability of the log-kernel that the second partials of $(1-\chi) \mathpzc{L}[\rho]$ are uniformly bounded on $\C$, say by the constant $C > 0$.
    Using an two dimensional extension of Taylor's remainder theorem, it follows that for any $\zeta \in \C$
    \begin{align}
        &\left| \big[ G_{n/\sqrt{2}} * \big( (1-\chi) \mathpzc{L}[\rho] \big) \big](\zeta) - (1-\chi(\zeta)) \mathpzc{L}[\rho](\zeta) \right| \\
        &\qquad \leq \int_{\C} G_{n/\sqrt{2}}(\zeta - \lambda) \Big| \big( (1-\chi) \mathpzc{L}[\rho] \big)(\lambda) - \big((1-\chi) \mathpzc{L}[\rho]\big)(\zeta) \Big| \, \dd \lambda \\
        &\qquad \leq \int_{\C} G_{n/\sqrt{2}}(\zeta - \lambda) C |\lambda - \zeta|^2 \, \dd \lambda \\
        &\qquad = \frac{2\sqrt{2}C}{n} \, .
    \end{align}
    Thus, we have uniform convergence of $G_{n/\sqrt{2}} * \big((1-\chi) \mathpzc{L}[\rho] \big) \to (1-\chi) \mathpzc{L}[\rho]$ as $n \to \infty$.

    In light of \eqref{eq:SpatL1LInftyNormDiffBound}, using the partition of unity and the triangle identity for the $L^\infty$-norm, passing to the subsequence we have shown
    \begin{equation}
        \lim_{m \to \infty} \ell_{n_m}^{\textrm{Spat}} = \int_{0}^{\kappa_\maxrm} \rho(\kappa) \mathpzc{L}[\rho](\zeta(\kappa)) \, \dd \kappa \, .
    \end{equation}

    Now we work on the Fourier side \eqref{eq:LFourier}.
    An application of Fubini demonstrates that
    \begin{align}
        \hat{\rho}_n(\ell, k) &= \int_{\C} \ee^{-\ii k \re{\zeta} - \ii \ell \im{\zeta}} \rho_n(\zeta) \, \dd^2 \zeta \\
        &= \int_{\C} \int_K \ee^{-\ii k (\re{\zeta} - \xi(\kappa)) - \ii \ell (\im{\zeta} - \kappa)} G_{n}(\zeta - \zeta(\kappa)) \ee^{-\ii k \xi(\kappa) - \ii \ell \kappa} \rho(\zeta) \, \dd \kappa \, \dd^2 \zeta \\
        &= \ee^{-(\ell^2 + k^2)/2n^2} \hat{\rho}(\ell, k) \, .
    \end{align}
    We see that the integrand of
    \begin{equation}
        \ell_{n}^{\textrm{Four}} = \frac{1}{4\pi} \int_{\R^2} \ee^{-(\ell^2 + k^2)/n^2} \frac{\left| \hat{\rho}(\ell, k) \right|^2}{\ell^2 + k^2} \, \dd \ell \, \dd k
    \end{equation}
    is, due to the Gaussian term, a monotonically increasing sequence of functions which have a clear point-wise limit almost everywhere and whose integrals have a limit on the subsequence $\{ n_m \}_{m=1}^\infty$ from our work on the spatial side.
    So by monotone convergence theorem, we have the result
    \begin{equation}
        \int_{0}^{\kappa_\maxrm} \rho(\kappa) \mathpzc{L}[\rho](\zeta(\kappa)) \, \dd \kappa = \lim_{m\to\infty} \ell_{n_m}^{\textrm{Spat}} = \lim_{m\to\infty} \ell_{n_m}^{\textrm{Four}} = \frac{1}{4\pi} \int_{\R^2} \frac{\left| \hat{\rho}(\ell, k) \right|^2}{\ell^2 + k^2} \, \dd \ell \, \dd k \, . \label{eq:SpatialFourierResult}
    \end{equation}
    The Fourier integral in \eqref{eq:SpatialFourierResult} is clearly always non-negative and zero if $\rho$ is almost everywhere zero.
    For the strict positivity, we rely on the fact that $\hat{\rho}$ is $\C^\infty(\C)$ since $\rho$ is compactly supported.
    Thus, the Fourier integral in \eqref{eq:SpatialFourierResult} is zero, iff $\hat{\rho}$ is zero everywhere.
    Additionally, we realize that if $\hat{\rho}$ is identically zero then
    \begin{equation}
        0 = \ii^n \frac{\partial^n \hat{\rho}}{\partial \ell^n}(0, 0) = \int_K \kappa^n \rho(\kappa) \, \dd \kappa \, .
    \end{equation}
    Determining $\rho$ from these moments can be recast in the form of the Hausdorff moment problem, for which the criteria of a unique solution for the measure $\rho(\kappa) \, \dd \kappa$ from these moments is satisfied.
    And, we already know $\rho = 0$ has these moments, so we're done.
\end{proof}

Now we are ready to characterize the minimizing density.

\begin{theorem} \label{thm:VariationCondDetMin}
    For each $(x, t)$, the minimizing density $\rho^\minrm(\kappa) = \rho^\minrm(\kappa; x, t) \in \mathcal{A}$ is uniquely determined by the variational conditions
    \begin{equation}
        \left\{ \begin{array}{clc}
            \rho^\minrm(\kappa) = 0 & \mathrm{if} \ \kappa \in I_+, & \textrm{(voids)} \\
            0 \leq \rho^\minrm(\kappa) \leq \rho^\wyl(\kappa) & \mathrm{if} \ \kappa \in I_0, & \textrm{(bands)} \\
            \rho^\minrm(\kappa) = \rho^\wyl(\kappa) & \mathrm{if} \ \kappa \in I_-, \qquad &\textrm{(saturations)} 
        \end{array} \right. \, ,
        \label{eq:VariatConds}
    \end{equation}
    where the sets $I_-, I_0,$ amd $I_+$ partition $[0, \kappa_\maxrm]$ according to
    \begin{align}
        I_- = I_-(x,t) &:= \left\{ \kappa \in [0, \kappa_\maxrm] \ \Bigg| \ \frac{\delta E}{\delta \rho}\{\rho^\minrm\}\big(\zeta(\kappa); x, t\big) < 0 \ \right\} \, , \label{eq:VoidDef} \\
        I_0 = I_0(x,t) &:= \left\{ \kappa \in [0, \kappa_\maxrm] \ \Bigg| \ \frac{\delta E}{\delta \rho}\{\rho^\minrm\}\big(\zeta(\kappa); x, t\big) = 0 \ \right\} \, , \label{eq:BandDef} \\
        I_+ = I_+(x,t) &:= \left\{ \kappa \in [0, \kappa_\maxrm] \ \Bigg| \ \frac{\delta E}{\delta \rho}\{\rho^\minrm\}\big(\zeta(\kappa); x, t\big) > 0 \ \right\} \, .
        \label{eq:SatDef}
    \end{align}
    and the Fr\'echet derivative of the functional $E$ is given by
    \begin{equation}
        \frac{\delta E}{\delta \rho}\{\rho\}(\zeta) = \frac{\delta E}{\delta \rho}\{\rho\}(\zeta; x, t) := V(\zeta; x, t) + \mathpzc{L}[\rho](\zeta) \, .
        \label{eq:FrecetDeriv}
    \end{equation}
\end{theorem}
\begin{proof}
    Uniqueness of the minimizer $\rho^\minrm$ follows from the convexity of the admissible set of charge densities $\mathcal{A}$ and the strict convexity of the functional $\mathpzc{E}[ \diamond ](x, t)$, which in turn follows trivially from the quadratic form of $E$ and the positive-definiteness of the quadratic term from Lemma \ref{lem:PositiveDefL}. 
    The variational conditions can then be derived using these properties.
    Take $\rho \in \mathcal{A}$ and let $\rho^\minrm$ be a minimizer of consider $\rho_\theta = (1-\theta) \rho^\minrm + \theta \rho \in \mathcal{A}$ for $0 \leq \theta \leq 1$. Then, we have
    \begin{align}
        \mathpzc{E}[ \rho_\theta ](x, t) - \mathpzc{E}[ \rho^\minrm ](x, t) &= \theta \int_0^{\kappa_\maxrm} \Big[ V\left( \zeta(\kappa); x, t \right) + \mathpzc{L}[ \rho^\minrm ](\zeta(\kappa)) \Big] \left( \rho(\kappa) - \rho^\minrm(\kappa) \right) \, \dd \kappa \\
        &\quad + \frac{\theta^2}{2} \int_0^{\kappa_\maxrm} \mathpzc{L}[ \rho - \rho^\minrm ](\zeta(\kappa)) \left( \rho(\kappa) - \rho^\minrm(\kappa) \right) \, \dd \kappa \, .
    \end{align}
    Since $\rho^\minrm$ is the unique minimizer, this difference is nonnegative for all $0 \leq \theta \leq 1$, and thus, the coefficient $\theta$ must be nonnegative.
    The integrand of the linear term without the difference $\rho - \rho^\minrm$ we can identify as the Fr\'echet derivative \eqref{eq:FrecetDeriv}.
    
    Now we'll show that the variational conditions do indeed characterize the unique minimizer $\rho^\minrm$.
    We begin by constructing the particular $\rho_{-} \in \mathcal{A}$ which is
    \begin{equation}
        \rho_{-}(\kappa) = \left\{ \begin{array}{cl}
            \rho^\wyl(\kappa) & \text{if } s \in I_-(x, t) \\
            \rho^\minrm(\kappa; x, t)  & \text{if } s \notin I_-(x, t)
        \end{array} \right. \, .
    \end{equation}
    We then have for the coefficient of $\theta$,
    \begin{align}
        &\int_0^{\kappa_\maxrm} \frac{\delta E}{\delta \rho}\{ \rho^\minrm \}(\zeta(\kappa)) \left( \rho_{-}(\kappa) - \rho^\minrm(\kappa) \right) \, \dd \kappa = \int_{I_-} \frac{\delta E}{\delta \rho}\{ \rho^\minrm \}(\zeta(\kappa)) \left( \rho^\wyl(\kappa) - \rho^\minrm(\kappa) \right) \, \dd \kappa \, .
    \end{align}
    So if $\rho^\minrm$ differs from the upper bound $\rho^\wyl$ on $I_-$ then $\rho^\wyl - \rho^\minrm$ is positive while the Fr\'echet derivative is negative yielding a negative total value if the difference is supported on a nonzero measure set.
    Thus, we conclude that for $\rho^\minrm$ to truly be the minimizer, it can differ from the upper bound $\rho^\wyl$ by no more than a set of measure zero on $I_-$.
    Of course, sets of measurable functions are only distinguished up to sets of measure zero, and thus we may choose the set the output on the potentially deviating set to the upper bound and get the third condition in \ref{eq:VariatConds}.
    By an identical argument with the lower bound $0$ on $I_+$, we arrive at the first condition in \ref{eq:VariatConds}.
    And by exhaustion of cases, we have the second condition in \ref{eq:VariatConds}.
    We have shown that the unique minimizer $\rho^\minrm$ satisfies the variational conditions.
    
    Now suppose $\rho^* \in \mathcal{A}$ satisfies the variational conditions.
    Since $\rho^*$ satisfies the variational conditions and the fact that $\rho^\minrm$ is the minimizer, we have
    \begin{align}
        &\int_0^{\kappa_\maxrm} \frac{\delta E}{\delta \rho} \{ \rho^* \} \left( \zeta(\kappa); x, t \right) \left( \rho^\minrm(\kappa) - \rho^*(\kappa) \right) \, \dd \kappa \geq 0 \, , \\
        &\int_0^{\kappa_\maxrm} \frac{\delta E}{\delta \rho} \{ \rho^\minrm \} \left( \zeta(\kappa); x, t \right) \left( \rho^*(\kappa) - \rho^\minrm(\kappa) \right) \, \dd \kappa \geq 0 \, .
    \end{align}
    Adding these two expressions together, the $V$-terms cancel and we can combine the $\mathpzc{L}$-terms, accounting for an overall negative sign, to get
    \begin{equation}
        \int_0^{\kappa_\maxrm} \mathpzc{L}[ \rho - \rho^\minrm ](\kappa) \left( \rho^*(\kappa) - \rho^\minrm(\kappa) \right) \, \dd \kappa \leq 0.
    \end{equation}
    By lemma \ref{lem:PositiveDefL}, it must be that this integral is zero and $\rho^* = \rho^\minrm$.
    This shows that the variational conditions determine the minimizer uniquely.
\end{proof}

\subsection{Solving the Variational Conditions}
\label{sec:Variations}

Instead of attempting to solve the variational conditions directly to find an explicit form for the minimizing density $\rho^\minrm$, we will assume first that we may differentiate the minimizing density with respect to $x$ and $t$ and that the sets $I_-(x,t)$, $I_0(x,t)$ and $I_+(x,t)$ change continuously as $x$ and $t$ do. Then, we can differentiate the variational conditions:
\begin{equation}
    \left\{ \begin{array}{cl}
        \mathpzc{L}[\rho^\minrm_x](\kappa) = -\im{\zeta(\kappa)} & \text{if } s \in I_0(x, t), \\
        \rho^\minrm_x(\kappa) = 0 & \text{if } s \notin I_0(x, t).
   \end{array} \right. \, ,
   \label{eq:XDiffVarConds}
\end{equation}
\begin{equation}
    \left\{ \begin{array}{cl}
        \displaystyle \mathpzc{L}[\rho^\minrm_t](\kappa) = -\im{\left( \zeta(\kappa) - \frac{1}{2\delta} \right)^2} & \text{if } s \in I_0^{\mathrm{o}}(x, t), \\
        \rho^\minrm_t(\kappa) = 0 & \text{if } s \notin I_0^{\mathrm{o}}(x, t).
   \end{array} \right. \, .
   \label{eq:TDiffVarConds}
\end{equation}
Solving these differentiated variational conditions, \eqref{eq:XDiffVarConds} and \eqref{eq:TDiffVarConds}, then integrating in $x$ and $t$ consistently (mixed derivatives are equal) and verifying the result satisfies the full variational conditions \eqref{eq:VariatConds}, then by Theorem \ref{thm:VariationCondDetMin}, we have constructed the unique minimizing density $\rho^\minrm$.
Once we have $\rho^\minrm$, then the distributional limit of the ILW semiclassical soliton ensemble can be given expressed explicitly using Theorem \ref{thm:DistLimit}.

Looking at the differentiated variational conditions \eqref{eq:XDiffVarConds} and \eqref{eq:TDiffVarConds} and the form of $\mathpzc{L}$ \eqref{eq:LContDef}, we realize these are two independent problems in logarithm potential theory for a measure with density supported on a Schwartz symmetric subset of $[\zeta_\maxrm^*, \zeta_\maxrm]_{\quadratrix{\delta}}$ and which has odd symmetry across the real axis.
On its support, we have a prescribed value for the logarithmic potential $\mathpzc{L}[ \rho^\minrm ]$.

To continue, we will make one more assumption, which is that $I_0(x, t)$ consists of a single interval: $I_0 = [0, \beta]$ where $0 < \beta \leq \kappa_\maxrm$.
This is consistent with the KdV small-dispersion limit \cite{LaxLevermoreI_1983} and, as it will turn out, it will constrain us to the simplest case possible, that of ``small time'' before the formation of the DSW (see Figures \ref{fig:SmallDispEps05T00}, \ref{fig:SmallDispEps05T03} and \ref{fig:SmallDispEps05T06}).

\subsubsection{The One Band Ansatz}
\label{sec:OneBand}

As it turns out, this problem can be solved by inspection and then verifying that our ansatz has all of the properties required to establish it as the unique solution.
This is because the logarithmic potential for the solution to the differentiated variational conditions with a prescribed symmetric interval $I$ is determined uniquely by the conditions enumerated in the following theorem.

\begin{lemma} \label{lem:LogPotentialUniqueSol}
    If $\rho'$ is H\"older continuous with power $-1/2$ and has logarithmic potential $\mathpzc{L}[\rho']$ which satisfies the differentiated variational conditions \eqref{eq:XDiffVarConds} or \eqref{eq:TDiffVarConds}, then the logarithmic potential $\mathpzc{L}[\rho']$ is uniquely identified by the properties:
    \begin{enumerate}
        \item harmonic on $\C \setminus I$ and continuous on $\C$,
        \item the value on $I$ is dictated by the differentiated variational conditions \eqref{eq:XDiffVarConds} or \eqref{eq:TDiffVarConds},
        \item odd symmetry under reflection across the real axis: $\mathpzc{L}[\rho'](\zeta) = -\mathpzc{L}[\rho'](\zeta^*)$,
        \item decay at infinity $\mathpzc{L}[\rho'](\zeta) = \bigO{\sfrac{1}{\zeta}}$ as $\zeta \to \infty$.
    \end{enumerate}
    Additionally, the density $\rho'$ can be extracted by the formula
    \begin{equation}
        \rho'(\kappa) = \im{ \zeta'(\kappa) \Big( \partial_{\zeta} \mathpzc{L}[\rho']^-(\zeta(\kappa)) - \partial_{\zeta} \mathpzc{L}[\rho']^+(\zeta(\kappa)) \Big) }
        \label{eq:RecoverChargeDensityFromPotential}
    \end{equation}
    where the $+$ and $-$ indicate the left and right boundary values to the quadratrix respectively and $\partial_\zeta = (\partial_{\re{\zeta}} - \ii \partial_{\im{\zeta}})/2$ is the holomorphic derivative in $\zeta$.
\end{lemma}
\begin{proof}
    First, it is clear from the formula for $\mathpzc{L}[\rho']$ \eqref{eq:LContDef} that it satisfies these properties.
    Now, if we subtract some other logarithmic potential with these properties, the result is a function of $\zeta \in \C$ which is harmonic away from the quadratrix interval, zero on it and decays to zero as $\zeta \to \infty$.
    By the extremum principal of harmonic functions, any maximum or minima of this difference of functions must occur on the interval, or at infinity, or the difference of functions is constant.
    In all cases, it must be identically zero.
    Thus, $\mathpzc{L}[\rho']$ is uniquely determined by the conditions. Lastly, from \cite[Theorem 1.3]{SaffTotic_1997_LogPotExtFields}, the density $\rho'$ can be recovered from the log-potential by \eqref{eq:RecoverChargeDensityFromPotential}.
\end{proof}

Inspecting the solution process of \cite{LaxLevermoreI_1983, LaxLevermoreII_1983} to the analogous problem for the KdV equation, we see that underlying their solution to the $x$-differentiated variational condition for a single symmetric band is an extension of the function which determines the derivative of the WKB phases for the Shr\"odinger equation, one that is analytic away from the band in the complex plane.
The properties of this analytic function are exactly what leads to it playing a direct role in satisfying the variational conditions due to the asymptotic scattering data being derived from the same function.
Thus, to look for the solution to the ILW differentiated variational conditions, we should try to determine the analytic continuation of the derivative of the fast phase.

Since the Lambert $W$-function was used to compute the derivative of the fast phase in Section \ref{sec:WKB}, we should use it to analytically extend off the quadratrix.
The central map involved will be $z \mapsto W(-z \ee^{-z-y})$ where we use the branch of Lambert $W$ which corresponds to the input $z \in \C$, i.e. if we define a function $n: \C \to \Z$ where $n(z)$ indicates which branch range region $z$ belongs to, then
\begin{equation}
    W(-z \ee^{-z-y}) := W_{n(z)}(-z \ee^{-z-y}) \, .
\end{equation}
Furthermore, we consider only $y \geq 0$.
First off, we realize that in each branch region $z \mapsto W(-z \ee^{-z-y})$ is the composition of three individually analytic maps: $z_0 \mapsto z_1 := -z_0 \ee^{-z_0}$, $z_1 \mapsto z_2 := z_1 \ee^{-y}$, and then $z_2 \mapsto z_3 := W_{n(z_0)}(z_2)$.
This means, we should trace the behavior of the boundary points of each branch region to see where the resulting map is analytic or discontinuous.
Follow along the steps in Figure \ref{fig:ZeroPhaseWMap} to get a visual understanding of how each of these maps acts of the three characteristic branch regions: $n=0$, $n=-1$ and $n=-2$.
\begin{enumerate}
    \item In each of these regions the map $z_0 \mapsto z_1 = -z_0 \ee^{-z_0}$ is onto the whole complex plane.
    For the $0$-region, the boundary points along the inside of the quadratrix are mapped to the two sides of the negative real axis to the left of the branch point value $-\ee^{-1}$.
    For the $\pm 1$-regions, there are two branch point values: $0$ and $-\ee^{-1}$.
    The outside boundary points of the quadratrix are also mapped to top or bottom of the negative real axis to the left of $-\ee^{-1}$ while the boundary points along the positive real axis to the right of $1$ are mapped between $-\ee^{-1}$ and $0$.
    The other continuous boundary curves, the ones bordering the $\pm 2$-regions, are mapped around the other side of the negative real axis.
    Then, for all other branch regions $|n| \geq 2$, the boundaries map to complementary sides of the negative real axis, around the single branch point value of $0$.
    
    \item For $y \geq 0$, $\ee^{-y} \leq 1$ and so $z_1 \mapsto z_2 \ee^{-y}$ is a shrinking of the plane.
    In the $0$-region, under the shrinking the mapped quadratrix inner boundary points ``slide inwards'' along the negative real axis, resulting in an interval of the quadratrix inner boundary points mapping to negative real numbers greater than $-\ee^{-1}$.
    In the $\pm1$-regions, an identical thing happens to the quadratrix outer boundary points along the same interval.
    Note that the boundary points shared with the $\pm 2$-regions are still mapped to the negative real line.
    This interval which has ``slid past'' the point $-\ee^{-1}$ is the image of a symmetric interval along the quadratrix.
    For all the other branches $|n| \geq 2$, complementary boundary points between regions map to complementary sides of the negative real line since all other branches have only the branch point $0$, which is invariant under the dilation.
    
    \item Lastly, $z_2 \mapsto z_3 = W_{n(z_0)}(z_2)$, because of the chosen branch evaluations, returns each complex plane mapped to in the first step to the negative of the original region.
    Since the scaling of step 2 never moved points across the branch cuts, the full mapping must be analytic on the interiors of each region.
    For the boundary points in the $0$-region, the points which were scaled past the branch point now map to the positive real line interior to the quadratrix.
    All of the other points which are still below the branch point map along the quadratrix, just a different point from where they started.
    A similar result is true of the quadratrix boundary points in the $\pm1$-regions, but they now map to the positive real line exterior to the quadratrix.
    For all other boundaries, they only respected the branch point value of $0$, which is invariant under the scaling in step 2, they all still map along their same boundaries, again just to different points along these boundaries.
\end{enumerate}
For any boundary points which still map to their original boundary, the full map has to be continuous at these points since taking the same boundary value in two adjacent regions must give the same value under $z_0 \mapsto z_1 = -z_0 \ee^{-z_0}$ and $z_1 \mapsto z_2 = -z_1 \ee^{-y}$.
Then, choosing the correct branch for the region, $z_2 \mapsto z_3 = W_{n(z_0)}(z_2)$ gives the same value as well since the points are still on the shared boundary between the two regions.
It is easy to show that this continuity implies analyticity using  Morera's theorem.
This means the only discontinuity in the map $z \mapsto W(-z \ee^{-z-y})$ is along a symmetric interval of the quadratrix where on one side the values lie on the positive real line interior to the negated quadratrix, and on the other side the values lie on the positive real line exterior to the quadratrix.

We note that the branch points of the map $z \mapsto W(-z \ee^{-z-y})$ are the points on the quadratrix which satisfy $-z \ee^{-z-y} = -\ee^{-1}$.
Since $-z \ee^{-z}$ is monotone increasing from $-\ee^{-1}$ to $+\infty$ on the upper and lower quadratrix, there is a monotonic, increasing relationship between the width of the branch cut in the map $z \mapsto W(-z \ee^{-z-y})$ and $y$.
In particular, for the upper branch point, it is
\begin{equation}
    \beta(y) := - \frac{1}{2 \delta} W_{-1}(-\ee^{y - 1}) \in \quadplus{\delta}  \label{eq:BetaEndPntDef}
\end{equation}
while the lower branch point is always the conjugate.

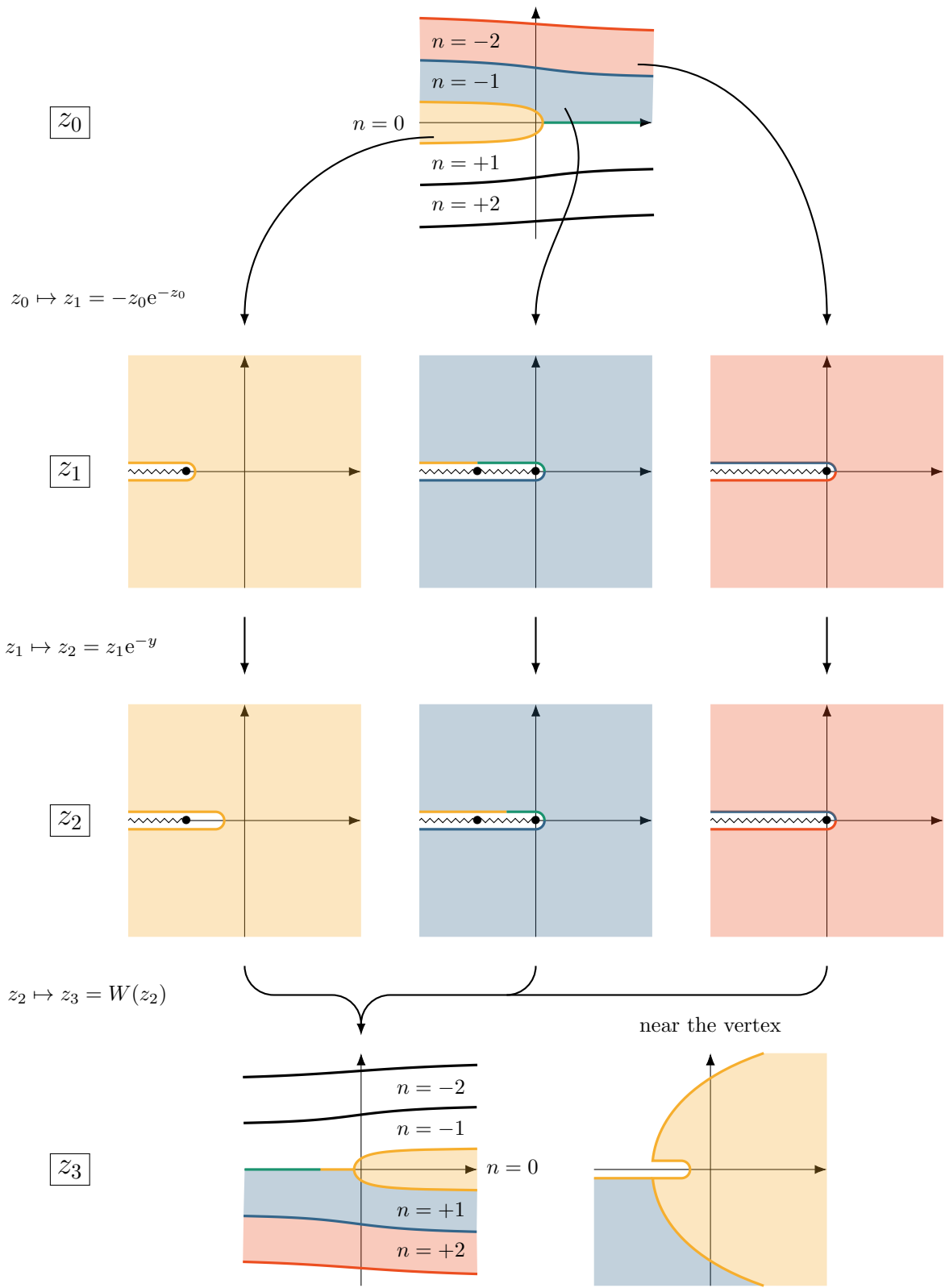
\begin{figure}
    \centering
    \begin{tikzpicture}[scale=0.96]
        \node at (-8, 0) {\Large $\boxed{z_0}$};
        \node at (-8, -6) {\Large $\boxed{z_1}$};
        \node at (-8, -12) {\Large $\boxed{z_2}$};
        \node at (-8, -18) {\Large $\boxed{z_3}$};

        \node at (-7.7, -15) {$z_2 \mapsto z_3 = W(z_2)$};
        
        \draw[-{Latex[length=2mm]}] (0, -2) -- (0, 2);
        \draw[-{Latex[length=2mm]}] (-2, 0) -- (2, 0);
        \fill[fire, opacity=0.3, variable=\y] plot[scale=0.12, domain=13.23:14.975] ({\y*cos(\y r)/sin(\y r)}, {\y}) -- plot[scale=0.12, domain=8.93:6.66] ({\y*cos(\y r)/sin(\y r)}, {\y});
        \fill[navy, opacity=0.3, variable=\y] plot[scale=0.12, domain=6.66:8.93] ({\y*cos(\y r)/sin(\y r)}, {\y}) -- plot[scale=0.12, domain=2.965:0.01] ({\y*cos(\y r)/sin(\y r)}, {\y}) -- (2, 0);
        \fill[maize, opacity=0.3, variable=\y] plot[scale=0.12, domain=0.01:2.965] ({\y*cos(\y r)/sin(\y r)}, {\y}) -- plot[scale=0.12, domain=-2.965:-0.01] ({\y*cos(\y r)/sin(\y r)}, {\y});
        \draw[scale=0.12, domain=13.23:14.975, smooth, variable=\y, fire, very thick]  plot ({\y*cos(\y r)/sin(\y r)}, {\y});
        \draw[scale=0.12, domain=8.93:6.66, smooth, variable=\y, navy, very thick]  plot ({\y*cos(\y r)/sin(\y r)}, {\y});
        \draw[seafoam, very thick] (0.12, 0) -- (1.8, 0);
        \draw[scale=0.12, domain=2.965:0.01, smooth, variable=\y, maize, very thick]  plot ({\y*cos(\y r)/sin(\y r)}, {\y});
        \draw[scale=0.12, domain=-2.965:-0.01, smooth, variable=\y, maize, very thick]  plot ({\y*cos(\y r)/sin(\y r)}, {\y});
        \draw[scale=0.12, domain=-8.93:-6.66, smooth, variable=\y, black, very thick]  plot ({\y*cos(\y r)/sin(\y r)}, {\y});
        \draw[scale=0.12, domain=-13.23:-14.975, smooth, variable=\y, black, very thick]  plot ({\y*cos(\y r)/sin(\y r)}, {\y});
        \node at (-1.2,  1.4) {$n = -2$};
        \node at (-1.2,  0.7) {$n = -1$};
        \node at (-2.7,  0.0) {$n = 0$};
        \node at (-1.2, -0.7) {$n = +1$};
        \node at (-1.2, -1.4) {$n = +2$};
        
        \node at (-7.5, -3) {$z_0 \mapsto z_1 = -z_0 \ee^{-z_0}$};
        \draw[-{Latex[length=2mm]}, thick] (-1.75, -0.25) to [out=180, in=90] (-5, -3.5);
        \draw[-{Latex[length=2mm]}, thick] (0.5, 0.25) to [out=-60, in=90] (0, -3.5);
        \draw[-{Latex[length=2mm]}, thick] (1.75, 1) to [out=0, in=90] (5, -3.5);
        
        \draw[-{Latex[length=2mm]}] (-5, -8) -- (-5, -4);
        \draw[-{Latex[length=2mm]}] (-6, -6) -- (-3, -6);
        \draw[decorate, decoration={zigzag, segment length=4, amplitude=1}] (-7, -6) -- (-6, -6);
        \fill[maize, opacity=0.3] (-7, -5.85) -- (-6, -5.85) arc (90:-90:0.15) -- (-7, -6.15) -- (-7, -8) -- (-3, -8) -- (-3, -4) -- (-7, -4);
        \draw[maize, very thick] (-7, -5.85) -- (-6, -5.85) arc (90:-90:0.15) -- (-7, -6.15);
        \node at (-6, -6) {\textbullet};
        
        \draw[-{Latex[length=2mm]}] (0, -8) -- (0, -4);
        \draw[-{Latex[length=2mm]}] (0, -6) -- (2, -6);
        \draw[decorate, decoration={zigzag, segment length=4, amplitude=1}] (-2, -6) -- (0, -6);
        \fill[navy, opacity=0.3] (-2, -5.85) -- (0, -5.85) arc (90:-90:0.15) -- (-2, -6.15) -- (-2, -8) -- (2, -8) -- (2, -4) -- (-2, -4);
        \draw[maize, very thick] (-2, -5.85) -- (-1, -5.85);
        \draw[seafoam, very thick] (-1, -5.85) -- (0, -5.85) arc (90:0:0.15);
        \draw[navy, very thick] (0.15, -6) arc (0:-90:0.15) -- (-2, -6.15);
        \node at (-1, -6) {\textbullet};
        \node at (0, -6) {\textbullet};
        
        \draw[-{Latex[length=2mm]}] (5, -8) -- (5, -4);
        \draw[-{Latex[length=2mm]}] (5, -6) -- (7, -6);
        \draw[decorate, decoration={zigzag, segment length=4, amplitude=1}] (3, -6) -- (5, -6);
        \draw[navy, very thick] (3, -5.85) -- (5, -5.85) arc (90:0:0.15);
        \draw[fire, very thick] (5.15, -6) arc (0:-90:0.15) -- (3, -6.15);
        \fill[fire, opacity=0.3] (3, -5.85) -- (5, -5.85) arc (90:-90:0.15) -- (3, -6.15) -- (3, -8) -- (7, -8) -- (7, -4) -- (3, -4);
        \node at (5, -6) {\textbullet};
        
        \node at (-7.8, -9) {$z_1 \mapsto z_2 = z_1 \ee^{-y}$};
        \draw[-{Latex[length=2mm]}, thick] (-5, -8.5) -- (-5, -9.5);
        \draw[-{Latex[length=2mm]}, thick] (0, -8.5) -- (0, -9.5);
        \draw[-{Latex[length=2mm]}, thick] (5, -8.5) -- (5, -9.5);
        
        \draw[-{Latex[length=2mm]}] (-5, -14) -- (-5, -10);
        \draw[-{Latex[length=2mm]}] (-6, -12) -- (-3, -12);
        \draw[decorate, decoration={zigzag, segment length=4, amplitude=1}] (-7, -12) -- (-6, -12);
        \fill[maize, opacity=0.3] (-7, -11.85) -- (-5.5, -11.85) arc (90:-90:0.15) -- (-7, -12.15) -- (-7, -14) -- (-3, -14) -- (-3, -10) -- (-7, -10);
        \draw[maize, very thick] (-7, -11.85) -- (-5.5, -11.85) arc (90:-90:0.15) -- (-7, -12.15);
        \node at (-6, -12) {\textbullet};
        
        \draw[-{Latex[length=2mm]}] (0, -14) -- (0, -10);
        \draw[-{Latex[length=2mm]}] (0, -12) -- (2, -12);
        \draw[decorate, decoration={zigzag, segment length=4, amplitude=1}] (-2, -12) -- (0, -12);
        \fill[navy, opacity=0.3] (-2, -11.85) -- (0, -11.85) arc (90:-90:0.15) -- (-2, -12.15) -- (-2, -14) -- (2, -14) -- (2, -10) -- (-2, -10);
        \draw[maize, very thick] (-2, -11.85) -- (-0.5, -11.85);
        \draw[seafoam, very thick] (-0.5, -11.85) -- (0, -11.85) arc (90:0:0.15);
        \draw[navy, very thick] (0.15, -12) arc (0:-90:0.15) -- (-2, -12.15);
        \node at (-1, -12) {\textbullet};
        \node at (0, -12) {\textbullet};
        
        \draw[-{Latex[length=2mm]}] (5, -14) -- (5, -10);
        \draw[-{Latex[length=2mm]}] (5, -12) -- (7, -12);
        \draw[decorate, decoration={zigzag, segment length=4, amplitude=1}] (3, -12) -- (5, -12);
        \draw[navy, very thick] (3, -11.85) -- (5, -11.85) arc (90:0:0.15);
        \draw[fire, very thick] (5.15, -12) arc (0:-90:0.15) -- (3, -12.15);
        \fill[fire, opacity=0.3] (3, -11.85) -- (5, -11.85) arc (90:-90:0.15) -- (3, -12.15) -- (3, -14) -- (7, -14) -- (7, -10) -- (3, -10);
        \node at (5, -12) {\textbullet};
        
        \draw[-{Latex[length=2mm]}, thick] (-5, -14.5) arc(180:270:0.5) -- (-3.5, -15) arc(90:0:0.5) -- (-3, -15.7);
        \draw[thick] (0, -14.5) arc(0:-90:0.5) -- (-0.5, -15) -- (-2.5, -15) arc(90:180:0.5);
        \draw[thick] (5, -14.5) arc(0:-90:0.5) -- (-0.5, -15);
       
        \begin{scope}[shift={(-3, -18)}]
            \draw[-{Latex[length=2mm]}] (0, -2) -- (0, 2);
            \draw[-{Latex[length=2mm]}] (-2, 0) -- (2, 0);
            \fill[fire, opacity=0.3, variable=\y] plot[scale=0.12, domain=13.23:14.975] ({-\y*cos(\y r)/sin(\y r)}, {-\y}) -- plot[scale=0.12, domain=8.93:6.66] ({-\y*cos(\y r)/sin(\y r)}, {-\y});
            \fill[navy, opacity=0.3, variable=\y] plot[scale=0.12, domain=6.66:8.93] ({-\y*cos(\y r)/sin(\y r)}, {-\y}) -- plot[scale=0.12, domain=2.965:0.01] ({-\y*cos(\y r)/sin(\y r)}, {-\y}) -- (-2, 0);
            \fill[maize, opacity=0.3, variable=\y] plot[scale=0.12, domain=0.01:2.965] ({-\y*cos(\y r)/sin(\y r)}, {-\y}) -- plot[scale=0.12, domain=-2.965:-0.01] ({-\y*cos(\y r)/sin(\y r)}, {-\y});
            \draw[scale=0.12, domain=13.23:14.975, smooth, variable=\y, fire, very thick]  plot ({-\y*cos(\y r)/sin(\y r)}, {-\y});
            \draw[scale=0.12, domain=8.93:6.66, smooth, variable=\y, navy, very thick]  plot ({-\y*cos(\y r)/sin(\y r)}, {-\y});
            \draw[maize, very thick] (-0.12, 0) -- (-0.7, 0);
            \draw[seafoam, very thick] (-0.7, 0) -- (-2, 0);
            \draw[scale=0.12, domain=2.965:0.01, smooth, variable=\y, maize, very thick]  plot ({-\y*cos(\y r)/sin(\y r)}, {-\y});
            \draw[scale=0.12, domain=-2.965:-0.01, smooth, variable=\y, maize, very thick]  plot ({-\y*cos(\y r)/sin(\y r)}, {-\y});
            \draw[scale=0.12, domain=-8.93:-6.66, smooth, variable=\y, black, very thick]  plot ({-\y*cos(\y r)/sin(\y r)}, {-\y});
            \draw[scale=0.12, domain=-13.23:-14.975, smooth, variable=\y, black, very thick]  plot ({-\y*cos(\y r)/sin(\y r)}, {-\y});
            \node at (1.2,  1.4) {$n = -2$};
            \node at (1.2,  0.7) {$n = -1$};
            \node at (2.6,  0.05) {$n = 0$};
            \node at (1.2, -0.7) {$n = +1$};
            \node at (1.2, -1.4) {$n = +2$};
        \end{scope}
        
        \begin{scope}[shift={(3, -18)}]
            \draw[-{Latex[length=2mm]}] (0, -2) -- (0, 2);
            \draw[-{Latex[length=2mm]}] (-2, 0) -- (2, 0);
            \fill[maize, opacity=0.3, variable=\y] plot[domain=2:0.15] ({-\y*cos(\y r)/sin(\y r)}, {\y}) -- (-0.5, 0.15) arc (90:-90:0.15) -- plot[domain=-0.15:-2] ({-\y*cos(\y r)/sin(\y r)}, {\y}) -- (2, -2) -- (2, 2);
            \fill[navy, opacity=0.3, domain=-0.15:-2, smooth, variable=\y] (-2, -0.15) -- plot ({-\y*cos(\y r)/sin(\y r)}, {\y}) -- (-2, -2);
            \draw[maize, very thick] (-2, -0.15) -- (-1, -0.15);
            \draw[maize, very thick, variable=\y] plot[domain=2:0.15] ({-\y*cos(\y r)/sin(\y r)}, {\y}) -- (-0.5, 0.15) arc (90:-90:0.15) -- plot[domain=-0.15:-2] ({-\y*cos(\y r)/sin(\y r)}, {\y});
            \node at (0, 2.5) {near the vertex};
        \end{scope}
        
    \end{tikzpicture}
    \caption{$z \mapsto W(-z \ee^{-z-y})$ demonstrated via three intermediate mappings for three characteristic branch regions. The colored regions and boundaries describe where the points in the top image are mapped to in the complex plane at each step.}
    \label{fig:ZeroPhaseWMap}
\end{figure}

Lastly, we can compute the asymptotic behavior of $z \mapsto W(-z\ee^{-z-y})$ using the asymptotics of Lambert $W$ \cite[Eq. 4.13.10]{DLMF}.
The result is 
\begin{equation}
    W(-z \ee^{-z - y}) = -z - y - \frac{y}{z} + \frac{y^2/2 - y}{z^2} + \bigO{\frac{1}{z^2}} \, ,
    \label{eq:AnalyticLambertWExp}
\end{equation}
which can also be derived heuristically by letting $w = W(-z\ee^{-z-y})$ and considering the equation $w \ee^{w} = -z \ee^{-z-y}$ while requiring $w = \bigO{z}$ as $z \to \infty$.
The properties that we've concluded directly lead us to an expression for logarithmic potentials which satisfy the properties in Lemma \ref{lem:LogPotentialUniqueSol}.

\begin{corollary} \label{cor:SolToDiffVarConds}
    For any sized interval there is a $y \leq 0$ such that the expressions
    \begin{align}
        \mathpzc{L}[\rho^\minrm_x](\zeta) &= -\im{\zeta + \frac{y}{2 \delta} + \frac{1}{2 \delta} W \left( -2 \delta \zeta \ee^{-2 \delta \zeta - y} \right)} \, , \label{eq:LogPotentXDeriv} \\
        \mathpzc{L}[\rho^\minrm_t](\zeta) &= -\im{\left( \zeta - \frac{1}{2 \delta} \right)^2 + \frac{y}{2 \delta^2} - \left( \frac{y + 1 + W \left( -2 \delta \zeta \ee^{-2 \delta \zeta - y} \right)}{2 \delta} \right)^2} \, , 
        \label{eq:LogPotentTDeriv}
    \end{align}
    solve the differentiated variational conditions and the associated densities for $\kappa \in [0, \kappa_\maxrm]$ are given by
    \begin{align}
        \rho_x^\minrm(\kappa) &= \frac{\zeta'(\kappa) E'(\zeta)}{2} r\left(\zeta(\kappa); \frac{y}{2 \delta}\right) \, , \label{eq:RhoXDensity} \\
        \rho^\minrm_t(\kappa) &= \frac{\zeta'(\kappa) E'(\zeta)}{2\delta} \re{ W_0 \left( -2 \delta \zeta \ee^{-2 \delta \zeta - y} \right) - W_{-1} \left( -2 \delta \zeta \ee^{-2 \delta \zeta - y} \right) } \Bigg|_{\zeta = \zeta(\kappa)} - \frac{y}{\delta} \rho_x^\minrm(\kappa) \, ,
        \label{eq:RhoTDensity}
    \end{align}
    where $r$ is as in \eqref{eq:WyelLawDerivIntegrand}.
\end{corollary}
\begin{proof}
    First, the $x$-derivative case.
    For fixed $y$, we see that the expression of \eqref{eq:LogPotentXDeriv} on the inside of the imaginary part as a function of $\zeta \in \C$ is
    \begin{enumerate}
        \item analytic on $\C \setminus [\beta(y)^*, \beta(y)]_{\quadratrix{\delta}}$ and has a real jump in boundary values across $[\beta(y)^*, \beta(y)]_{\quadratrix{\delta}}$, 
        \item the imaginary part on $[\beta(y)^*, \beta(y)]_{\quadratrix{\delta}}$ coincides with $-\im{\zeta}$ since the Lambert $W$ expression has real boundary values,
        \item Schwartz symmetric,
        \item and decays like $\bigO{1/\zeta}$ as $\zeta \to \infty$ as a direct result of \eqref{eq:AnalyticLambertWExp}.
    \end{enumerate}
    After accounting for the imaginary part, \eqref{eq:LogPotentXDeriv} satisfies the conditions of Lemma \ref{lem:LogPotentialUniqueSol} which uniquely determine the solution to the $x$-derivative variational conditions \eqref{eq:XDiffVarConds}.
    A direct application of \eqref{eq:RecoverChargeDensityFromPotential} to \eqref{eq:LogPotentXDeriv} yields the density \eqref{eq:RhoXDensity}.\footnote{
    Note that for $\zeta \in (\beta(y), \zeta_\maxrm]_{\quadratrix{\delta}}$, the branches used in \eqref{eq:WyelLawDerivIntegrand} do not coincide with the boundary values in the extraction formula \eqref{eq:RecoverChargeDensityFromPotential}.
    However, because $\mathpzc{L}[\rho^\minrm_x]$ \eqref{eq:LogPotentXDeriv} is harmonic here, and the real part evaluation in \eqref{eq:WyelLawDerivIntegrand} kills the manifestly imaginary difference of values due to conjugate symmetry, the formula \eqref{eq:RhoXDensity} is incidentally correct for all $\kappa \in [0, \kappa_\maxrm]$.}
    
    For the $t$-derivative \eqref{eq:LogPotentTDeriv}, all of the same properties can be shown regarding the expression inside of the imaginary part, with the one difference being that the value of \eqref{eq:LogPotentTDeriv} on $[\beta(y)^*, \beta(y)]_{\quadratrix{\delta}}$ coincides with $-\im{\left( \zeta - 1/2 \delta \right)^2}$.
    And again, \eqref{eq:RhoTDensity} is a simple application of \eqref{eq:RecoverChargeDensityFromPotential} to \eqref{eq:LogPotentTDeriv}.
\end{proof}

Due to Corollary \ref{cor:SolToDiffVarConds}, we now have formulas for the logarithmic potentials and the prospective densities which solve the differentiated variational conditions \eqref{eq:TDiffVarConds} and \eqref{eq:XDiffVarConds} in the case where there is only one band which is a continuous interval, $[0, \beta(y)]_{\quadratrix{\delta}}$.
These are in terms of the the auxiliary parameter $y \geq 0$, which controls the size of the band interval.
Since in general the band interval may change size as we change $x$ and $t$, we regard $y(x, t)$ as a dynamical parameter and suppose it has at least first derivatives.
The task now is to integrate these logarithmic potentials \eqref{eq:LogPotentXDeriv} and \eqref{eq:LogPotentXDeriv} in $x$ and $t$ in a consistent manner and show that the result solves the variational conditions.
By consistency, we mean that logarithmic potential $\mathpzc{L}[\rho]$ satisfies Clairaut's theorem regarding the cross-partial derivatives in  $x$ and $t$.
From a direct computation of the derivatives of \eqref{eq:LogPotentXDeriv} and \eqref{eq:LogPotentTDeriv}, we require that
\begin{equation}
    0 = \partial_x \mathpzc{L}[ \rho_t^\minrm ](\zeta) - \partial_t \mathpzc{L}[ \rho_x^\minrm ](\zeta) 
    = \left( \frac{y_t}{2 \delta} + 2 \frac{y_x}{2 \delta} \frac{y}{2 \delta} \right) \im{\frac{1}{1 + W \left( -2 \delta \zeta \ee^{-2 \delta \zeta - y} \right)}} \, . \label{eq:CompatCond}
\end{equation}
Since the the expression with the imaginary part in \eqref{eq:CompatCond} is not zero for all $\zeta \in \C$, in order for the logarithmic potentials to be consistently integrated in $x$ and $t$, we must require the scaled auxiliary parameter $y/2\delta$ solve invicid Burgers' equation \eqref{eq:IB}.
This aligns with the expectation that the small-dispersion limit for small time should be described by invicid Burgers' equation. 
This evidence pairs nicely with the distributional limit of the soliton ensemble from Theorem \ref{thm:DistLimit} in terms of the auxiliary parameter.
\begin{align}
    \dd-\lim_{N \to \infty} u_N^\sse(x, t) &= -\frac{4 \delta}{\pi} \partial_x \int_0^{\kappa_\maxrm} \Bigg[ \left( \im{\zeta(\kappa)} + \frac{1}{2} \mathpzc{L}[\rho^\minrm_x](\zeta(\kappa)) \right) \rho^\minrm(\kappa) \\
    &\Quad[8] + \left( V(\zeta(\kappa); x, t) + \frac{1}{2} \mathpzc{L}[\rho^\minrm](\zeta(\kappa)) \right) \rho^\minrm_x(\kappa) \Bigg] \, \dd \kappa \nonumber \\
    &= -\frac{4 \delta}{\pi} \partial_x \int_0^{\kappa_\maxrm} \Bigg[ \im{\zeta(\kappa)} \rho^\minrm(\kappa) + \frac{\delta E}{\delta \rho}\{\rho^\minrm\}(\zeta(\kappa); x, t) \rho^\minrm_x(\kappa) \Bigg] \, \dd \kappa \\
    &= -\frac{4 \delta}{\pi} \im{ \int_0^{\kappa_\maxrm} \zeta(\kappa) \ \rho^\minrm_x(\kappa) \, \dd \kappa } \, .
\end{align}
where in the first application of an $x$-derivative we have used that the resulting integrand is absolutely integrable, which will be justified \textit{a posteriori}.
We then rearranged using the symmetry of the $\mathpzc{L}$ operator and to collect the Fr\'echet derivative multiplied by $\rho^\minrm_x$.
Since the Fr\'echet derivative is zero on the band and $\rho^\minrm_x$ is zero on the rest of $[0, \kappa_\maxrm]$, that term is eliminated.
On what remains, we can use the formula for $\rho^\minrm_x$ \eqref{eq:RhoXDensity}.
We notice that, because of its derivation through \eqref{eq:RecoverChargeDensityFromPotential}, the resulting integrand is a parameterization of the difference of the boundary values of a function analytic on $\C \setminus [\beta(y)^*, \beta(y)]_{\quadratrix{\delta}}$ across its branch cut along $[0, \beta(y)]_{\quadratrix{\delta}}$.
Additionally, we can divide by two and extend the integral to the $[\beta(y)^*, 0]_{\quadratrix{\delta}}$ due to the even symmetry of the integrand.
We then employ the usual trick, writing this as a contour integral over $C$ which traces along the right and left of $[\beta(y)^*, \beta(y)]_{\quadratrix{\delta}}$ with positive orientation
\begin{equation}
    \dd-\lim_{N \to \infty} u_N^\sse(x, t) = -\frac{\delta}{\pi} \im{ \int_C \zeta E'(\zeta) \frac{W \left( -2 \delta \zeta \ee^{-2 \delta \zeta - y} \right)}{1 + W \left( -2 \delta \zeta \ee^{-2 \delta \zeta - y} \right)} \, \dd \zeta } \, .
\end{equation}
Because the function has inverse square-root singularities at the branch points, we may use Cauchy's theorem to pull the contour away from the cut.
Noticing everything besides the leading $\zeta$ factor is a total derivative, we use integration-by-parts to simplify the integrand,
\begin{equation}
    \dd-\lim_{N \to \infty} u_N^\sse(x, t) = \frac{\delta}{\pi} \im{ \int_C \frac{1}{2 \delta} W \left( -2 \delta \zeta \ee^{-2 \delta \zeta - y} \right) \, \dd \zeta } \, .
\end{equation}
Finally, blowing-up the contour to infinity, we pick up the residue at infinity (see expansion \eqref{eq:AnalyticLambertWExp}),
\begin{align}
    \dd-\lim_{N \to \infty} u_N^\sse(x, t) &= \frac{1}{2 \pi} \im{ \int_C \left( 2 \delta \zeta - y + \frac{y}{2 \delta \zeta} + \bigO{\frac{1}{\zeta^2}} \right) \, \dd \zeta } \\
    &= \frac{1}{2 \pi} \im{ 2 \pi \ii \frac{y}{2 \delta} } \\
    &= \frac{y}{2 \delta} \, .
    \label{eq:DistLimValueFromAnsatz}
\end{align}
At $t = 0$, we should expect this limit to recover the initial condition.
Combining this fact with the observation of $x$ and $t$ integrable consistency from \eqref{eq:CompatCond}, we have the ansatz
\begin{equation}
    y(x, t) := 2 \delta u^\Brm(x, t) \label{eq:IBChoiceForYParam}
\end{equation}
where $u^\Brm$ solves invicid Burgers' equation \eqref{eq:IB} with initial condition $u^\Brm(x, 0) = u_0(x)$.
Recall that invicid Burgers' equation will only have a global solution up to a critical time $t_\mathrm{c}$ given by \eqref{eq:CatTime}.

Now, we are ready to show that our ansatz, after integrating, satisfies the variational conditions.
\begin{theorem} \label{thm:ConstLogPotentChargeDensSol}
    For admissible initial condition $u_0$ (Definition \ref{def:AdmisInitCond}), the variational conditions are satisfied for $x \in \R$ and $0 \leq t \leq t_\mathrm{c}$ by the logarithmic potential for $\zeta \in \C$
    \begin{equation}
        \mathpzc{L}[\rho^\minrm(\diamond; x, t)](\zeta) = \int_x^{+\infty} \im{ \zeta + \frac{1}{2 \delta} W \left( -2 \delta \zeta \ee^{-2 \delta ( \zeta + u^\Brm(x', t) )} \right) } \, \dd x'
        \label{eq:ConsMinL}
    \end{equation}
    and energy minimizing density for $\zeta \in \quadplus{\delta}$
    \begin{align}
        \rho^\minrm(\kappa; x, t) &= \frac{\zeta'(\kappa) E'(\zeta(\kappa))}{2} \int_x^{+\infty} r\big(\zeta(\kappa); u^\Brm(x', t)\big) \, \dd x' \, .
        \label{eq:MinChargeDensity}
    \end{align}
    where $r$ is as in \eqref{eq:WyelLawDerivIntegrand}.
\end{theorem}
\begin{proof}
    We split into two cases. \\
    \textit{Case 1}: $t = 0$.
    Fix $x \in \R$ and replace $u^\Brm(x, 0) = u_0(x)$.
    The formulas we then see are just integrating the supposed $x$-derivatives $\rho_x^\minrm$ and $\mathpzc{L}[\rho^\minrm_x]$ with ``initial condition'' of zero at $+\infty$.
    First, we need to show that \eqref{eq:ConsMinL} indeed matches $\mathpzc{L}[\rho^\minrm]$ where $\rho^\minrm$ is as in \eqref{eq:MinChargeDensity}.
    Consider the natural analytic extension of $\mathpzc{L}[\rho^\minrm]$, $f : \C \setminus [\zeta^*_\maxrm, \zeta_\maxrm]_{\quadratrix{\delta}} \to \C$, defined by
    \begin{equation}
        f(\zeta) := \frac{1}{\pi} \int_0^{\kappa_\maxrm} \log\left( \frac{\zeta - \zeta(\kappa)^*}{\zeta - \zeta(\kappa)\phantom{^*}} \right) \rho^\minrm(\kappa) \, \dd \kappa
        \label{eq:PotentAnalyticExtensionDef}
    \end{equation}
    where $\rho^\minrm$ is as given in \eqref{eq:MinChargeDensity} and we take the branch cut of the log-kernel for each $\kappa \in [0, \kappa_\maxrm]$ to coincide with $[\zeta(\kappa)^*, \zeta(\kappa)]_{\quadratrix{\delta}}$.
    Using the odd symmetry of $\rho^\minrm$ \eqref{eq:MinChargeDensity}, integrating-by-parts in \eqref{eq:PotentAnalyticExtensionDef} allows us to conclude
    \begin{equation}
        f(\zeta) = \frac{1}{2 \pi \ii} \int_{\zeta_\maxrm^*}^{\zeta_\maxrm} \frac{1}{\zeta'-\zeta} \frac{\ii}{2 \delta} \int_x^{+\infty}W_{-1}\left(-2 \delta \zeta' \ee^{-2\delta(\zeta' + u_0(x'))} \right) - W_{0}\left(-2 \delta \zeta' \ee^{-2\delta(\zeta' + u_0(x'))} \right) \, \dd x' \, \dd \zeta'
        \label{eq:PotentAnalyticExtensionCauchy}
    \end{equation}
    where the $\dd \zeta'$ is complex contour integration along $[\zeta_\maxrm^*, \zeta_\maxrm]_{\quadratrix{\delta}}$ traced vertically.
    This is a Cauchy-type integral applied to a density which can be shown to be Holder continuous with power $3/2$.
    Noting that the density in \eqref{eq:PotentAnalyticExtensionCauchy} is $g^-(\zeta')-g^+(\zeta')$ for $\zeta' \in [\zeta_\maxrm^*, \zeta_\maxrm]_{\quadratrix{\delta}}$ with $g : \C \setminus [\zeta^*_\maxrm, \zeta_\maxrm]_{\quadratrix{\delta}} \to \C$ the analytic function with continuous boundary values
    \begin{equation}
        g(\zeta) := \ii \int_x^{+\infty} \zeta + \frac{1}{2 \delta} W \left( -2 \delta \zeta \ee^{-2 \delta ( \zeta + u_0(x') )} \right) \, \dd x' \, .
        \label{eq:PotentialFormualAnalyticExtens}
    \end{equation}
    Note, $g$ has these properties since we know the integrand is uniformly $\bigO{\sqrt{u_0(x')}}$ in magnitude as $x' \to \pm \infty$.
    For any $\zeta \in \C \setminus [\zeta^*_\maxrm, \zeta_\maxrm]_{\quadratrix{\delta}}$, \eqref{eq:PotentialFormualAnalyticExtens} allows us to deform the complex contour integration in \eqref{eq:PotentAnalyticExtensionCauchy} into a positively oriented loop around $[\zeta^*_\maxrm, \zeta_\maxrm]_{\quadratrix{\delta}}$ and evaluate by residues:
    \begin{equation}
        f(\zeta) = -\res_{\zeta' = \zeta} \frac{1}{\zeta'-\zeta}g(\zeta') + \res_{\zeta' \to \infty} \frac{1}{\zeta'-\zeta}g(\zeta') = -g(\zeta) \, . \label{eq:PotentialEqualsFormulaAnalyticExtens}
    \end{equation}
    The residue at infinity was zero since $g(\zeta) = \bigO{1/\zeta}$ as $\zeta \to \infty$ by \eqref{eq:AnalyticLambertWExp}.
    Taking the real part on each side \eqref{eq:PotentialEqualsFormulaAnalyticExtens}, we have shown \eqref{eq:ConsMinL} is true for $\zeta \notin [\zeta^*_\maxrm, \zeta_\maxrm]_{\quadratrix{\delta}}$.
    
    For $\zeta \in [1/2\delta, \zeta_\maxrm]_{\quadratrix{\delta}}$, we take a sequence $\{\zeta_n\}_{n \in \N}$ which approaches $\zeta$ perpendicular to $\quadratrix{\delta}$ and from the $-$ side (outside).
    It can be shown that the integrand of $\mathpzc{L}[\rho^\minrm](\zeta_n)$ as in \eqref{eq:LContDef} will eventually be a non-decreasing sequence of functions ($\rho^\minrm \geq 0$ as shown independently in \eqref{eq:ConsMinRho} below) and thus, by monotone convergence theorem, we have
    \begin{equation}
        \mathpzc{L}[\rho^\minrm](\zeta) = \lim_{n \to \infty} \mathpzc{L}[\rho^\minrm](\zeta_n) = \lim_{n \to \infty} -\re{g(\zeta_n)} = \int_x^{+\infty} \im{ \zeta + \frac{1}{2 \delta} W_{-1} \left( -2 \delta \zeta \ee^{-2 \delta ( \zeta + u_0(x') )} \right) } \, \dd x' \, ,
    \end{equation}
    which coincides with \eqref{eq:ConsMinL} in this case.
    The final case of $\zeta \in [\zeta_\maxrm^*, 1/2\delta)_{\quadratrix{\delta}}$ can be done by conjugate symmetry.

    Now, we will identify the voids, bands, and saturations of $\rho^\minrm$ \eqref{eq:MinChargeDensity} by considering the logarithmic potential formula \eqref{eq:ConsMinL}.
    As $x$ decreases from $+\infty$, the cut $[\beta(2 \delta u_0(x))^*, \beta(2 \delta u_0(x))]_{\quadratrix{\delta}}$ opens from the vertex of the quadratrix.
    Due to the Klaus-Shaw character of $u_0$, for a given $\zeta \in (1/2\delta, \zeta_\maxrm)_{\mathcal{Q}_\delta}$, the edge of the cut $\beta(2 \delta u_0(x))$ will pass this point exactly when $x = x_+(\zeta)$, so long as $x > x_\maxrm$.
    As $x$ decreases past $x_\maxrm$ (where the cut's size is maximal $\beta(2 \delta u_0(x_\maxrm)) = \zeta_\maxrm$), the size of the cut will begin to decrease and, at $x = x_-(\zeta)$, the endpoint of the cut $\beta(2 \delta u_0(x))$ again passes $\zeta$.
    The size of the cut continues to shrink to zero as $x \to -\infty$.
    For this reason, we split up the evaluation of the integral into cases and simplify using the values in these cases for $\zeta \in [1/2 \delta, \zeta_\maxrm]_{\mathcal{Q}_\delta}$
    \begin{equation}
        \mathpzc{L}[\rho^\minrm](\zeta) = \left\{ \begin{array}{ll}
            \displaystyle \int_x^{+\infty} \im{\zeta + \frac{1}{2 \delta} W_{-1} \left( -2 \delta \zeta \ee^{-2 \delta ( \zeta + u_0(x') )} \right) } \, \dd x' & \begin{aligned}
                \text{if } &\zeta \in \big( \beta(2 \delta u_0(x)), \zeta_\maxrm \big]_{\quadratrix{\delta}} \\
                &\text{and } x_+(\zeta) < x
            \end{aligned} \\[3em]
            \displaystyle \begin{aligned}
                &\int_x^{x_+(\zeta)} \im{ \zeta } \, \dd x' \\
                & + \int_{x_+(\zeta)}^{+\infty} \im{ \zeta + \frac{1}{2 \delta} W_{-1} \left( -2 \delta \zeta \ee^{-2 \delta ( \zeta + u_0(x') )} \right) } \, \dd x' 
            \end{aligned}
            & \begin{aligned}
                \text{if } &\zeta \in \big[ 1/2\delta, \beta(2 \delta u_0(x)) \big]_{\quadratrix{\delta}} \\
                &\text{that is } x_-(\zeta) \leq x \leq x_+(\zeta)
            \end{aligned} \\[4em]
            \displaystyle \begin{aligned}
                &\int_x^{x_-(\zeta)} \im{ \zeta + \frac{1}{2 \delta} W_{-1} \left( -2 \delta \zeta \ee^{-2 \delta ( \zeta + u_0(x') )} \right) } \, \dd x' \\
                & + \int_{x_-(\zeta)}^{x_+(\zeta)} \im{ \zeta } \, \dd x' \\
                & + \int_{x_+(\zeta)}^{+\infty} \im{ \zeta + \frac{1}{2 \delta} W_{-1} \left( -2 \delta \zeta \ee^{-2 \delta ( \zeta + u_0(x') )} \right) } \, \dd x' 
            \end{aligned}
            & \begin{aligned}
                \text{if } &\zeta \in \big( \beta(2 \delta u_0(x)), \zeta_\maxrm\big]_{\quadratrix{\delta}} \\
                &\text{and } x < x_-(\zeta)
            \end{aligned}
        \end{array} \right. \, .
    \end{equation}
    We can then use the definition of $\theta_+$ \eqref{eq:TailIntegralsDef} from the modified scattering data, as well as directly evaluating some of the integrals and adding $V(\zeta; x, 0) = \im{\zeta} x - \theta_+(\zeta)$ to compute the value of the Fr\'echet derivative in these three cases,
    \begin{align} 
        \frac{\delta E}{\delta \rho}\{\rho^\minrm\}(\zeta; x, 0) &= \left\{ \begin{array}{cl}
            \displaystyle -\frac{1}{2 \delta} \int_{x_+(\zeta)}^x \im{ W_{-1} \left( -2 \delta \zeta \ee^{-2 \delta ( \zeta + u_0(x') )} \right) } \, \dd x'
            & \begin{aligned}
                \text{if } &\zeta \in \big( \beta(2 \delta u_0(x)), \zeta_\maxrm \big]_{\quadratrix{\delta}} \\
                &\text{and } x_+(\zeta) < x
            \end{aligned} \\[3em]
            \displaystyle 0 
            & \begin{aligned}
                \text{if } &\zeta \in \big[ 1/2\delta, \beta(2 \delta u_0(x)) \big]_{\quadratrix{\delta}} \\
                &\text{that is } x_-(\zeta) \leq x \leq x_+(\zeta)
            \end{aligned} \\[3em]
            \displaystyle \phantom{-}\frac{1}{2 \delta} \int_x^{x_-(\zeta)} \im{ W_{-1} \left( -2 \delta \zeta \ee^{-2 \delta ( \zeta + u_0(x') )} \right) } \, \dd x'
            & \begin{aligned}
                \text{if } &\zeta \in \big( \beta(2 \delta u_0(x)), \zeta_\maxrm\big]_{\quadratrix{\delta}} \\
                &\text{and } x < x_-(\zeta)
            \end{aligned}
        \end{array} \right. \, .
        \label{eq:MinFrechetT0}
    \end{align}
    In the first case above, the imaginary part in the integrand is negative for $\zeta$ above $\beta(2 \delta u_0(x))$ on the quadratrix with $x > x_+(\zeta)$.
    This means the Fr\'echet derivative in the first case is always positive in this case.
    A similar analysis reveals that the Fr\'echet derivative in the third case is negative.
    From these facts and the definitions \eqref{eq:VoidDef}, \eqref{eq:BandDef} and \eqref{eq:SatDef}, we identify the intervals along the quadratrix in the three cases as a void, band and saturation, respectively.

    Now we look at \eqref{eq:MinChargeDensity}.
    Here, we make use of the fact that $r(\zeta; u_0(x))$ \eqref{eq:WyelLawDerivIntegrand} is only non-zero when $\zeta \in [1/2\delta, \beta(2 \delta u_0(x)))_{\quadratrix{\delta}}$, that is $x_-(\zeta) < x < x_+(\zeta)$.
    So, in the same three cases as above,
    \begin{align} 
        \rho^\minrm(\kappa) &= \left\{ \begin{array}{cl}
            \displaystyle 0
            & \begin{aligned}
                \text{if } &\kappa \in \big( \im{\beta(2 \delta u_0(x))}, \kappa_\maxrm \big] \\
                &\text{and } x_+(\zeta(\kappa)) < x
            \end{aligned} \\[3em]
            \displaystyle \frac{\zeta'(\kappa) E'(\zeta(\kappa))}{2} \int_x^{x_+(\zeta(\kappa))} r\big( \zeta(\kappa); u_0(x') \big) \, \dd x' 
            & \begin{aligned}
                \text{if } &\kappa \in \big[ 0, \im{\beta(2 \delta u_0(x))} \big] \\
                &\text{that is } x_-(\zeta(\kappa)) \leq x \leq x_+(\zeta(\kappa))
            \end{aligned} \\[3em]
            \displaystyle \displaystyle \frac{\zeta'(\kappa) E'(\zeta(\kappa))}{2} \int_{x_-(\zeta(\kappa))}^{x_+(\zeta(\kappa))} r\big( \zeta(\kappa); u_0(x') \big) \, \dd x' 
            & \begin{aligned}
                \text{if } &\kappa \in \big( \im{\beta(2 \delta u_0(x))}, \kappa_\maxrm\big] \\
                &\text{and } x < x_-(\zeta(\kappa))
            \end{aligned}
        \end{array} \right. \, .
        \label{eq:ConsMinRho}
    \end{align}
    Noting that the integrands above are both nonnegative, and identifying the third case as being exactly $\rho^\wyl$ \eqref{eq:RhoWylDef}, we have the second case takes on a value strictly between those of the first and third.
    Combining with the labeling of the cases above as a void, band or saturation, we see our constructed $\rho^\minrm$ \eqref{eq:ConsMinRho} exactly satisfies variational conditions \eqref{eq:VariatConds}.
    This completes the proof for the case of $t = 0$.
    
    \textit{Case 2}: $0 < t < t_{\mathrm{c}}$.
    Fix $t$ and consider the modified scattering data (Definition \ref{def:ModScattData}) for the potential $u^\Brm(\diamond, t)$.\footnote{
    This is a well-posed question since it is clear that if $u_0$ is an admissible initial condition, then so too will $u^\Brm(\diamond, t)$ for $0 < t < t_{\mathrm{c}}$.}
    For this we, use the fact that $u^\Brm$ satisfies invicid Burgers' equation \eqref{eq:IB} to derive the dependence of the scattering data on the parameter $t$.
    
    First, the turning points $x_\pm(\zeta)$.
    We use the characteristics \eqref{eq:IBMethOfCharSolution} and the turning points condition $E(\zeta) = u^\Brm(x_\pm(\zeta; t), t)$ where $E$ is as in \eqref{eq:TurningPointFromELevel}, yielding the simple linear evolution
    \begin{equation}
        x_\pm(\zeta; t) := x_\pm(\zeta) - 2 E(\zeta) t \, . \label{eq:TDepTurningPoints}
    \end{equation}
    
    Next, we see how the Weyl law \eqref{eq:RWeylLawDef} depends on the parameter $t$,
    \begin{equation}
        R^\wyl(\zeta; t) := \int_{x_-(\zeta; t)}^{x_+(\zeta; t)} \Bigg( W_{-1} \left( -2 \delta \zeta \ee^{-2 \delta (\zeta + u^\Brm(x', t))} \right) - W_0 \left( -2 \delta \zeta \ee^{-2 \delta (\zeta + u^\Brm(x', t))} \right) \Bigg) \, \dd x' \, . \label{eq:TDepRWeylLawDef}
    \end{equation}
    We do this by computing the $t$-derivative and exploiting that $u^\Brm$ solves invicid Burgers' equation.
    Differentiating in the limits of integration of \eqref{eq:TDepRWeylLawDef} yields two terms which cancel each other due to the turning point condition $E(\zeta) = u^\Brm(x_\pm(\zeta; t), t)$.
    So we need only consider the $t$-derivative of the integrand, for which 
    \begin{align}
        R_t(\zeta; t) &= \int_{x_-(\zeta; t)}^{x_+(\zeta; t)} \Bigg( \frac{W_{-1} \left( -2 \delta \zeta \ee^{-2 \delta (\zeta + u^\Brm(x', t))} \right)}{1 + W_{-1} \left( -2 \delta \zeta \ee^{-2 \delta (\zeta + u^\Brm(x', t))} \right)} - \frac{W_0 \left( -2 \delta \zeta \ee^{-2 \delta (\zeta + u^\Brm(x', t))} \right)}{1 + W_0 \left( -2 \delta \zeta \ee^{-2 \delta (\zeta + u^\Brm(x', t))} \right)} \Bigg) u^\Brm_t(x', t) \, \dd x' \\
        &= -2 \int_{x_-(\zeta; t)}^{x_+(\zeta; t)} r\big( \zeta; u^\Brm(x', t) \big) u^\Brm(x', t) u^\Brm_x(x', t) \, \dd x' \, ,
    \end{align}
    where we identified the formula \eqref{eq:WyelLawDerivIntegrand} and used that $u^\Brm$ solved invicid Burgers' equation \eqref{eq:IB}.
    Next, we change variables: $u = u^\Brm(x', t)$, $\dd u = u^\Brm_x(x', t) \, \dd x'$ on two separate intervals $x' \in [x_-(\zeta; t), x_\maxrm + 2u_\maxrm t]$ and $x' \in [x_\maxrm + 2u_\maxrm t, x_+(\zeta; t)]$ where $u^\Brm(\diamond, t)$ is respectively increasing and decreasing.
    The result:
    \begin{align}
        R_t(\zeta; t) &= -2 \int_{-E(\zeta)}^{u_\maxrm} r\big( \zeta; u \big) u \, \dd u - 2 \int_{u_\maxrm}^{-E(\zeta)} r\big( \zeta; u \big) u \, \dd u = 0 \, ,
    \end{align}
    This shows $R^\wyl(\zeta; t)$ is independent of $t$, and thus
    \begin{equation}
        R^\wyl(\zeta; t) = R^\wyl(\zeta) \, .
        \label{eq:TDepRWeylLawResult}
    \end{equation}
    In particular, this means that the maximal density $\rho^\wyl$ is also $t$-invariant.
    
    Lastly, we consider the $t$-dependence of the tail integrals \eqref{eq:TailIntegralsDef},
    \begin{equation}
        \theta_\pm(\zeta; t) := \im{\zeta} x_\pm(\zeta; t) + \int_{x_\pm(\zeta; t)}^{\pm\infty} \im{\zeta + \frac{1}{2 \delta} W_{-1} \left( -2 \delta \zeta \ee^{-2 \delta ( \zeta + u^\Brm(x', t) )} \right) } \, \dd x' \, . \label{eq:TDepTailIntegralsDef}
    \end{equation}
    Again, we compute the $t$-derivative of this to determine the dependence,
    \begin{align}
        \theta_{\pm, t}(\zeta; t) &= -2 \im{\zeta} E(\zeta) +2 E(\zeta) \ \im{\zeta + \frac{1}{2 \delta} W_{-1} \left( -2 \delta \zeta \ee^{-2 \delta ( \zeta + u^\Brm(x_\pm(\zeta; t), t) )} \right) } \nonumber \\
        &\qquad + \int_{x_\pm(\zeta; t)}^{\pm\infty} \im{ \frac{W_{-1} \left( -2 \delta \zeta \ee^{-2 \delta ( \zeta + u^\Brm(x', t) )} \right)}{1 + W_{-1} \left( -2 \delta \zeta \ee^{-2 \delta ( \zeta + u^\Brm(x', t) )} \right)} } u^\Brm_t(x', t) \, \dd x' \\
        &= - 2 \int_{x_\pm(\zeta; t)}^{\pm\infty} \im{ \frac{W_{-1} \left( -2 \delta \zeta \ee^{-2 \delta ( \zeta + u^\Brm(x', t) )} \right)}{1 + W_{-1} \left( -2 \delta \zeta \ee^{-2 \delta ( \zeta + u^\Brm(x', t) )} \right)} } u^\Brm(x', t) u^\Brm_x(x', t) \, \dd x' \, ,
    \end{align}
    where we again have used the turning point condition $E(\zeta) = u^\Brm(x_\pm(\zeta; t), t)$ and that $u^\Brm$ satisfies invicid Burgers' equation \eqref{eq:IB}.
    Again, we change variables $u = u^\Brm(x', t)$,$\dd u = u^\Brm_x(x', t) \dd x'$, using that $u^\Brm(\diamond, t)$ is monotonic on $(-\infty, x_-(\zeta; t)]$ and $[x_+(\zeta; t), +\infty)$.
    Then, after pulling the imaginary part evaluation outside of the integral, since what remains inside the integral is a total derivative times $u$, we can integrate by parts:
    \begin{align}
        \theta_{\pm, t}(\zeta; t) &= -2 \im{ \int^0_{E(\zeta)} \frac{W_{-1} \left( -2 \delta \zeta \ee^{-2 \delta ( \zeta + u )} \right)}{1 + W_{-1} \left( -2 \delta \zeta \ee^{-2 \delta ( \zeta + u )} \right)} u \, \dd u } \\
        &= \im{\left[ \frac{1}{\delta} W_{-1} \left( -2 \delta \zeta \ee^{-2 \delta ( \zeta + u )} \right) u \right]^0_{E(\zeta)}} - \frac{1}{\delta} \im{ \int^0_{E(\zeta)} W_{-1} \left( -2 \delta \zeta \ee^{-2 \delta ( \zeta + u )} \right) \, \dd u } \, .
    \end{align}
    The evaluations at the endpoints of the first term give zero and a real number, so after taking the imaginary part this term dies.
    Now we change variables once more to conduct the final integral with new variable $w = W_{-1}(-2\delta \zeta \ee^{-2 \delta(\zeta + u)})$ which results in $(1 + w) \, \dd w = -2 \delta w \, \dd u$.
    The contour of integration is along the negative, unit quadratrix from $w = -1$ to $w = -2 \delta \zeta$:
    \begin{align}
        \theta_{\pm, t}(\zeta; t) &= -\frac{1}{2 \delta^2} \im{ \int^{-2 \delta \zeta}_{-1} (1 + w) \, \dd w } \\
        &= -\frac{1}{2 \delta^2} \im{ \left[ \frac{(1 + w)^2}{2} \right]^{-2 \delta \zeta}_{-1} } \\
        &= -\im{ \left(\zeta - \frac{1}{2 \delta} \right)^2 } \, .
    \end{align}
    This with our initial value at $t = 0$, gives the result:
    \begin{equation}
        \theta_\pm(\zeta; t) =  \theta_\pm(\zeta) - \im{ \left(\zeta - \frac{1}{2 \delta} \right)^2 } t \, .
    \end{equation}
    
    Now that we have the modified scattering data for $u^\Brm(\diamond, t)$, we realize that, because $u^\Brm(\diamond, t)$ is itself an admissible initial condition for $0 < t < t_\mathrm{c}$, we use the result of the first case to guarantee that the constructed formulas \eqref{eq:ConsMinL} and \eqref{eq:MinChargeDensity} are the solutions to the variational conditions \eqref{eq:VariatConds} \textit{but posed for the new $t$-dependent modified scattering data and with $V(\zeta; x, 0)$}.
    As it turns out, the Fr\'echet derivative for this $t$-dependent problem is identical to the Fr\'echet derivative for the true variational conditions \eqref{eq:VariatConds} \textit{posed with the original modified scattering data and with $V(\zeta; x, t)$}:
    \begin{align}
        \frac{\delta E}{\delta \rho}\{\rho\}(\zeta; x, t) &= \im{\zeta} x + \im{ \left(\zeta - \frac{1}{2 \delta} \right)^2 } t - \theta_+(\zeta) + \mathpzc{L}[\rho](\zeta) \\
        &= \im{\zeta} x - \theta_+(\zeta; t) + \mathpzc{L}[\rho](\zeta) \, .
    \end{align}
    This coincidence means that we have actually solved the variational conditions \eqref{eq:VariatConds} for our chosen $0 < t < t_\mathrm{c}$, finishing the second case.
\end{proof}

From this we make two major conclusions.
First, from the calculation \eqref{eq:DistLimValueFromAnsatz} and our choice \eqref{eq:IBChoiceForYParam}, Theorem \ref{thm:ConstLogPotentChargeDensSol} guarantees by way of Theorem \ref{thm:DistLimit} that the distributional limit of the semiclassical soliton ensemble is indeed given by the invicid Burgers' solution with the prescribed initial condition for $0 \leq t < t_\mathrm{c}$,
\begin{equation}
    \dd-\lim_{N \to \infty} u^\sse_N(\diamond, t) = u^\Brm(x,t) \, . \label{eq:SemiSolitonEnsembConstDistLim}
\end{equation}
Second, we have the following corollary.
\begin{corollary} \label{cor:LRhoWyl}
    For all $\zeta \in \C$,
    \begin{equation}
        \mathpzc{L}[\rho^\wyl](\zeta) = \int_{-\infty}^{+\infty} \im{\zeta + \frac{1}{2\delta} W\left(-2\delta \zeta^{-2\delta(\zeta + u_0(x))} \right)} \, \dd x \, .
    \end{equation}
\end{corollary}
This follows by plugging in $t = 0$ and taking the limit $x \to -\infty$ of the minimizing potential \eqref{eq:ConsMinL}.
This is because in this limit $\rho^\minrm = \rho^\wyl$ as seen in the third case of \eqref{eq:ConsMinRho} since $\beta(2 \delta u_0(x)) \to 1/2\delta$ as $x \to -\infty$.
Additionally, by taking the limit $x \to -\infty$ in the third case of \eqref{eq:MinFrechetT0} and rearranging according to \eqref{eq:TailIntegralsDef}, we have the proof of the first statement in Theorem \ref{thm:LContAndBound}.

At this point, we wish to upgrade the distributional limit of the semiclassical soliton ensemble \eqref{eq:ILWSemiclassSolEnsemb} which we have shown by way of Theorem \ref{thm:ConstLogPotentChargeDensSol} to a strong $L^2$ limit.
For this, we need the following lemma.
\begin{lemma} \label{lem:UNL2Lim}
    \begin{equation} 
        \lim_{N \to \infty} \norm{u^\sse_N(\diamond, t)}_{L^2} = \norm{u_0}_{L^2}
        \label{eq:UNL2Lim}
    \end{equation}
    uniformly for all $t \in \R$.
\end{lemma}
\begin{proof}
    Fix $t \in \R$ and $N \in \N$.
    Since $u^\sse_N$ \eqref{eq:ILWSemiclassSolEnsemb} solves exactly the small dispersion ILW equation \eqref{eq:ILW}, we may integrate in $x$ out to positive infinity of $u^\sse_N$ plugged into the small dispersion ILW equation. 
    This integral is well-defined due to the exponential decay of $u^\sse_N$ and its derivatives as $x \to +\infty$ (see \eqref{eq:ILWSemiclassSolEnsemb} and \eqref{eq:FreeEnergyFunctionDef}),
    \begin{equation}
        0 = \int^{+\infty}_x \partial_t u^\sse_N(x', t; \epsN) \, \dd x' - \left( u_N^\sse(x, t) \right)^2 - \epsN \partial_x \mathcal{T}_{\delta \epsN} \left[ u_N^\sse(\diamond, t) \right](x) \, .
        \label{eq:FirstIntegralILW}
    \end{equation}
    Writing the first term using \eqref{eq:ILWSemiclassSolEnsemb}, we can evaluate the integral by anti-derivative.
    Then rearranging \eqref{eq:FirstIntegralILW} allows us to compute the $L^2$ norm
    \begin{equation}
        \norm{u^\sse_N(\diamond, t)}_{L^2}^2 = -\int^{+\infty}_{-\infty} \partial_t D_{\ii \delta \epsN} F_N(x, t) \, \dd x - \epsN \Big[ \mathcal{T}_{\delta \epsN} \left[ u_N^\sse(\diamond, t) \Big](x) \right]_{x \to -\infty}^{x \to +\infty} \, .
        \label{eq:L2NormFromFirstIntegralILW}
    \end{equation}
    
    Focusing on the first term above, we exchange the differentiation and finite difference evaluation, as always writing $z = x + \ii y$, and use \eqref{eq:FreeEnergyFunctionDef} to write
    \begin{equation}
        \partial_t F_N(z, t) = -\frac{4 \delta}{\pi} \frac{\displaystyle \sum_{S \subset \Z_N} Q_S \exp \left( -\frac{2}{\epsN^2 \pi} E_{N,S}(z, t) \right)}{\displaystyle \sum_{S \subset \Z_N} \exp \left( -\frac{2}{\epsN^2 \pi} E_{N,S}(z, t) \right)} \, ,
    \end{equation}
    where recall that $E_{N,S}$, $P_S$ and $Q_S$ are defined in \eqref{eq:ElecEnergy}, \eqref{eq:ElecDipole} and \eqref{eq:ElecQuadpole}.
    Now as $x \to +\infty$, almost all of the exponentials above decay due to $E_{N,S}$ for each $S \subset \Z_N$ being linear in $x$ with a non-negative coefficients $P_S$.
    The only ones which do not decay are those with $S = \emptyset$ since $P_\emptyset = Q_\emptyset = E_{N,\emptyset}(z, t) = 0$,
    \begin{equation}
        \lim_{x \to +\infty} \partial_t F_N(z, t) = \lim_{x \to +\infty} -\frac{4 \delta}{\pi} \frac{\displaystyle 0 + \sum_{\substack{S \subset \Z_N \\ S \neq \emptyset}} Q_S \exp \left( -\frac{2}{\epsN^2 \pi} E_{N,S}(z, t) \right)}{\displaystyle 1 + \sum_{\substack{S \subset \Z_N \\ S \neq \emptyset}} \exp \left( -\frac{2}{\epsN^2 \pi} E_{N,S}(z, t) \right)} = 0 \, . \label{eq:TDerivFreeEnergyPosInfLim}
    \end{equation}
    Now looking at $x \to -\infty$ we factor out the exponential from the numerator with the largest $x$ coefficient $P_S$, that is $S = \Z_N$.
    The remaining exponentials then must decay,
    \begin{equation}
        \lim_{x \to -\infty} \partial_t F_N(z, t) = \lim_{x \to -\infty} -\frac{4 \delta}{\pi} \frac{\displaystyle Q_{\Z_N} + \sum_{\substack{S \subset \Z_N \\ S \neq \Z_N}} Q_S \exp \left( \frac{2}{\epsN^2 \pi} \left( E_{N,\Z_N}(x, t) - E_{N,S}(x, t) \right) \right)}{\displaystyle 1 + \sum_{\substack{S \subset \Z_N \\ S \neq \Z_N}} \exp \left( \frac{2}{\epsN^2 \pi} \left( E_{N,\Z_N}(z, t) - E_{N,S}(z, t) \right) \right)} = \frac{4 \delta}{\pi} Q_{\Z_N} \, . \label{eq:TDerivFreeEnergyNegInfLim}
    \end{equation}
    Both of these limits are uniform for $-\delta \epsN \leq y \leq \delta \epsN$ since we know from Theorem \ref{thm:BoundsOnDet} that the denominator is non-zero.
    
    From \eqref{eq:L2NormFromFirstIntegralILW}, recall we need the integral of the finite difference of $\partial_t F_N$.
    For this, consider the contour $C(a, b)$ which is the positively oriented boundary of the rectangular strip segment $R(a, b) = \{ \ z \in \strip{\delta \epsN} \ | \ \re{z} \in [a, b] \ \}$.
    Since $\partial_t F_N(\diamond, t)$ is analytic on $\strip{\delta \epsN}$, we have that
    \begin{align}
        0 &= \int_{C(a, b)} \partial_t F_N(z, t) \, \dd z \\
        &= -2 \ii \delta \epsN \int_a^b D_{\ii \delta \epsN} \partial_t F_N(x, t) \, \dd x + \ii \int_{-\delta \epsN}^{\delta \epsN} \Big( \partial_t F_N(b + \ii y, t) - \partial_t F_N(a + \ii y, t) \Big) \, \dd y \, . \label{eq:RectangleContourIntOfTDeriv}
    \end{align}
    Solving for the $x$-integral in \eqref{eq:RectangleContourIntOfTDeriv} and taking the limits as $a \to -\infty$ and $b \to +\infty$, from the uniform convergence of the terms in the other integral by \eqref{eq:TDerivFreeEnergyPosInfLim} and \eqref{eq:TDerivFreeEnergyNegInfLim}, we have
    \begin{equation}
        \int_{-\infty}^{+\infty} D_{\ii \delta \epsN} \partial_t F_N(x, t) \, \dd x = \frac{1}{2 \delta \epsN} \int_{-\delta \epsN}^{\delta \epsN} \left( 0 + \frac{4 \delta}{\pi} Q_{\Z_N} \right) \, \dd y = \frac{4 \delta}{\pi} Q_{\Z_N} \, . \label{eq:L2NormFirstTermFinalAns}
    \end{equation}
    
    Returning to \eqref{eq:L2NormFromFirstIntegralILW}, we now focus on the second term.
    Since $u_N^\sse$ is expressed as the difference of boundary values of an analytic function on $\strip{\delta \epsilon}$ by \eqref{eq:ILWSemiclassSolEnsemb}, the convolution with the cotangent kernel in $\mathcal{T}_{\delta \epsN}$ \eqref{eq:ILWTOp} operates according to \cite[Equation 5.144]{Matsuno_1984_BiTransMeth}, i.e. turning the difference of boundary values into the sum.
    The $\sgn$-convolution can be computed explicitly \cite[Equation 5.149]{Matsuno_1984_BiTransMeth},
    \begin{align}
        \epsN \Big[ \mathcal{T}_{\delta \epsN} \left[ D_{\ii \delta \epsN} \partial_x F_N(\diamond, t) \right](x) \Big]_{x \to -\infty}^{ x \to +\infty} &= \epsN \Bigg[ \partial_x F_N(x+\ii \delta \epsN, t) + \partial_x F_N(x-\ii \delta \epsN, t) \nonumber \\
        &\qquad \quad - \frac{F_N(x+\ii \delta \epsN, t) - F_N(x-\ii \delta \epsN, t)}{\ii \delta \epsN} \Bigg]_{x \to -\infty}^{x \to +\infty} \, .
        \label{eq:IntTTerm}
    \end{align}
    Using \eqref{eq:FreeEnergyFunctionDef} and the same logic as before for $x \to +\infty$, isolating the $S = \emptyset$ term and everything else decays:
    \begin{equation}
        \lim_{x \to +\infty} F_N(z, t) = \lim_{x \to +\infty} 2 \delta \epsN^2 \log \left( 1 + \sum_{\substack{S \subset \Z_N \\ S \neq \emptyset}} \exp \left( -\frac{2}{\epsN^2 \pi} E_{N,S}(z, t) \right) \right) = 0 \, . \label{eq:FreeEnergyPosInfLim}
    \end{equation}
    As $x \to - \infty$, we again factor out the largest exponential growth term $S = \Z_N$,
    \begin{align}
        F_N(z, t) &= 2 \delta \epsN^2 \log \left( \exp \left( -\frac{2}{\epsN^2 \pi} E_{N,\Z_N}(z, t) \right) \left( 1 + \sum_{\substack{S \subset \Z_N \\ S \neq \Z_N}} \exp \left( \frac{2}{\epsN^2 \pi} \Big[ E_{N,\Z_N}(z, t) - E_{N,S}(z, t) \Big] \right) \right) \right) \\
        &= -\frac{4 \delta}{\pi}  E_{N,\Z_N}(z, t) + o(1) \label{eq:FreeEnergyNegInfAsym}
    \end{align}
    with the error term uniform for all $-\delta \epsN \leq y \leq \delta \epsN$.
    While the limit of $F_N$ does not exist as $x \to -\infty$, the difference of strip $\strip{\delta \epsN}$ boundary values in \eqref{eq:IntTTerm}, using \eqref{eq:FreeEnergyNegInfAsym} and the linearity of $E_{N,S}$ \eqref{eq:ElecEnergy}, does have a finite limit:
    \begin{align}
        \lim_{x \to -\infty} \frac{F_N(x + \ii \delta \epsN, t) - F_N(x - \ii \delta \epsN, t)}{\ii \delta \epsN} &= \lim_{x \to -\infty} - \frac{4}{\ii \pi \epsN} \Big( E_{\Z_N}(x + \ii \delta \epsN, t) - E_{\Z_N}(x - \ii \delta \epsN, t) \Big) + o(1) \\
        &= - \frac{8 \delta}{\pi} P_{\Z_N} \, . \label{eq:FreeEnergyDiffNegInfLim}
    \end{align}
    
    Now for the limits of the $x$-derivative terms in \eqref{eq:IntTTerm}.
    Since $F_N(\diamond, t)$ is analytic on $\strip{\delta \epsN}$, the derivative in $x$ is equivalent to the analytic derivative.
    Using \eqref{eq:FreeEnergyFunctionDef}, we have 
    \begin{equation}
        \partial_z F_N(z, t) = -\frac{4 \delta}{\pi} \frac{\displaystyle \sum_{S \subset \Z_N} P_S \exp \left( -\frac{2}{\epsN^2 \pi} E_S(z, t) \right)}{\displaystyle \sum_{S \subset \Z_N} \exp \left( -\frac{2}{\epsN^2 \pi} E_S(x, t) \right)} \, .
    \end{equation}
    By identical logic to computing the limits \eqref{eq:TDerivFreeEnergyPosInfLim} and \eqref{eq:TDerivFreeEnergyNegInfLim}, it is clear that
    \begin{equation}
        \lim_{x \to +\infty} \partial_z F_N(z, t) = 0 \quad \text{and} \quad \lim_{x \to -\infty} \partial_z F_N(z, t) = -\frac{4 \delta}{\pi} P_{\Z_N} \, . \label{eq:ZDerivFreeEnergyNegPosInfLims}
    \end{equation}
    uniformly for all $-\delta \epsN \leq \im{z} \leq \delta \epsN$.
    Using results \eqref{eq:FreeEnergyPosInfLim}, \eqref{eq:FreeEnergyDiffNegInfLim} and \eqref{eq:ZDerivFreeEnergyNegPosInfLims} in \eqref{eq:IntTTerm} gives the miraculously cancellation
    \begin{align}
        \epsN \Big[ \mathcal{T}_{\delta \epsN} \left[ D_{\ii \delta \epsN} \partial_x F_N(\diamond, t) \right](x) \Big]_{x \to -\infty}^{ x \to +\infty} &= \epsN \left[ 0 - 0 - \left( -\frac{4 \delta}{\pi} P_{\Z_N} - \frac{4 \delta}{\pi} P_{\Z_N} - \left( -\frac{8 \delta}{\pi} P_{\Z_N} \right) \right) \right] = 0 \, . \label{eq:L2NormSecondTermFinalAns}
    \end{align}
    Backing up to \eqref{eq:L2NormFromFirstIntegralILW}, using our calculations \eqref{eq:L2NormFirstTermFinalAns} and \eqref{eq:L2NormSecondTermFinalAns}, we finally have
    \begin{equation}
        \norm{u_N^\sse(\diamond, t)}^2_{L^2} = -\frac{4 \delta}{\pi} Q_{\Z_N} \, .
        \label{eq:UNNormSquaredLim}
    \end{equation}
    
    Now we take the limit as $N \to +\infty$.
    Evoking Theorem \ref{thm:BoundedVarWeakConv},
    \begin{equation}
        \lim_{N \to \infty} Q_{\Z_N} = \lim_{N \to \infty} \int \im{\left( \zeta(\kappa) - \frac{1}{2\delta} \right)^2} \, \dd \mu_{\Z_N}(\kappa) = \int_0^{\kappa_\maxrm} \im{\left( \zeta(\kappa) - \frac{1}{2\delta} \right)^2} \rho^\wyl(\kappa) \, \dd \kappa = \mathpzc{Q}\{\rho^\wyl\} \, .
        \label{eq:QLim}
    \end{equation}
    We can conduct this final integral by using the formula for $\rho^\wyl$ \eqref{eq:RhoWylDef} and extending the $x$-integration to $\R$ using the fact that the real-part is zero for $x$ outside of $[x_-(\zeta(\kappa)), x_+(\zeta(\kappa))]$. 
    Then, since \eqref{eq:RhoWylDef} involves functions which are absolutely integrable with no worse than square-root singularities we can exchange the order of integration.
    Lastly, we write the real-part of the difference of branch evaluations in \eqref{eq:RhoWylDef} as a difference of boundary values across the quadratrix:
    \begin{equation}
        \mathpzc{Q}\{\rho^\wyl\} = \int_{-\infty}^{+\infty} \int_0^{\kappa_\maxrm} \im{\left( \zeta(\kappa) - \frac{1}{2\delta} \right)^2} \left[ \frac{\zeta'(\kappa) E'(\zeta)}{2} \frac{W \left( -2 \delta \zeta \ee^{-2 \delta (\zeta + u_0(x))} \right)}{1 + W \left( -2 \delta \zeta \ee^{-2 \delta (\zeta + u_0(x))} \right)} \right]^{\zeta = \zeta(\kappa)^-}_{\zeta = \zeta(\kappa)^+} \, \dd \kappa \, \dd x \, ,
        \label{eq:QRhoWylFirstAns}
    \end{equation}
    where the square brackets are to indicate that the difference is take between the two evaluations and the $\pm$ in the evaluations indicate the boundary, $+$ inside and $-$ outside the quadratrix.
    Using the even symmetry of the integrand, we can extend the inner integral to $[-\kappa_\maxrm, \kappa_\maxrm]$ and pull the imaginary part outside the inner integral.
    Then, this inner integral is exactly the parametrization of a complex contour integral where the contour $C$ traces the $-$ side of $[\zeta_\maxrm^*, \zeta_\maxrm]_{\quadratrix{\delta}}$ upwards followed by the $+$ side downwards.
    Since the integrand is an analytic function for $\zeta \in \C \setminus [\zeta_\maxrm^*, \zeta_\maxrm]_{\quadratrix{\delta}}$ with no worse that square-root singularities, we can use Cauchy's theorem and take $C$ to be any positively oriented loop around $[\zeta_\maxrm^*, \zeta_\maxrm]_{\quadratrix{\delta}}$.
    \begin{align}
        \mathpzc{Q}\{\rho^\wyl\} &= \int_{-\infty}^{+\infty} \im{ \int_C \left( \zeta - \frac{1}{2\delta} \right)^2 \frac{E'(\zeta)}{4} \frac{W \left( -2 \delta \zeta \ee^{-2 \delta (\zeta + u_0(x))} \right)}{1 + W \left( -2 \delta \zeta \ee^{-2 \delta (\zeta + u_0(x))} \right)} \, \dd \zeta } \, \dd x \\
        &= -\frac{1}{4 \delta} \int_{-\infty}^{+\infty} \im{ \int_C \left( \zeta - \frac{1}{2\delta} \right) W \left( -2 \delta \zeta \ee^{-2 \delta (\zeta + u_0(x))} \right) \, \dd \zeta } \, \dd x \, ,
        \label{eq:QRhoWylSecondAns}
    \end{align}
    where we used that the Lambert $W$ term is a total derivative and integrated by parts. 
    Pushing the contour $C$ off to infinity and evaluating by residue using the Laurent series \eqref{eq:AnalyticLambertWExp}, 
    \begin{equation}
        \mathpzc{Q}\{\rho^\wyl\} = -\frac{\pi}{4 \delta} \int_{-\infty}^{+\infty} \big(u_0(x)\big)^2 \, \dd x = - \frac{\pi}{4 \delta} \norm{u_0}_{L^2} \, .
        \label{eq:QRhoWylThirdAns}
    \end{equation}
    Combining \eqref{eq:QRhoWylThirdAns} with \eqref{eq:QLim} and \eqref{eq:UNNormSquaredLim}, and taking a square-root, we have the desired \eqref{eq:UNL2Lim}.
    Lastly, since nothing after \eqref{eq:UNNormSquaredLim} has any $t$-dependence, it follows trivially the limit \eqref{eq:UNL2Lim} as $N \to \infty$ is uniform in $t$.
\end{proof}

With Lemma \ref{lem:UNL2Lim} and Theorem \ref{thm:DistLimit}, following the logic of \cite[Theorems 2.12, 4.2, 4.3, and 4.5]{LaxLevermoreI_1983, LaxLevermoreII_1983}, it is straightforward to show the $L^2$ convergence of Theorem \ref{thm:L2Limit}.

\section{Conclusion}
\label{sec:Conclusion}

The present work leaves several questions open. In Section \ref{sec:Direct}, we presented a heuristic WKB analysis of the ILW scattering equation, culminating in Conjecture \ref{conj:WeylLaw}.
A natural question is can this work be made rigorous?
In 1937, R. E. Langer devised a rigorous method from studying the semiclassical Schr\"odinger operator's spectrum \cite{Langer_1937}.
The Langer transform uses a change of variables, valid for an $\bigO{1}$ interval around a turning point, to convert the eigenvalue problem exactly into the model turning point equation, Airy's equation in his case, with the addition of a $\epsilon$-sized forcing term.
Then, using variation of parameters, exact solutions can be analyzed, leading to rigorous error bounds on eigenvalues.

If one wants to conduct something of a Langer transform on the ILW scattering equation, attempting to find a transformation that takes \eqref{eq:CannonFormScatteringEq} to the model turning point equation \eqref{eq:ModelTuringPointEquation}, appears to require we break analyticity of the wavefunction.
For example, \eqref{eq:ModelTuringPointEquation} can be recast as a homogeneous $\overline{\partial}$-problem interior to the strip with a boundary condition:
\begin{align}
    0 &= \ii \epsilon \Phi_x^+(x) + u(x) \psi^+(x) - \zeta \ee^{-2 \delta \zeta} \psi^-(x) \ \mathrm{for} \ x \in \R \, , \\
    0 &= \overline{\partial}_{z} \Phi(z) \ \mathrm{for} \ z \in \strip{\epsilon \delta} \, ,
\end{align}
where $\overline{\partial}_{z} := (\partial_x + \ii \partial_y)/2$.
Let $x_0 \in \R$ satisfy the turning point condition \eqref{eq:TurningPointFromELevel}.
We introduce $\Phi(z) = \ee^{u(x_0) z / \epsilon} w(z)$. This yields:
\begin{align}
    0 &= \epsilon w_y^+(x) + \big( u(x) - u(x_0) \big) w^+(x) - \frac{\ee^{-1}}{2 \delta} w^-(x) \ \mathrm{for} \ x \in \R \, , \\
    0 &= \partial_{\overline{z}} w(z) \ \mathrm{for} \ z \in \strip{\epsilon \delta} \, ,
\end{align}
where we assumed that $w$ is analytic at the boundaries to exchange $\ii w_x^+(x) = w_y^+(x)$.
Introducing the coordinate transform $g$ and set $X = g(x)$, $Y = g'(x) y$ and $Z = X + \ii Y$, we define $W(Z) := w(z)$ and then by implicit differentiation
\begin{align}
    0 &= \epsilon W_Y^+(X) + \frac{u(x) - u(x_0)}{g'(x)} W^+(X) - \frac{\ee^{-1}}{2 \delta g'(x)} W^-(X) \ \mathrm{for} \ X \in \R \, , \label{eq:ILWLangerBoundary} \\
    -\frac{g''(x)}{g'(x)} W_X &= \overline{\partial}_{Z} W(Z) \ \mathrm{for} \ X \in \R \ \mathrm{and} \ |Y| \leq \delta \epsilon g(x) \, .
    \label{eq:ILWLangerBulk}
\end{align}
The idea behind the Langer transform is to require 
\begin{equation}
    \frac{u(x) - u(x_0)}{g'(x)} = g(x) = X \, ,
    \label{eq:LangerCond}
\end{equation}
so that the scattering equation takes the form of the model turning point equation \eqref{eq:ModelTuringPointEquation}. 
This condition is integrated easily
\begin{equation}
    (g(x))^2 = 2 \int_{x_0}^x u(x') - u(x_0) \ \dd x' \, ,
\end{equation}
where, since $u(x) - u(x_0)$ looks linear around $x = x_0$, the square-root can be taken to yield a differentiable $g$.\footnote{Note, $g$ will only be real when $u(x)$ is increasing through $x = x_0$.
If the derivative has the other sign, we need to set \eqref{eq:LangerCond} instead equal to $-X$ and use the conjugation symmetry, as explained in Section \ref{sec:TurningPointEq}.}

Interestingly, there is no external forcing term on the left of the boundary equation \eqref{eq:ILWLangerBoundary} like there is in the Langer transformed semiclassical time-independent Schr\"odinger equation.
However, the width of the strip is now spatially dependent, $2 \delta \epsilon g'(x)$, as seen in the domain of $Y$ in the bulk equation \eqref{eq:ILWLangerBulk} as well as the final term of \eqref{eq:ILWLangerBoundary}.
By our construction, the variable $X$ now only appears as a parameter in the boundary equation \eqref{eq:ILWLangerBoundary}, and thus for each fixed $X$, the solutions of the model turning point equation, $\psi_\nu(Z; \epsilon, 2 \delta g'(x))$ for $\nu \in \{ -\infty \}\cup(2 \Z - 1/2)$ as in \eqref{eq:PsiFundamentalDef}, solve this equation exactly. 
We might attempt a variation of parameters style argument by expressing the solution with $X$-dependent pre-factors
\begin{equation}
    W(Z) = \sum_{\nu \in \{ -\infty \}\cup(2 \Z - 1/2)} c_\nu(X) \ \psi_\nu \big( Z; \epsilon, 2 \delta g'(x) \big)
\end{equation}
and attempting to find functions $c_\nu$ that solve the bulk equation \eqref{eq:ILWLangerBulk}.
To continue this analysis, more must be understood about the basis of solutions to the model turning point equation. 

The main result of this paper, Theorem \ref{thm:L2Limit}, only provides a limit up until break time.
As such, we currently have no information on the generated DSW that appears in Figure \ref{fig:SmallDispEps05T15}.
P. D. Lax and C. D. Levermore in \cite{LaxLevermoreII_1983} were able to calculate the small-dispersion limit in a weak-$L^2$ sense beyond the catastrophe time $t_\mathrm{c}$ by constructing multi-band solutions to their minimizing density problem.
They showed that KdV DSW generated via semiclassical soliton ensembles coincide with H. Flaschka, M. G. Forest and D. W. McLaughlin's algebro-geometric description of modulated wave trains of the KdV equation in terms of differentials on a slowly varying hyperelliptic surface \cite{FlaschkaForestMcLaughlin_1980}. 
In analogy, we may attempt to construct multi-band solutions to the equilibrium problem \eqref{eq:MinEnergyInfDef}.
I. M. Krichever has an algebro-geometric description of the modulated wave trains of the ILW equation \cite{Krichever_1991} in terms of differentials on what he dubs a ``$\delta$-deformed hyperelliptic surface,'' identified as a Riemann surface of genus-$n$ with a function $\lambda$ analytic everywhere except one point $P_\infty$, identified as infinity in a local coordinate $\zeta$, where $\lambda(\zeta) = \zeta \ee^{2 \delta \zeta} (1 + \bigO{1/\zeta})$.
More work needs to be done on understanding the Riemann surfaces I. M. Krichever has defined and understanding their connection with multi-phase solutions to the ILW equation.

\vspace*{0.25cm}

\noindent\textbf{Funding:} This work was partially funded through the National Science Foundation grant DMS-2204896.

\bibliographystyle{plain}
\bibliography{references.bib}

\end{document}